\documentclass{amsart}
\usepackage{latexsym,amsfonts,amsmath,amssymb}
\usepackage{hyperref}
\usepackage{graphicx}

\allowdisplaybreaks

\newtheorem{theorem}{Theorem}[section]
\newtheorem{lemma}[theorem]{Lemma}
\newtheorem{proposition}[theorem]{Proposition}
\newtheorem{corollary}[theorem]{Corollary}
\newtheorem{conjecture}[theorem]{Conjecture}
\theoremstyle{definition}
\newtheorem{definition}[theorem]{Definition}
\newtheorem{remark}[theorem]{Remark}
\theoremstyle{remark}

\newcounter{smalllist}
\newenvironment{SL}{\begin{list}{{\hss\rm(\roman{smalllist})\hss}}{\usecounter{smalllist}%
\setlength{\topsep}{0mm}\setlength{\parsep}{0mm}\setlength{\itemsep}{0mm}%
\setlength{\labelwidth}{2.2em}\setlength{\leftmargin}{2.5em}\setlength{\itemindent}{0em}%
}}{\end{list}}

\let\Re=\undefined\DeclareMathOperator*{\Re}{Re}
\let\Im=\undefined\DeclareMathOperator*{\Im}{Im}
\DeclareMathOperator*{\dist}{dist}

\DeclareMathOperator*{\wlim}{w-lim}

\newcommand{\qtq}[1]{\quad\text{#1}\quad}
\newcommand{\Fm}{{\mathcal{F}}}
\newcommand{\Em}{{E_\text{\upshape min}}}
\newcommand{\Et}{{E^V_\text{\upshape min}}}

\newcommand{\Z}{\mathbb{Z}}
\newcommand{\R}{\mathbb{R}}
\newcommand{\C}{\mathbb{C}}

\newcommand{\RR}{\mathcal{R}}

\newcommand{\eps}{\varepsilon}

\newcommand{\g}{\gamma}

\newcommand{\ld}{\lambda}

\newcommand{\Lw}{\mathcal{L}_\omega}
\newcommand{\Pw}{P_\omega}
\newcommand{\Rw}{R_\omega}
\newcommand{\rad}{\text{rad}}

\newcommand{\jb}[1]{\langle #1 \rangle}

\numberwithin{equation}{section}
\numberwithin{theorem}{section}
\numberwithin{figure}{section}
\numberwithin{table}{section}

\newcommand{\les}{\lesssim}
\newcommand{\ges}{\gtrsim}

\begin{document}
\title[The cubic-quintic NLS on $\R^3$]{Solitons and scattering for the cubic-quintic nonlinear Schr\"odinger equation on $\R^3$}

\author{Rowan Killip, Tadahiro Oh, Oana Pocovnicu, and Monica Vi\c{s}an}

\address{%
Rowan Killip\\
Department of Mathematics\\
University of California, Los Angeles\\
Los Angeles, CA 90095, USA}
\email{killip@math.ucla.edu}

\address{%
Tadahiro Oh\\
School of Mathematics\\
The University of Edinburgh,
and The Maxwell Institute for the Mathematical Sciences\\
James Clerk Maxwell Building\\
The King's Buildings\\
Mayfield Road\\
Edinburgh\\
EH9 3JZ, United Kingdom}
\email{hiro.oh@ed.ac.uk}

\address{%
Oana Pocovnicu,
School of Mathematics, Institute for Advanced Study, Einstein Dr., Princeton, NJ 08540, USA
and Department of Mathematics, Princeton University, Fine Hall, Washington Rd., Princeton, NJ 08544, USA
}
\email{opocovnicu@math.princeton.edu}

\address{%
Monica Vi\c{s}an\\
Department of Mathematics\\
University of California, Los Angeles\\
Los Angeles, CA 90095, USA}
\email{visan@math.ucla.edu}



\date{\today}

\begin{abstract}
We consider the cubic-quintic nonlinear Schr\"odinger equation:
\[i\partial_t u = -\Delta u - |u|^2u + |u|^4u.\]

In the first part of the paper, we analyze the one-parameter family of ground-state solitons associated to this equation with particular attention to the shape of the associated mass/energy curve.
Additionally, we are able to characterize the kernel of the linearized operator about such solitons and to demonstrate that they occur as optimizers for a one-parameter family of inequalities of Gagliardo--Nirenberg type.
Building on this work, in the latter part of the paper we prove that scattering holds for solutions belonging to the region $\mathcal{R}$ of the mass/energy plane where the virial is positive.  We show this region is partially bounded by solitons but also by rescalings of solitons (which are \emph{not} soliton solutions in their own right).  The discovery of rescaled solitons in this context is new and highlights an unexpected limitation
of any virial-based methodology.
\end{abstract}

\maketitle

\tableofcontents

\section{Introduction}
In this paper,  we consider the Cauchy problem for the cubic-quintic nonlinear Schr\"odinger equation (NLS) on $\R^3$:
\begin{align}\label{3-5}
\begin{cases}
i\partial_t u = - \Delta u -|u|^2u + |u|^4u, \\
u(0)=u_0\in H^1(\R^3),
\end{cases}
\end{align}
which arises as the natural Hamiltonian evolution associated to the energy
\begin{equation}\label{eq: E}
E(u):=\int_{\R^3} \tfrac{1}{2}|\nabla u|^2-\tfrac{1}{4}|u|^4+\tfrac{1}{6}|u|^6\, dx.
\end{equation}
Here $u(t,x)$ is a complex-valued function of $(t,x)\in\R\times\R^3$.

This model appears in numerous problems in physics, including field theory, nonlinear optics,
the mean-field theory of superconductivity, Langmuir waves in plasma physics, and
the motion of Bose--Einstein condensates; see, for example, \cite{Anderson,Grikurov,Desyatnikov,Ginzburg,Mihalache2000,Mihalache2002,Sulem,SSulem}.  This ubiquity is strongly connected
to the particular signs appearing in \eqref{3-5} and \eqref{eq: E}.  As we hope to convince the reader, the combination of focusing cubic nonlinearity and defocusing quintic
nonlinearity is physically very natural and leads to interesting mathematics.

For systems at low densities, nonlinear effects should be weak and hence it is natural to perform a Taylor expansion on the associated part of the energy
and keep only the lowest nonlinear term.  This produces the cubic term in \eqref{3-5}, at least in a gauge invariant setting.  Such a term with a positive coefficient (i.e.,
the defocusing case) represents an inherent repulsion of the constituents of the system and leads to global solutions that disperse.  A mathematically
rigorous proof of this assertion has been available for some time; see, for example, \cite{Cazenave}.

A cubic term with a negative coefficient (the focusing case) is physically a very natural scenario.  Atoms/molecules do experience an
attractive (van der Waals) force at moderate densities;  indeed, such attraction underlies the condensation of gases into liquids at low temperatures.
Similarly, the focusing cubic NLS is also commonly used to model the self-focusing of laser beams in certain nonlinear materials.

As might be expected, the three-dimensional focusing cubic NLS has also received considerable mathematical attention.
It was proved in the 1970's that blowup occurs for an open set of large initial data; see \cite{Glassey,VlasovPT}.  For a more
up-to-date view of this phenomenon, see \cite{DR10,MerleRaphael} and the references therein.

In the physical systems alluded to above, no true singularity occurs.  The very concentration that mathematics predicts degrades the accuracy of the model; new
physics comes into play, preventing further collapse.  From a mathematical point of view, the simplest way to incorporate such phenomenology is to introduce a defocusing
nonlinearity of the next higher power.  A deeper understanding of the underlying physics is, of course, needed to predict the appropriate coupling constant.

By making an appropriate choice of units for space, time, and the solution values, it is possible to scale away any coupling constants in front of the nonlinearities;
only the signs remain.  Unlike the case of a single power nonlinearity, this exhausts all scaling symmetries.  Indeed, the dynamics of solutions living at disparate length
scales are inherently different; the amplitude and spatial scale of a solution affects the relative strengths of the linear dispersion and each of the two nonlinearities.

We will exploit two conservation laws associated to the flow \eqref{3-5} in addition to the energy \eqref{eq:  E}, namely, the mass and momentum;
these are given by
\begin{equation}\label{eq: M,P}
M(u):=\int_{\R^3}|u|^2\,dx \qtq{and} P(u):= \int_{\R^3}  2\Im( \bar{u}\nabla u) \,dx,
\end{equation}
respectively.  Through the rescaling used to normalize the equation, information about the relative strength of the two nonlinearities is transfered to the mass and energy.
This behooves us to consider initial data for as broad a mass/energy region as we are able.  In Section~\ref{SEC:ME}, we determine precisely which mass/energy pairs
are actually possible; Figure~\ref{F:ME} summarizes our results.  Ultimately, we would like to produce a `phase diagram'
for \eqref{3-5}, indicating which dynamical behaviours are possible in which regions of the mass/energy plane.  This is the long-term dream that guides the investigations
in this paper.

There is no reason to introduce momentum as a third axis in this phase diagram.  The Galilei symmetry can be exploited to normalize the momentum to zero,
while leaving the mass unchanged and modifying the energy in the obvious manner (see the proof of Proposition~\ref{prop:  P(u)=0}).  Indeed, this transformation
simply amounts to passing to the rest frame of the centre of mass.

The physical intuition espoused earlier suggests that \eqref{3-5} should have global solutions.  This has been proved rigorously:

\begin{theorem}[Global well-posedness of cubic-quintic NLS in $\R^3$, \cite{Zhang}]\label{T:Matador}
The initial value problem \eqref{3-5} admits a unique global solution $u$ in the class $C_t^{} H^1_x$ and the solution depends continuously
\textup(in $C_t^{} H^1_x$\textup) on the initial data $u_0\in H^1(\R^3)$.  Furthermore, the solution obeys conservation of mass, energy, and momentum.
\end{theorem}

The cited paper presents a proof of uniqueness only within the subclass of solutions that also belong to $L^{10}_{t,x}$.  However, the ideas
needed to upgrade this to the `unconditional uniqueness' formulated above appear already in \cite{CKSTT}; see also \cite{KOPV2011} for an adaptation of
the argument to an equation very similar to \eqref{3-5}.

The global well-posedness of \eqref{3-5} cannot be proved by the same simple direct arguments that apply to the defocusing cubic NLS.
This stems from the fact that the quintic nonlinearity is energy-critical in three spatial dimensions.  To better explain this point and its effects
on the analysis, let us first redirect our discussion to the defocusing quintic NLS
\begin{equation}\label{5NLS}
i\partial_tu = -\Delta u + |u|^4u
\end{equation}
in $\R^3$, whose energy functional is given by
\begin{equation}\label{H5NLS}
u\mapsto \int_{\R^3} \tfrac12|\nabla u|^2 + \tfrac16|u|^6\,dx.
\end{equation}

The rescaling $u(t,x)\mapsto u^\ld (t, x) := \ld^\frac{1}{2} u(\ld^2 t, \ld x)$ is a symmetry of solutions to \eqref{5NLS}.  This scaling also preserves the energy
functional \eqref{H5NLS}, which is why \eqref{5NLS} is termed energy-critical.

For defocusing nonlinearities with power smaller than five, energy conservation prevents concentration.  This combined with simple contraction mapping arguments
yields a proof of global well-posedness.  For the energy-critical problem \eqref{5NLS}, this is not the case; indeed, the scaling symmetry shows that concentration
is perfectly consistent with energy conservation.  Note that we cannot use other NLS conservation laws to prevent concentration; energy has the highest regularity among
all known conservation laws.

The development of methods to tackle this problem of criticality constitutes one of the major breakthroughs in the study of nonlinear dispersive equations.  In the case of
NLS, it was precisely for the equation \eqref{5NLS} that this breakthrough was first made:

\begin{theorem}[GWP of the defocusing quintic NLS in $\R^3$, \cite{BO99,CKSTT}]\label{T:quintic}
Equation \eqref{5NLS} admits a unique global $C_t^{} \dot H^1_x$ solution for every initial data $u_0\in\dot H^1(\R^3)$. This
solution obeys
\begin{equation}\label{1010}
\int_\R\! \int_{\R^3} |u(t,x)|^{10}\,dx\,dt \leq C\bigl( \|u_0\|_{\dot H^1_x}\bigr).
\end{equation}
\end{theorem}

The spacetime bound \eqref{1010} provides an explicit expression that solutions disperse, something that we previously intuited must hold in
the defocusing case.  In fact, this bound leads to the following conclusion: there exist asymptotic states $u_\pm\in\dot H^1_x$ so
that
$$
\|u(t) - e^{it\Delta}u_\pm\|_{\dot H^1_x}\to 0 \qtq{as} t\to \pm\infty.
$$
This says that the solution scatters (to the linear flow) as $t\to\pm\infty$.  More formally, this is the statement of asymptotic completeness
of wave operators.  The question of existence of wave operators is easily settled by contraction mapping arguments, even in this scaling-critical case.

The proof of Theorem~\ref{T:quintic} is long and subtle.  This remains true, even after incorporating the extensive developments spawned by this breakthrough
over the intervening decade (cf. \cite{KV:Gopher}).

As will be evident from the arguments in Section~\ref{SEC:6}, any solution to \eqref{5NLS} can be embedded as a solution to \eqref{3-5} in a certain scaling regime.
Correspondingly, any reasonably quantitative proof of Theorem~\ref{T:Matador} must automatically imply Theorem~\ref{T:quintic}.  The key theme underlying the proof
of Theorem~\ref{T:Matador} in \cite{Zhang} is to argue conversely, namely, to start with Theorem~\ref{T:quintic} and treat the cubic term as a perturbation.  (This idea is further expanded upon in \cite{TaoVisanZhang}).  Moreover, it is shown that if one considers the cubic-quintic NLS with \emph{both} nonlinearities defocusing, then scattering holds (in $H^1_x$) even for large initial data $u_0\in H^1_x$.

The papers \cite{TaoVisanZhang,Zhang} also prove a scattering result for our equation \eqref{3-5}.  Specifically, they show that if the mass of the initial data is
sufficiently small, depending on the $\dot H^1_x$ norm of the initial data, then scattering holds.  To be more precise, sufficiently small mass means smaller than the
reciprocal of a tower of exponentials in the $\dot H^1_x$ norm.  This is the best one can achieve via these perturbative methods, because of the best known quantitative
bounds in \eqref{1010}.

One of the main results of this paper is a proof of scattering for initial data whose mass and energy belong to the larger region $\RR$ in the mass/energy plane:

\begin{theorem}\label{thm: main}
Let $u_0\in H^1(\R^3)$ be such that $(M(u_0),E(u_0))$ belongs to the region $\RR$ defined in Section~\ref{SEC:Virial}.
Then, the unique global solution $u \in C(\R; H^1(\R^3))$ to the cubic-quintic NLS \eqref{3-5} satisfies
\begin{equation*}
\|u\|_{L^{10}_{t,x}(\R\times\R^3)}\leq C\big(M(u_0),E(u_0)\big).
\end{equation*}
In particular, the solution $u$ scatters in $H^1(\R^3)$
both forward and backward in time.
\end{theorem}

The region $\RR$ is depicted in Figure~\ref{F:R preview}.  We will give a precise description of $\RR$ later in the introduction, once we have covered the necessary
prerequisites.  For now, we note the following consequence of Theorem~\ref{thm: main}: there is a mass threshold $M_*$ so that scattering holds for all initial data with $M(u)<M_*$, irrespective of the energy (or $\dot H^1_x$-norm).

\begin{figure}[h]
\noindent
\begin{center}
\fbox{
\setlength{\unitlength}{1mm}
\begin{picture}(95,37)(-7,-7)
\put(0,-5){\vector(0,1){32}}\put(-5,25){$E$}
\put(-5,0){\vector(1,0){87}}\put(82,-3){$M$}
\put(30,0){\line(0,-1){2}}
\put(30,-5){\hbox to 0mm{\hss$M_*$\hss}}
\qbezier(40,10)(60,5)(65,0)
\qbezier(30,15)(30,10)(40,10)
\put(15.7,7){$\RR$}
\linethickness{0.05mm}
\qbezier(03,0)(16,13)(30,27)\qbezier(06,0)(18,12)(30,24)
\qbezier(09,0)(13,4)(15.6,6.6)\qbezier(18.7,9.7)(25,16)(30,21)
\qbezier(12,0)(21,09)(30,18)\qbezier(15,0)(22,07)(30,15)\qbezier(18,0)(23,05)(30.7,12.7)\qbezier(21,0)(26,05)(32.3,11.3)
\qbezier(24,0)(29,05)(34.5,10.5)\qbezier(27,0)(32,05)(37.1,10.1)\qbezier(30,0)(35,05)(40,10)\qbezier(33,0)(38,05)(42.3,9.3)
\qbezier(36,0)(41,05)(44.7,8.7)\qbezier(39,0)(44,05)(47.1,8.1)\qbezier(42,0)(45,03)(49.4,7.4)\qbezier(45,0)(49,04)(51.7,6.7)
\qbezier(48,0)(51,03)(53.9,5.9)\qbezier(51,0)(54,03)(56.1,5.1)\qbezier(54,0)(56,02)(58.2,4.2)\qbezier(57,0)(59,02)(60.2,3.2)
\qbezier(60,0)(61,1)(62.2,2.2)\qbezier(63,0)(63.5,0.5)(63.9,0.9)
\qbezier(0,00)(15,15)(27,27)\qbezier(0,03)(12,15)(24,27)\qbezier(0,06)(10,16)(21,27)
\qbezier(0,09)(9,18)(18,27)\qbezier(0,12)(07,19)(15,27)\qbezier(0,15)(06,21)(12,27)\qbezier(0,18)(04,22)(09,27)\qbezier(0,21)(03,24)(06,27)
\qbezier(0,24)(01,25)(03,27)\qbezier( 0, 3)( 9,12)(17,20)\qbezier( 0, 6)( 7,13)(14,20)\qbezier( 0, 9)( 6,15)(11,20)
\end{picture}
}
\end{center}
\caption{Schematic diagram of the open set $\RR$ in Theorem~\ref{thm:  main}.}
\label{F:R preview}
\end{figure}
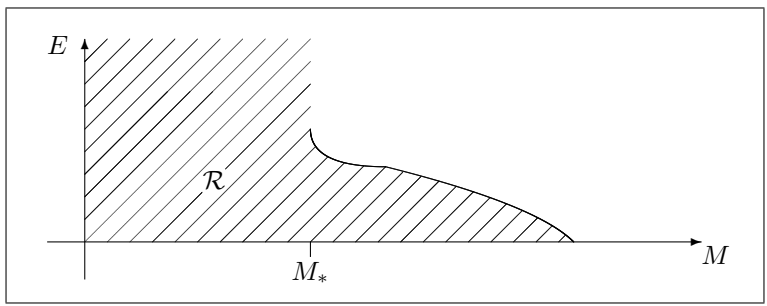

Theorem~\ref{thm: main} is not a perturbative result. We will be able to show that scattering fails for certain solutions whose mass and energy lie on the boundary of $\RR$.
Moreover, the exact value of the mass threshold $M_*$ noted above is dictated by the optimizers of a certain Gagliardo--Nirenberg-type inequality
(see Section~\ref{SEC:GNH}).  We show that these optimizers are radially symmetric solutions to a certain elliptic PDE.  This reduces matters to an ODE problem
that is readily susceptible to numerical investigation via the shooting method.  In this way, we are lead to the assertion that scattering holds for $M(u)<185.10$.

Scattering does not hold for all initial data $u_0\in H^1(\R^3)$; our equation admits solitons.  In this paper, we will use the word `soliton' to refer to
solutions to \eqref{3-5} of the form $u(t,x) = e^{i\omega t} P(x)$.  Naturally, the Galilei symmetry can be exploited to introduce (or remove) translational motion.

As well as limiting the region where scattering can occur, solitons also constitute an essential ingredient in any purported phase diagram of the dynamical behaviour
of \eqref{3-5}.  Section~\ref{SEC:Solitons} is devoted to the study of solitons for our equation.

Pohozaev identities show that solitons can only exist when $0<\omega<3/16$; see Lemma~\ref{L:Poho}.  That solitons do indeed exist for all such frequencies $\omega$
is then deduced from the general results of \cite{Berestycki}.  Note that the soliton profile is necessarily different for differing $\omega$; recall that in the case
of a single focusing nonlinearity there is a scaling symmetry, which means that solitons have the same profile for all $\omega\in(0,\infty)$.

The paper \cite{Berestycki} actually shows that \eqref{3-5} admits infinitely many radially symmetric solitons for each $0<\omega<3/16$.  We focus our attention on what
are known as ground state solitons, that is, solitons for which $P(x)$ is non-negative.  As discussed in Section~\ref{SEC:Solitons}, these exist for all
$0<\omega<3/16$; indeed, there is a unique ground state soliton $P_\omega$ for each such $\omega$.  We curtail our investigation in this particular way because the ground state
solitons are the smallest solitons in a certain mass/energy sense; see Theorem~\ref{T:solitons}(i).  In particular, it is reasonable to predict that it is
these solitons which mark the boundary of the region in the mass/energy plane where only scattering occurs.  Note that pure scattering behaviour will not
hold in any region where both mass and energy exceed the mass and energy of a soliton.  This is easily seen by constructing solutions containing a soliton and a radiation term.

As part of our analysis of the ground state solitons $P_\omega$, we are able to determine the kernel of the linearized operator about $\Pw$:
$$
\Lw : u \mapsto -\Delta u + 5 \Pw^4 u - 3 \Pw^2 u + \omega u.
$$
Specifically, we prove that it is spanned by the components of $\nabla P_\omega$.  This spectral condition is an invaluable stepping stone for any
subsequent rigourous analysis of the stability or asymptotic stability of soliton solutions, at least by the usual methods.  In particular, our verification
of the spectral condition allows one to apply the arguments in \cite{Shatah} to see that solitons are stable wherever the mass/energy curve is concave and unstable where it is convex.

Figures~\ref{F:Pw}, \ref{F:numerics1} and~\ref{F:numerics2} depict the mass/energy curve for the solitons $P_\omega$, based on numerics.  We note
that there is an upper branch of solitons where the curve is convex (signaling instability) and a lower branch where it is concave (indicating stability).
Note that from our result on the linearized operator, we know that $\omega\mapsto P_\omega$ is a real-analytic $H^1_x$-valued function; the explanation for
the cusp is that $\partial_\omega M(\Pw)$ and $\partial_\omega E(\Pw)$ both vanish at the same point.  Indeed, by \eqref{E:dEdM}, both vanish to the same order.

Our numerics show a single cusp, corresponding to the vanishing of $\partial_\omega M(\Pw)$ at a single point $\omega_*$, which in turn is the global minimum of
$\omega\mapsto M(\Pw)$ and the global maximum of $\omega\mapsto E(\Pw)$.  Unfortunately, we have not been able to prove this; it constitutes our Conjecture~\ref{Conj:mass}.

The solitons $P_{\omega}$ have been the subject of several numerical studies in the physics literature \cite{Desyatnikov,
Mihalache2000, Mihalache2002}, in their role as ``light bullets" or ``3D spatiotemporal optical solitons".
These numerics confirm our independent investigations and should be consulted by readers interested in plots of mass or energy against the frequency parameter $\omega$.

The behavior of unstable (upper branch) soliton solutions to \eqref{3-5} has been investigated numerically in lower dimensions \cite{Grikurov,Sulem}.  These authors observe that solutions
beginning near the unstable branch of solitons approach a soliton on the lower branch, shedding their excess mass/energy in the form of radiation.

For NLS with a single focusing nonlinearity, ground state solitons can be characterized as optimizers of certain Gagliardo--Nirenberg inequalities.  This has played an important
role in many investigations, beginning with the seminal work \cite{Weinstein}.  Section~\ref{SEC:GNH} is devoted to the discussion of an analogue for our problem.
Because there is a one-parameter family of ground state soliton profiles (unlike the scale-invariant case of a single power nonlinearity), we will need a one-parameter family of inequalities.  Our candidates are the following:
\begin{align}\label{E:I:GNH}
\|u\|_{L^4(\R^3)}^4 \lesssim \|u\|_{L^2(\R^3)}\|u\|_{L^6(\R^3)}^{\frac{3\alpha}{1+\alpha}} \|\nabla u\|_{L^2(\R^3)}^{\frac{3}{1+\alpha}}
\end{align}
where $0<\alpha<\infty$.  The veracity of \eqref{E:I:GNH} is easily deduced by interpolating between the classical Gagliardo--Nirenberg inequality
(which corresponds to $\alpha =0$) and the H\"older inequality (which corresponds to $\alpha=\infty$).

In Proposition~\ref{prop:GN}, we show that optimizers of \eqref{E:I:GNH} are ground state solitons, up to scaling and translation.  Moreover, optimizing solitons
$\Pw$ have the property that $\beta(\omega) = \alpha$, where $\beta(\omega)$ is defined via
$$
 \int_{\R^3} |\Pw(x)|^6\,dx = \beta(\omega) \int_{\R^3} |\nabla\Pw(x)|^2\,dx.
$$
In particular, no soliton occurs as the optimizer in \eqref{E:I:GNH} for more than one value of~$\alpha$.  However, we must acknowledge one short-comming of our results
in this direction.  Numerics show that $\omega\mapsto\beta(\omega)$ is strictly increasing, which would imply that each soliton occurs as the unique optimizer for the
corresponding value of $\alpha$.  We have been unable to prove this; see Conjecture~\ref{Conj:beta}.

As mentioned earlier, Section~\ref{SEC:ME} is devoted to determining all feasible mass/energy pairs; see Figure~\ref{F:ME}.  A transition occurs at mass $M(Q_1)$, where
$Q_1$ is a ground state soliton optimizing \eqref{E:I:GNH} in the case $\alpha=1$.  For masses below $M(Q_1)$, the energy is necessarily positive; the infimal mass
is zero, but this is not achieved.
At the mass $M(Q_1)$, zero energy is still the infimal energy, but is now achieved; moreover, it is achieved only by certain solitons.  For masses strictly greater
than $M(Q_1)$, the infimal energy is now negative and is achieved precisely by some soliton.  Our region $\RR$ is wholly contained in the mass strip where $M<M(Q_1)$.
Indeed, in the region $M>M(Q_1)$ it is impossible, using only mass and energy variables, to segregate solutions that scatter from those that converge to a soliton plus radiation.

Beginning with Section~\ref{SEC:Virial} we focus more tightly toward the proof of Theorem~\ref{thm: main}.  Naturally, to obtain a non-perturbative proof of scattering,
one needs some intrinsically nonlinear information about the equation.  This is true even in the defocusing case, where this role is invariably fulfilled by (traditional
or interaction) Morawetz identities.  For focusing equations, sharp thresholds are typically determined via the virial identity (suitably truncated); see, for example, \cite{AkahoriNawa,DHR08,KenigMerle,Berbec}.  One exception is the work \cite{Dodson} on the mass-critical NLS, which uses a hybrid of the virial and interaction Morawetz identities.  Dodson-style variants are not advantageous for our problem.  We will use the virial identity.

The virial identity stems from the behaviour of a system under scaling; however, it is not essential that the system has a scaling symmetry.
As is easily verified, the operator
$$
\mathcal{A} = \tfrac{1}{i} \bigl( x\cdot\nabla + \nabla \cdot x \bigr)
$$
is the generator of unitary dilations on $L^2(\R^3)$.  An elementary computation shows that if $u$ is a solution to \eqref{3-5}, then
\begin{equation*}
\frac{d\ }{dt} \langle u(t),\, \mathcal{A} u(t)\rangle = 4V(u(t)) \ \ \ \text{where}\ \ \ V(f) := \int_{\R^3} |\nabla f(x)|^2 + |f(x)|^6 - \tfrac34  |f(x)|^4 \,dx.
\end{equation*}
This is the virial identity for our equation; correspondingly, we will refer to the functional $V:H^1(\R^3)\to\R$ as the virial.

For soliton solutions $u(t,x)=e^{i\omega t} P(x)$, we have $\langle u(t),\, \mathcal{A} u(t)\rangle=\langle P,\, \mathcal{A} P\rangle$ and correspondingly,
$V(P)=0$.  On the other hand, for solutions $v(t,x)=e^{it\Delta}v_0$ of the linear Schr\"odinger equation, we have
$$
\frac{d\ }{dt} \langle v(t),\, \mathcal{A} v(t)\rangle = 4\int_{\R^3} |\nabla v_0|^2\,dx \qtq{and}  \lim_{t\to\pm\infty} V(v(t)) = \int_{\R^3} |\nabla v_0|^2\,dx.
$$
In particular, the virial of linear solutions is positive for large times.

The way that the virial identity is employed in the proof of scattering is rather subtle.  One of the key properties of the region $\RR$ is that
\begin{equation}\label{E:V>0 on R}
\bigl(M(u),E(u)\bigr) \in \RR \ \ \ \implies\ \ \   V(u(t)) > 0 \ \text{ for all $t\in\R$}
\end{equation}
for any solution $u$ to \eqref{3-5}.  Nevertheless, it is \emph{not} true that solutions whose virial is positive for all times must
necessarily scatter.  Imagine, for example, a solution which consists of a soliton together with radiation.  By suitable construction, one may decouple the contributions
of the two parts of the solution to the virial; indeed, decoupling will be automatic in the $t\to\pm\infty$ limits.  As solitons have zero virial and radiation has positive
virial, such a solution would have positive virial for all times, but definitely does not scatter.

Our example highlights a central problem that arises when applying conservation laws and similar identities: one only observes the aggregate of all the parts of the solution.
Unwanted behaviour of one part of the solution cannot be ruled out if it can be compensated for by unreasonable behaviour of another part of the
solution.  While sometimes space and/or frequency localization techniques can be applied, this is untenable for generic large data due to the resulting combinatorial
complexity.

In \cite{BO99}, Bourgain introduced the induction on energy technique to overcome an obstruction of precisely this type.  His solution to the combinatorial morass is
to inductively exclude unreasonable behaviour at ever increasing energies. (The energy is a coercive conserved quantity for the problem treated in \cite{BO99}.)

The guiding principle, advanced significantly in \cite{CKSTT}, is the following: Assume all solutions with energy $E\leq E_0$ have been shown to scatter;
this is the inductive hypothesis.  We wish to show that this remains true for all solutions $u$ with $E(u)\leq E_0 + \eta$, where $\eta$ is a very small parameter.
Suppose that at some time, such a solution $u$ can be written as two well-separated parts, each with energy at least $\eta$.   Then, we may apply the inductive hypothesis
to approximate $u$ by the sum of two solutions, each of which scatters.  That the parts are well-separated is essential.  Our equation is nonlinear; it is only in such a
regime that one may obtain an approximate solution by summing two solutions. For the models studied in \cite{BO99,CKSTT} and in this paper, two solutions are well-separated if they live in different spatial locations or at differing length scales.

As consequence of the preceding, we see that the inductive step is reduced to proving scattering for solutions $u$ with $E(u)\leq E_0 + \eta$ that consist of a single
piece (up to errors of size $\eta$), with a clearly-defined location in space and spatial/frequency scale.  This substantially enhances the efficacy of conservation
laws and monotonicity formulae.

More recent applications of the induction on energy technique employ a contradiction argument, closer in spirit to the well-ordering principle.
Some of these ideas are already hinted at by the language used in \cite{CKSTT}, for example, the phrase ``minimal energy blowup solution''.
However, the true mathematical realization of this belongs to Keraani \cite{keraani-l2} and to Kenig--Merle \cite{KenigMerle}.  This is the variant we will describe
below.  It leads to proofs that are more modular and simpler to understand; it has also fueled an explosion of applications of the underlying paradigm.

Next we will explain the construction of a minimal blowup solution, which captures the essence of the induction on energy argument.  After that, we will return to our discussion of Section~\ref{SEC:Virial} and
the role of the virial identity.

Until now, the size of solutions has been determined by two variables: mass and energy.  For a workable notion of
minimality, we need to combine them into a single quantity.  To this end, we introduce a continuous map $D: \RR \to [0,\infty)$ in Subsection~\ref{SS: exhaustion of R}.
We will also regard $D$ as a function of solutions via $D(u)=D(M(u),E(u))$.  For expositional clarity, we will arrange that $D(u)=0$ if and only if $u\equiv 0$ and
define $D(u)=\infty$ when $(M(u),E(u))\in\RR^c$.  Proposition~\ref{P:all about D} describes further properties of this function.

Given $0<D<\infty$, we define
\begin{align}\label{QL}
L(D):=\sup \big\{\|u\|_{L^{10}_{t,x}(\R\times\R^3)}: \text{ $u$ solves \eqref{3-5} and $D(u)\leq D$}\big\},
\end{align}
In this way, Theorem~\ref{thm: main} becomes the statement that $L(D)<\infty$ for all $0<D<\infty$.

Suppose now that Theorem~\ref{thm:  main} were to fail and let $D_c$ be the supremum of all values of $D$ for which $L(D)$ is finite.  Failure of Theorem~\ref{thm:  main}
is the assertion that $D_c<\infty$.  Note that $L(0)=0$, so $D_c\geq0$.
In fact, Proposition~\ref{thm: Small data scattering} shows that $D_c>0$.  This proposition is a refinement of the small-data theory developed in \cite{TaoVisanZhang}.
The positivity of $D_c$ corresponds to the `base step' of the induction.

Under our contradiction hypothesis, there must be a sequence of solutions $\{u_n\}$ so that
$$
D(u_n)\to D_c \qtq{and} \int_\R \int_{\R^3} |u_n(t,x)|^{10} \,dx\,dt \to \infty.
$$
If we could conclude that a subsequence of $\{u_n\}$ converges in $H^1_x$, perturbation theory (see Proposition~\ref{Perturbation of {3-5}}) would guarantee
that the limit $u_\infty$ is a minimal blowup solution. It would be minimal because $D(u_\infty)=D_c$ and a blowup solution because
$\iint |u_\infty(t,x)|^{10}\,dx\,dt = \infty$.

The assertion that bounded sequences of solutions converge (subsequentially) to a solution is a narrow form of the well-known Palais--Smale condition in the
calculus of variations.  (Our sequence $\{u_n\}$ is bounded in $L^\infty_t H^1_x$ because $\{D(u_n)\}$ is bounded.)  The usual Palais--Smale condition is immediately broken
by the presence of non-compact symmetries.  For our problem, the pertinent symmetries are spatial and temporal translations, as well as a vestigial/broken scaling symmetry.
The first step in handling this issue is to prove an appropriate concentration-compactness principle for the \emph{linear} equation.
This is the task of proving a linear profile decomposition; it is discharged in Section~\ref{SEC:5}.

Our concentration-compactness principle has strong similarities to earlier work on the energy-critical case \cite{bahouri-gerard,keraani-h1}; however, our proofs follow
the slightly different path laid out in \cite{ClayNotes,Visan:Oberwolfach}.  For our problem, we must work in $H^1_x$ rather than the homogeneous space $\dot H^1_x$ and
prove decoupling of $L^4_x$-norms as well as $L^6_x$-norms.  The techniques needed to adapt the argument have appeared before in the context of other problems with broken
symmetries, \cite{IonPaus,KKSV:KdV,KSV:2DKG,KVZ:Ob}.  Nevertheless, as a service to the reader we provide full details.

The linear profile decomposition, Theorem~\ref{thm:profile decomposition}, decomposes a subsequence of initial data $u_n(t=0)$ into linear combinations of asymptotically
orthogonal (as $n\to\infty$) bubbles of concentration.  Each bubble $\phi^j$ appears at an $n$-dependent location in space-time $(t_n^j,x_n^j)$ at some characteristic
length scale $\lambda_n^j$.  As noted earlier, for solutions at very small length scales, the quintic nonlinearity is dominant --- they behave as solutions to
the quintic NLS \eqref{5NLS}.  The role of Section~\ref{SEC:6} is to treat bubbles living at such small length scales by exploiting Theorem~\ref{T:quintic}.

In Section~\ref{SEC:7} we combine the results of Sections~\ref{SEC:5} and~\ref{SEC:6} to prove that a version of the Palais--Smale condition holds for our problem
(see Proposition~\ref{prop: Palais-Smale}) and then that failure of Theorem~\ref{thm: main} would imply the existence of a minimal blowup solution (see 
Theorem~\ref{prop: existence min blowup sol}).
It is here that we perform the `inductive step': the sequence of initial data cannot split into more than one bubble of concentration, because
this would violate the minimality of $D_c$.  In the language of \cite{Lions}, \emph{dichotomy} cannot occur.   This argument requires
that $D(m,e)$ strictly decreases if either the mass or the energy is decreased (cf. Proposition~\ref{P:all about D}(v) and~(vi)).

Because they are exact symmetries of our equation, space and time translations can be employed directly to tackle these sources of non-compactness.  However, the problem of scaling seems to have simply
disappeared in the statements of Proposition~\ref{prop: Palais-Smale} and Theorem~\ref{prop: existence min blowup sol}.
Let us explain.  Both nonlinearities are mass supercritical; correspondingly, solutions with bounded mass but living at very large length scales behave essentially linearly.
By the Strichartz inequality, linear solutions obey the spacetime bound \eqref{1010}.
Thus solutions living at very large length scales cannot arise as minimal blowup solutions.  On the other hand,
by the analysis of Section~\ref{SEC:6}, solutions living at very small length scales cannot blow up either  and so cannot occur as minimal blowup solutions.  Naturally, if the sequence $\{u_n\}$ of solutions all
live at intermediate length scales, there is no need to rescale in order to achieve subsequential convergence.

Theorem~\ref{prop: existence min blowup sol} gives more than just the existence of a minimal blowup solution (under the assumption that Theorem~\ref{thm: main} fails).
It shows that any such minimal counter-example $u$ consists of a single well-localized bump.  More precisely, it shows that such a solution
$u$ is almost periodic modulo translations, that is, there is an $\R^3$-valued function $x(t)$ of time so that
$$
\bigl\{u(t,x-x(t)) : t\in \R \bigr\} \quad \text{is precompact in $H^1(\R^3)$.}
$$
In fact, this very strong property of such minimal blowup solutions is a rather trivial consequence of our strong Palais--Smale condition.

It is now that we finally see the true power of the induction on energy paradigm.  It tells us where to center our virial identity
and at what radius we can safely truncate; the latter is dictated by compactness.   There is one additional subtlety, however: we need to control how much $x(t)$ moves.
In Proposition~\ref{prop: P(u)=0} we observe that minimal blowup solutions have zero momentum.  Otherwise, one could apply a Galilei boost, which
preserves the mass and the blowup property, but reduces the energy.  In Proposition~\ref{prop: control x(t)} we use this to constrain the motion of $x(t)$; specifically, we show that $x(t)=o(t)$ as $|t|\to\infty$.  This is sufficient to exploit a truncated version of the virial identity to show that \eqref{E:V>0 on R} is inconsistent with the region $\RR$ admitting an almost periodic solution.  This is done in Section~\ref{SEC:8}.  As our previous arguments show that failure of Theorem~\ref{thm: main}
guarantees the existence of such a solution, this completes the proof of Theorem~\ref{thm: main}.

By realizing the arguments laid out above, we will be able to show that scattering occurs on any region $\RR$ of the mass/energy plane with the following property:
\begin{equation}\label{E:Rcond}
\textit{\small If $(m,e) \in \RR$ and $u\in H^1(\R)$ obeys $M(u)\leq m$ and $E(u)\leq e$, then $V(u)>0$.}
\end{equation}
Note that non-vanishing of the virial on such a larger set is needed to perform the induction argument we described.

In Section~\ref{SEC:Virial} we determine the largest $\RR$ region obeying \eqref{E:Rcond}; this is the region appearing in Theorem~\ref{thm: main}.  As we will see, one may
determine this region by finding for each mass $m>0$ the least energy $\Et(m)$ at which the virial vanishes.  This is formalized in Definition~\ref{D:RR defn}.

From prior investigations, it is reasonable to imagine that the boundary of the region $\RR$ is marked by solitons; that is, the functions that achieve minimal energy among those
with fixed mass and zero virial are solitons.  We will prove that this is \emph{not} the case for our model.  We find this a startling new observation.  It places a formidable limitation
on existing technology and raises curious questions for further investigation; most notably, what is the true extent of the mass/energy region where only scattering holds?

We are able to prove that for some values of the mass, functions with zero virial that achieve energy $\Et(m)$ are precisely solitons, but for other values of the mass such functions
are non-trivial rescalings of soliton profiles.  These rescalings are \emph{not} solitons in their own right!  Such non-soliton virial obstructions do not appear in earlier work.  See Theorem~\ref{T:Rbdry} for further information.  Theorem~\ref{T:MMR} gives general information on the shape of the curve $m\mapsto\Et(m)$.

Let us now discuss the relation to the works \cite{Akahori,Akahori',Miao}, which considered the case where the highest nonlinearity is energy-critical and focusing.  In that setting, the scattering threshold is dictated by the radial soliton $W$ associated to the purely energy-critical problem, which is unstable to finite-time blowup.  In particular, the lower order nonlinearity does not alter the threshold determined earlier in \cite{KenigMerle,Berbec}.  The authors of the papers \cite{Akahori,Akahori',Miao} consider variational problems based on minimizing a free energy of the form
$$
F_\lambda(u) = E(u) + \lambda M(u)
$$
subject to vanishing of the virial, expanding on the methodology introduced by Payne and Sattinger \cite{PayneSattinger}.  (In \cite{Miao}, $\lambda=0$.)
It is not difficult to see from the results of this paper that our problem has the following properties: (a) Minimizers may exist (depending on $\lambda$), but no soliton ever occurs as such a minimizer and (b) The region $\RR$ cannot be exhausted by
sub-level sets of such functions.  More specifically, we draw the reader's attention to the concavity proved in Theorem~\ref{T:feasible} and the shape of the curve of rescaled solitons
shown in Figure~\ref{F:numerics1}.  (The facts about this curve proved in Lemma~\ref{L:Rw} are sufficient to verify (a) and (b), but the convexity shown by the
numerics makes it instantly apparent.)

In closing this introduction, we wish to express two thoughts about future directions:

\noindent
1. We believe that the discovery of the new non-soliton obstructions to the usual tools used to prove scattering is almost as important as the positive results that we prove.
It breaks the long-standing tradition that minimization subject to zero virial produces solitons (or wave collapse with a soliton profile).  In doing so, it highlights an
unexpected inadequacy of prior methods and hopefully will stimulate the investigation of substitutes for the virial identity in the treatment of large data scattering.

\noindent
2. We hope that this paper may provide some preliminary guidelines for the problem of large-data scattering for the Gross--Pitaevskii and cubic-quintic problems with non-zero boundary conditions at infinity.  These problems are much more subtle.  The gauge and Galilei symmetries are both broken and the mass (more accurately, the mass defect relative to the constant background) is no longer coercive.  On the positive side, these problems are known to be well-posed for large data \cite{PG,KOPV2011} and to scatter for small data \cite{GNT3,GNT4}.  There have also been substantial advances in understanding the structure of solitons for these equations \cite{BGS,Maris_Existence of TW}.  Nonetheless, considerable obstacles (both variational and dispersive in character) currently stand in the way of a proof of scattering below the soliton threshold.

\subsection{Notation}
\label{SUBSEC:notation}
We write $X\lesssim Y$ to indicate that there exists some constant $C>0$ so that $X \leq CY$.
If the constant $C$ depends on some parameter $r$, we write $X\les_r Y$. We write $X\sim Y$ if $X\lesssim Y \lesssim X$.

\begin{definition}
A pair  of exponents $(q, r)$ is \emph{admissible} if $2\leq q, r \leq \infty$ and $\frac{2}{q} + \frac{3}{r} = \frac{3}{2}$.
Given a spacetime slab $I\times\R^3$, we define
\begin{align*}
\|u\|_{S^0(I)} := &\sup\Bigl\{ \|u\|_{L_t^qL_x^r(I\times \R^3)} : \text{ $(q,r)$ is admissible }\Bigr\}.
\end{align*}
Analogously, $N^0(I)$ denotes the corresponding dual Strichartz spaces.
\end{definition}

These notations (introduced in \cite{CKSTT}) allow a compact expression of the Strichartz estimates for the Schr\"odinger propagator on $\R^3$:

\begin{lemma}[Strichartz estimates, \cite{GV,KeelTao,Strichartz,Yajima}]\label{LEM:Stri}
Let $I$ be an interval in $\R$ and let $u:I\times\R^3\to \C$ be a solution to
$$
i\partial_t u = -\Delta u + F \qtq{with} u(t=0) = u_0.
$$
Then,
\begin{equation}\label{ZStri1}
\| u \|_{S^0(I)} \lesssim \|u_0\|_{L^2(\R^3)} + \|F\|_{N^0(I)}.
\end{equation}
\end{lemma}

We will often apply this estimate to the derivative of solutions to \eqref{3-5}; for example, combining \eqref{ZStri1} with Sobolev embedding, we observe that
\begin{equation*}
\|u \|_{L_{t,x}^{10}(\R\times\R^3)} \les \|u_0\|_{\dot H^1(\R^3)} + \|\nabla F\|_{N^0(\R)}.
\end{equation*}

Let us now describe our notations for Littlewood--Paley projectors.  Fix $\psi\in C^\infty_c(\R^3)$ that is non-negative, radial, and such that $\psi(x)=1$ if $|x|\leq 1$ and $\psi(x)=0$ if $|x|\geq \frac{11}{10}$.  We then define Fourier multipliers as follows:
\begin{gather*}
\widehat{P_{\leq N} f}(\xi):=\psi\big(\tfrac{\xi}{N}\big)\hat{f}(\xi), \quad \widehat{P_Nf}(\xi):=\Bigl[\psi\big(\tfrac{\xi}{N}\big)-\psi\big(\tfrac{2\xi}{N}\big)\Bigr]\hat{f}(\xi), \\
\text{and} \quad \widehat{P_{> N} f}(\xi):=\Bigl[1- \psi\big(\tfrac{\xi}{N}\big)\Bigr]\hat{f}(\xi).
\end{gather*}

\begin{lemma}[Bernstein inequalities]
For $1\leq p\leq q\leq \infty$, we have
\begin{align*}
\|P_Nf\|_{L^q(\R^3)} &\les N^{\frac{3}{p}-\frac{3}{q}} \|P_Nf\|_{L^p(\R^3)}, \\
\|P_{\geq N}f\|_{L^p(\R^3)} &\les N^{-s}\big\||\nabla|^sP_{\geq N}f\big\|_{L^p(\R^3)}.
\end{align*}
\end{lemma}

The Brezis-Lieb Lemma is a refinement of Fatou's Lemma that has proven invaluable in the calculus of variations:

\begin{lemma}[\cite{BL}]\label{LEM:B-L}
Fix $1\leq p < \infty$ and suppose $\{f_n\}_{n\in \mathbb{N}}$ is bounded in $L^p(\R^d)$ and converges almost everywhere to $f$.  Then 
\begin{align*}
\lim_{n\to \infty}
\Big(\|f_n\|_{L^{p}(\R^d)}^p
-\|f_n - f \|_{L^{p}(\R^d)}^{p}\Big)
=
\|f\|_{L^{p}(\R^d)}^{p}.
\end{align*}
\end{lemma}

As one last preliminary, we remind the reader of a particular consequence of local smoothing; see \cite[Lemma~3.7]{keraani-h1}, \cite[Lemma~2.5]{Berbec}, or \cite[Corollary~4.15]{ClayNotes}.

\begin{lemma}\label{L:Keraani3.7}
Given $\phi\in \dot H^1(\R^d)$,
$$
\| \nabla e^{it\Delta} \phi \|_{L^2_{t,x}([-T,T]\times\{|x|\leq R\})}^3 \lesssim
     T^{\frac25} R^{\frac{11}5} \| e^{it\Delta} \phi \|_{L^{10}_{t,x}}\| \nabla \phi \|_{L^2_x}^2.
$$
\end{lemma}

\section{Solitons}\label{SEC:Solitons}

The purpose of this section is to discuss soliton solutions of the cubic-quintic NLS, that is, solutions of the form $u(t,x)=e^{i\omega t}P(x)$
with $P\in H^1(\R^3)$ and $\omega\in\R$.    Evidently, this corresponds to the analysis of the elliptic equation
\begin{align}\label{eq:soliton}
-\Delta P + |P|^4P - |P|^2P + \omega P=0.
\end{align}
(Elliptic regularity guarantees that distributional solutions in $H^1(\R^3)$ are actually classical solutions, so we need not quibble about the
appropriate notion of solution.)

In this paper, we are primarily interested in understanding thresholds for scattering.  As we will see, this allows us to focus our attention on positive radial solutions to \eqref{eq:soliton};
these are typically known as ground states.  Incidentally, by the well-known results of Gidas, Ni, and Nirenberg \cite{GNN}, positive solutions to \eqref{eq:soliton} are automatically
radial (about some point).

Before stating our main results about ground state solitons, we pause to recall some fundamental identities valid for all $H^1_x$ solutions:

\begin{lemma}[Pohozaev Identites]\label{L:Poho}
Let $P\in H^1(\R^3)$ obey \eqref{eq:soliton} for some $\omega\in\C$.  Then
\begin{align}\label{E:Poh1}
\int |\nabla P|^2+|P|^6-|P|^4+\omega |P|^2\, dx &=0 \\
	\label{E:Poh2}
\text{and} \qquad \int \tfrac16|\nabla P|^2+\tfrac16|P|^6-\tfrac14|P|^4 +\tfrac\omega2 |P|^2 \, dx&=0.
\end{align}
In particular, if $P\not\equiv 0$ then $\omega\in(0,\tfrac3{16})$.
\end{lemma}

\begin{proof}
Pairing the equation \eqref{eq:soliton} with $\bar P(x)$ yields \eqref{E:Poh1}, from which it also follows that $\omega\in\R$.

The second identity, \eqref{E:Poh2}, follows from pairing the equation with $x\cdot\nabla \bar P(x)$.
As it is not given that  $x\cdot\nabla \bar P(x)\in H^1(\R^3)$, a little extra care is needed here; complete details can be found in \cite[\S2.1]{Berestycki}.

For $\omega\geq\tfrac3{16}$, the polynomial appearing in \eqref{E:Poh2} is non-negative;
this then forces $P\equiv 0$.  Lastly, taking a linear combination of \eqref{E:Poh1} and \eqref{E:Poh2} yields
\begin{equation}\label{eq:w4}
\int |P|^4\,dx = 4\omega \int |P|^2\, dx,
\end{equation}
which shows that $\omega>0$ unless $P\equiv 0$.
\end{proof}

\begin{theorem}[Basic properties of ground state solitons]\label{T:solitons}
For each $0<\omega<\tfrac3{16}$, there is a unique non-negative radially symmetric solution $\Pw\in H^1(\R^3)$ to
\begin{equation}\label{E:Pdefn}
-\Delta \Pw^{ }+\Pw^5-\Pw^3+\omega \Pw^{ }=0.
\end{equation}
In fact, $\Pw$ is strictly positive and a decreasing function of $|x|$.  Moreover,
\begin{SL}
\item $\Pw$ is a non-degenerate saddle point of $u \mapsto E(u)+\tfrac\omega2M(u)$, when viewed as a functional on $H^1_\textup{rad}(\R^3)$.  The Morse index is equal to one.
Furthermore among all non-zero solutions $u$ to \eqref{eq:soliton}, $P_\omega$ achieves the minimal value of $E(u)+\tfrac\omega2M(u)$.
\item The map $\omega\mapsto \Pw$ is $C^1;$ indeed, it is real analytic.
\item The $\dot H^1_x$ norm of $\Pw$ is strictly increasing; indeed,
\begin{equation}\label{dGdw}
\frac{d\ }{d\omega} \int |\nabla \Pw(x)|^2\, dx = \tfrac32 M(\Pw).
\end{equation}
Using the notation $\beta(\omega):=\int |\Pw(x)|^6\, dx \big/ \int |\nabla\Pw(x)|^2\, dx$, we also have
\begin{equation}\label{dMdw}
\frac{d\ }{d\omega} M(\Pw) <  \tfrac{3\beta(\omega)-1}{2\omega} M(\Pw).
\end{equation}
\item As $\omega\to0$ we have $\beta(\omega)=\omega \beta(g) + O(\omega^2)$,
$$
M( P_{\omega} ) = \tfrac{1}{\sqrt{\omega}} \int g(x)^2\,dx + \tfrac{\sqrt{\omega}}2 \int g(x)^6\,dx + O(\omega^{3/2} ),
$$
and
$$
E( P_{\omega} ) = \tfrac{\sqrt{\omega}}2 \int g(x)^2\,dx - \tfrac{\omega^{3/2}}{12} \int g(x)^6\,dx + O(\omega^{5/2} ),
$$
where $g$ is the unique positive radially symmetric solution to $-\Delta g - g^3 + g =0$.
\item As $\omega\to\frac3{16}$ we have $\beta(\omega)\to \infty$, $M(\Pw)\to\infty$, and $E(\Pw)\to-\infty$; indeed,
$$
\beta(\omega) \sim (\tfrac3{16}-\omega)^{-1}, \quad  (M( P_{\omega} ) \sim (\tfrac3{16}-\omega)^{-3}, \qtq{and} |E( P_{\omega} )| \sim (\tfrac3{16}-\omega)^{-3}.
$$
\end{SL}
\end{theorem}

The remainder of this section is devoted to the proof of this theorem.

Existence of solitons for $\omega\in(0,\frac3{16})$ follows from the main theorem in \cite{Berestycki}.  We remind the reader of the construction, since we will use it
in verifying other parts of the theorem.  Defining $p(u)=\tfrac16|u|^6-\tfrac14|u|^4+\tfrac\omega2|u|^2$, one first shows that the variational problem
\begin{equation}\label{E:VP}
\text{minimize}\ \ \int\! |\nabla P(x)|^2\,dx \ \ \text{over}\ \  P\in H^1(\R^3) \ \ \text{with}\ \ \!\int\! p(P(x))\,dx=-1
\end{equation}
admits an optimizer  $\tilde P_\omega$ that is non-negative, radially symmetric, and non-increasing (as a function of radius).
This is a standard argument, using rearrangement inequalities and the Strauss Lemma (compactness of
the embedding  $H^1_x\hookrightarrow L^4_x$ for radial functions).  As an optimizer, $\tilde P_\omega$ obeys the Euler--Lagrange equation
\begin{equation}\label{E:var char}
-\Delta \tilde P_\omega + \lambda \bigl[ \tilde P_\omega^5 - \tilde P_\omega^3 + \omega \tilde P_\omega\bigr] =0 \qtq{with} \lambda=\tfrac16\int |\nabla \tilde P_\omega|^2\,dx,
\end{equation}
the value of the Lagrange multiplier $\lambda$ being determined via the Pohozaev identities for this equation.  Lastly, observe that
$P_\omega(x)=\tilde P_\omega(x/\sqrt{\lambda})$ is the sought-after soliton.

The variational argument gives a bound on the supremum of the solution, which we will use later:
\begin{equation}\label{E:b0 defn}
\Pw(x) \leq b_0(\omega):= \sqrt{\smash[b]{\tfrac12}+\smash[b]{\tfrac12}\smash[t]{\sqrt{1-4\omega}}} \quad\text{for all $\omega\in(0,\tfrac3{16})$ and $x\in\R^3$.}
\end{equation}
To see this, we consider a generic trial function $P\in H^1_x$ for \eqref{E:VP} and the competitor
$Q(x) = \min( b_0, |P(\rho x)|)$, where $\rho>0$ is chosen so that $Q$ obeys the constraint $\int p(Q)\,dx=-1$.  As $p(u) > p(b_0)$ whenever $|u|>b_0$, it
is not difficult to verify that if $|P|$ exceeds $b_0$ at some point, then $\rho>1$ and so $Q$ will be a better trial function.  This proves \eqref{E:b0 defn}.

Uniqueness of non-negative radial $H^1_x$ solutions to the equation \eqref{E:Pdefn} follows from \cite[Theorem~$1'$]{SerrinTang2000}.
As discussed in that paper, there is a lengthy history of proving uniqueness of such ground state solutions; however,  \cite{SerrinTang2000}
is the earliest paper we have been able to find which covers the cubic-quintic equation.

That $\Pw$ is strictly positive, rather than merely being non-negative, follows from the fact that it solves a second-order ordinary differential equation (when viewed as a function of radius).
Specifically, at any point $x$ where $\Pw(x)=0$ one must also have $\nabla \Pw(x)=0$ to avoid a sign change, but then uniqueness for the ODE forces $\Pw\equiv 0$.
A similar argument also shows that $\Pw$ is a \emph{strictly} decreasing function of $|x|$.  The significance of this is that the strict rearrangement inequalities of Brothers and Ziemer \cite{BrothersZiemer}, then guarantee that any optimizer of \eqref{E:VP} must automatically be spherically symmetric (about some point) and
thus (by the uniqueness of radially symmetric solutions) a translate of $\tilde P_\omega$.  This proves that $\Pw$ is uniquely characterized by the variational problem \eqref{E:VP}.

Incidentally, ODE methods also allow one to see that $\Pw(x)$ is a real-analytic function of $x$ and that as $|x|\to\infty$,
\begin{equation}\label{E:P asym}
|x| \exp\bigl\{\sqrt{\omega}\,|x| \} \Pw(x) \to c \qtq{and} \exp\bigl\{\sqrt{\omega}\,|x| \} \, x\cdot\nabla\Pw(x) \to - c\sqrt{\omega}
\end{equation}
for some $c=c(\omega)>0$.  For these assertions, see Theorems~1.8.1 and~3.8.1 in \cite{CodLev}.

We have now verified all the assertions at the beginning of Theorem~\ref{T:solitons}.  To aid the reader in navigating the proofs of the numbered parts of the theorem,
each is given its own subsection below.

Theorem~\ref{T:solitons} gives considerable insight into the nature of the mass/energy curve for the system of solitons $\Pw$,
which is depicted in Figure~\ref{F:Pw}.  Some further results are included in Subsection~\ref{SS:II}; however, at least one conspicuous question remains open:

\begin{figure}
\noindent
\begin{center}
\fbox{
\setlength{\unitlength}{1mm}
\begin{picture}(85,37)(-7,-7)
\put(0,-5){\vector(0,1){32}}\put(-5,25){$E$}
\put(-5,10){\line(1,0){7}}\qbezier(2,10)(2.5,11)(3,12)\qbezier(3,12)(4,10)(5,8)\qbezier(5,8)(5.5,9)(6,10)\put(6,10){\vector(1,0){65}}\put(72,7){$M$}
\put(-3,6){$0$}
\qbezier(50,10)(40,15)(25,20)
\qbezier(25,20)(40,15)(70,12)
\qbezier(50,10)(60,5)(70,-5)
\end{picture}
}
\end{center}
\caption{Schematic mass/energy curve $(M(\Pw),E(\Pw))$ for ground state solitons, based on numerics.}\label{F:Pw}
\end{figure}
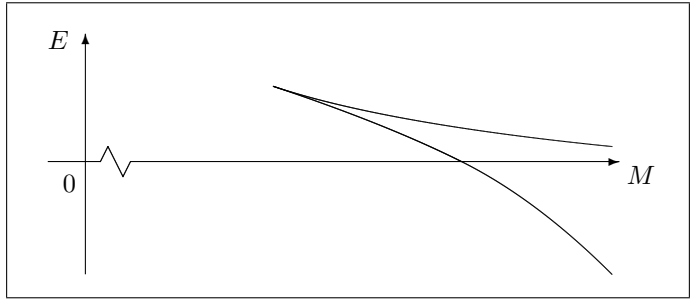

\begin{conjecture}\label{Conj:mass}
There is an $\omega_*\in(0,\tfrac3{16})$ so that $\omega\mapsto M(\Pw)$ is strictly decreasing for $\omega<\omega_*$ and strictly increasing for $\omega>\omega_*$.
\end{conjecture}

This conjecture is strongly supported by numerics.  By \eqref{E:dEdM}, we see that the map $\omega\mapsto E(\Pw)$ would have the opposite monotonicities,
so giving rise to the cusp seen in Figure~\ref{F:Pw}.  By \eqref{E:d2EdM2} we see that this conjecture would also demonstrate the convexity and concavity of the mass/energy curve when $\omega<\omega_*$ and
$\omega>\omega_*$, respectively, that is depicted in Figure~\ref{F:Pw}.

\subsection{Solitons as critical points of $E+\frac\omega2M$}\label{SS:Morse}
Recalling the definitions of mass and energy shows that
$$
E(u)+\tfrac\omega2M(u) = \int_{\R^3} \tfrac12|\nabla u(x)|^2 + \tfrac16|u(x)|^6 - \tfrac14|u(x)|^4 + \tfrac\omega2|u(x)|^2\,dx;
$$
thus, critical points of this functional are precisely solutions to \eqref{eq:soliton}.  A proof that the solution $\Pw$ constructed by the
variational argument minimizes $E(u)+\tfrac\omega2M(u)$ over the class of all non-zero solutions to \eqref{eq:soliton} can be found in \cite[\S4.3]{Berestycki}.

The Hessian of $u\mapsto E(u)+\tfrac\omega2M(u)$ at $\Pw$ is given by the operator
$$
\Lw = -\Delta + 5 \Pw^4 - 3 \Pw^2 + \omega.
$$
Note that this operator is also important as the linearization of \eqref{E:Pdefn} around $\Pw$.  Thirdly, this operator appears when seeking ground states via the
shooting method; specifically, it gives the equation for the derivative of the solution with respect to the shooting parameter, which is the value of the solution
at the origin.

As $\Pw$ is radially symmetric, the operator $\Lw$ decomposes as a direct sum of operators, one in each angular momentum eigenspace.
Moreover, separation of variables associates to each such restricted operator an ODE that can be studied by Sturm--Liouville methods.  With this
in mind, we regard $\Pw$ as a function of $r=|x|$, rather than $x$, whenever this is more convenient.
Note also that  $\Lw$ is a relatively compact perturbation of $-\Delta+\omega$ and so has only discrete spectrum on $(-\infty,\omega)$.

The statement of Theorem~\ref{T:solitons} focuses attention on the zero angular momentum component of $\Lw$ because standard simple arguments of
general applicability (see below) yield the structure at higher angular momentum.  Indeed, the central subtle question is whether zero is an eigenvalue of the
restriction of $\Lw$ to the space of radially symmetric functions.  This question arises, for example, in understanding the stability/instability
of the ground state solitons; see \cite{Shatah,Weinstein'}.  It is also a key step in the influential Coffman--Kolodner approach to uniqueness of the ground state
(cf. \cite{Coffman,Kolodner,Kwong,McLeod}).

While a number of general uniqueness theorems have been proven by the Coffman--Kolodner method, we have not found an instance that
is applicable to the cubic-quintic problem of interest to us.  Our proof of Theorem~\ref{T:solitons}(i) is strongly influenced by these papers, particularly \cite{McLeod},
and our success here perhaps leads to further generalization this method.  We do not pursue this since the competing approach used in \cite{SerrinTang2000} provides the uniqueness statement we need.

Proposition~\ref{P:ell} in this subsection reduces the study to the key question identified above; Proposition~\ref{P:2.5} settles it.
Note that the presence of $\nabla\Pw$ in the null space of
$\Lw$ is not surprising: from a PDE point of view it corresponds to the translation invariance of \eqref{E:Pdefn}, while from a variational
perspective it represents the same invariance for $u\mapsto E(u)+\frac\omega2M(u)$.

\begin{proposition}\label{P:ell}
Fix $\ell=0,1,2,\ldots$ and consider the restriction of $\Lw$ to functions of the form $f(|x|)Y(x/|x|)$ where $Y$ is a spherical harmonic
of degree $\ell$.
\begin{SL}
\item When $\ell=0$ the operator has exactly one negative eigenvalue; it is simple.
\item When $\ell=1$ there are no negative eigenvalues.  Zero is an eigenvalue and its eigenspace
is spanned by the three components of $\nabla \Pw$.
\item When $\ell\geq2$ the operator is positive definite.
\end{SL}
\end{proposition}

\begin{proof}
We treat the parts in order, beginning with $\ell=0$.  First observe that $v=\tfrac{x}{|x|}\cdot\nabla \Pw(x)$ belongs to $H^1_x$ and is radially symmetric.  Using this
as a trial vector shows that $\Lw$ has negative spectrum; indeed, a little computation shows
$$
\langle v, \Lw v\rangle = -2 \int_{\R^3} \frac{|x\cdot\nabla \Pw(x)|^2}{|x|^4}\,dx < 0.
$$

Next we show that $\Lw$ has at most one negative eigenvalue (counting multiplicity).  This is to be expected from the variational characterization \eqref{E:VP} of $\Pw$: it is
defined by a singly constrained minimization problem and one constraint can only counteract one concave direction.  To make this heuristic rigorous we argue by contradiction;
specifically, we suppose that there is a two-dimensional subspace $\mathcal M\subset H^1_x$ for which
\begin{equation}\label{downward dog}
\langle u, \Lw u\rangle  < 0 \quad\text{whenever $u\in\mathcal M\setminus\{0\}$}.
\end{equation}
Next, we apply the Implicit Function Theorem to construct a two-dimensional surface $\mathcal M\ni u\mapsto Q(u)\in H^1_x$ of functions obeying the constraint
\begin{equation}\label{Q constraint}
\int\!\tfrac16|Q|^6-\tfrac14|Q|^4+\tfrac\omega2|Q|^2\,dx=
\int\!\tfrac16|\Pw|^6-\tfrac14|\Pw|^4+\tfrac\omega2|\Pw|^2\,dx.
\end{equation}
To be precise, having chosen a real-valued $w\in H^1_x$ with
$
\int \bigl[\Pw^5 - \Pw^3 + \omega \Pw\bigr] w\,dx = 1,
$
there is an $\eps>0$ and a smooth real-valued function $h$ defined on $B(0,\eps)\subseteq\mathcal M$ so that $h(0)=0$ and
$
Q(u) := \Pw + u + h(u) w
$
obeys \eqref{Q constraint} for all $u\in B(0,\eps)$.

Recall that $\tilde\Pw(x)=\Pw(x\sqrt\lambda)$ where $\lambda$ is as in \eqref{E:var char}.  Applying the same rescaling $\tilde Q(u)(x) = Q(u)(x\sqrt\lambda)$, we see that the variational
characterization of $\Pw$ guarantees that $\int |\nabla Q(u)|^2\,dx \geq \int |\nabla \Pw|^2\,dx$.  Combining this with \eqref{Q constraint} implies
\begin{equation}\label{min at P}
E(Q(u))+\tfrac\omega2 M(Q(u)) \geq E(\Pw)+\tfrac\omega2 M(\Pw) \quad\text{for all $u\in B(0,\eps)\subseteq \mathcal M$.}
\end{equation}

We will now reach a contradiction by examining the behaviour of $E(Q)+\tfrac\omega2M(Q)$ near the point $\Pw$ on the surface.  As $\Pw$ is a soliton, $dE+\tfrac\omega2dM=0$
at $\Pw$.  Thus,
$$
\tfrac{d^2\ }{dt^2}\Bigr|_{t=0} E(Q(tu))+\tfrac\omega2 M(Q(tu)) = \bigl\langle \bigl[u + (u\cdot\nabla h(0))w\bigr] , \Lw \bigl[u + (u\cdot\nabla h(0))w\bigr]\bigr\rangle.
$$
As $\mathcal M$ is two-dimensional, we can choose $u\in \mathcal M\setminus\{0\}$ with $u\perp \nabla h(0)$.  By \eqref{downward dog}, this choice makes the above expression negative, which contradicts \eqref{min at P}.  This completes the proof that $\Lw$ has at most one negative eigenvalue (counting multiplicity).

It will be easier to understand the case of higher angular momentum if we separate variables.  To this end, let $Y$ be a spherical harmonic of degree $\ell$
and $u(x)=f(|x|)Y(x/|x|)$, then $\Lw u = \lambda u$ if and only if
\begin{equation}\label{E:Lw in Y}
 - f''(r) - \tfrac2rf'(r) + \tfrac{\ell(\ell+1)}{r^2} f(r) + [5\Pw^4(r)-3\Pw^2(r)+\omega] f(r) = \lambda f(r) .
\end{equation}

Consider first the case $\ell=1$.  An easy calculation (take the gradient of \eqref{E:Pdefn}) shows that the components of $\nabla\Pw$ are eigenvectors
of $\Lw$ with eigenvalue $0$.  Moreover, if we write $\nabla\Pw(x) = -\frac{x}{|x|}f(|x|)$ then the function $f$ is strictly positive; this is because
$\Pw$ is a strictly decreasing function of radius.  By the Sturm Oscillation Theorem  (cf. \cite[Ch. 8]{CodLev}), positivity of $f$ guarantees that zero
is at the bottom of the spectrum.
Lastly, the differential equation obeyed by the Wronskian shows that the ODE \eqref{E:Lw in Y} cannot admit a second zero-energy eigenfunction in $L^2(r^2dr)$.

Assume now $\ell\geq 2$.  Suppose, toward a contradiction, that $u(x)=f(|x|)Y(x/|x|)\in H^1$ obeys $\Lw u = \lambda u$ with $\lambda\leq 0$ and $Y$ a spherical harmonic
of degree $\ell$. Then $v(x)=f(|x|)\frac{x}{|x|}$ obeys $\langle v, \Lw v\rangle<0$ in contradiction to the preceding paragraph.
\end{proof}

\begin{proposition}\label{P:2.5}
Let $\delta$ be the solution to
\begin{equation}\label{E:delta}
 - \delta''(r) - \tfrac2r\delta'(r) + [5\Pw^4(r)-3\Pw^2(r)+\omega] \delta(r) = 0
\end{equation}
obeying $\delta(0)=1$.  Then $\delta(r)\to-\infty$ as $r\to\infty$.  Correspondingly,
zero is not an eigenvalue of $\Lw$ restricted to radial functions.
\end{proposition}

\begin{proof}
First we should explain why such a solution $\delta(r)$ exists and is unique.  The ODE \eqref{E:delta} has a regular singular point at $r=0$; however, changing variables to
$\sigma(r)=r\delta(r)$ transforms it to
\begin{equation}\label{E:r delta}
- \sigma''(r) + [5\Pw^4(r)-3\Pw^2(r)+\omega] \sigma(r) = 0.
\end{equation}
Thus $\delta$ corresponds (uniquely!) to initial data $\sigma(0)=0$, $\sigma'(0)=1$.

The basic question we need to answer is whether $\delta(r)\to 0$ as $r\to\infty$.  Applying Theorem~3.8.1 from \cite{CodLev} to \eqref{E:r delta} shows that
$\delta$ either grows or decays exponentially as $r\to\infty$.  Proposition~\ref{P:ell}(i) and the Sturm Oscillation Theorem guarantee that $\delta$ changes sign exactly
once. Thus, in the growing case we must have $\delta(r)\to-\infty$.  In the decaying case, $\delta\in L^2(r^2\,dr)$, which makes zero an eigenvalue.  Thus, we may prove the proposition by assuming that $\delta(r)\to 0$ as $r\to\infty$ and reaching a contradiction.

To obtain the contradiction, we will be following the arguments in \cite{McLeod}, which become applicable with the addition of one new observation.
The main part of the argument uses the Sturm Separation Theorem to compare $\delta$ with a function of the form $v_\lambda(r):=r\Pw'(r)+\lambda\Pw(r)$, with $\lambda\in\R$.
Note that $v_\lambda$ obeys
\begin{equation}\label{E:v_lambda}
 - v_\lambda''(r) - \tfrac2r v_\lambda'(r) + [5\Pw^4(r)-3\Pw^2(r)+\omega] v_\lambda(r) =  - I(\lambda,\Pw(r))
\end{equation}
with
\begin{equation}\label{E:I of u and lambda}
 I(\lambda,u) := 2[u^5 - u^3 + \omega u] -  2\lambda[2u^5-u^3].
\end{equation}

We wish to choose the parameter $\lambda$ so that $\lambda>0$ and $I(\lambda,\Pw(r_1))=0$, where $r_1\in(0,\infty)$ denotes the sole zero of $\delta(r)$ noted above.
A little effort shows that the possibility of finding such a $\lambda$ follows from
\begin{equation}\label{P of r1}
a_0 < \Pw(r_1) < \tfrac1{\sqrt{2}} \qtq{where} a_0 = \sqrt{\smash[b]{\tfrac12}-\smash[b]{\tfrac12}\smash[t]{\sqrt{1-4\omega}}}.
\end{equation}

Verification of the lower bound on can be found in the proof of \cite[Lemma 4]{McLeod}; the argument is as follows: From the equations obeyed by $\Pw(r)$ and $\delta(r)$, we have
$$
\tfrac{d\ }{dr} [ r^4 \Pw'(r) \delta'(r) ] = r^4 [(\Pw^5-\Pw^3+\omega \Pw)\delta]'(r)
$$
and so integrating over $[r_1,\infty)$ and using \eqref{E:P asym} yields
\begin{equation}\label{McL 4}
- r_1^4 \Pw'(r_1) \delta'(r_1)  = - \int_{r_1}^\infty 4r^3 [\Pw(r)^5-\Pw(r)^3+\omega \Pw(r)]\delta(r)\,dr.
\end{equation}
Now if $\Pw(r_1)\leq a_0$ then $[\Pw(r)^5-\Pw(r)^3+\omega \Pw(r)]> 0$ for $r>r_1$ which means that RHS\eqref{McL 4} is strictly positive.  (Recall that $\delta$ changes sign
at $r_1$ and remains negative there-after.)  This yields a contradiction because $\Pw'(r_1)<0$ and $\delta'(r_1)<0$ making LHS\eqref{McL 4} negative.

Verifying the upper bound on $\Pw(r_1)$ in \eqref{P of r1} is the new input.  Suppose $\Pw(r_1)\geq2^{-1/2}$ and so $\Pw(r)>2^{-1/2}$ on $[0,r_1)$.  Comparing the equations for $\delta$ and $\Pw$
we see that this implies that $\Pw$ oscillates faster than $\delta$ on the interval $[0,r_1]$, in the sense of the Sturm Comparison Theorem.  This gives a contradiction
since $\delta(r_1)=0$, while $\Pw$ is non-vanishing.

Having verified \eqref{P of r1} we may now choose $\lambda>0$ so that $I(\lambda,\Pw(r_1))=0$.  We claim that $I(\lambda,\Pw(r))<0$ for $r\in[0,r_1)$
and $I(\lambda,\Pw(r))>0$ for $r\in(r_1,\infty)$.  To check this, we first rewrite \eqref{E:I of u and lambda} as
$$
I(\lambda,u) = 2u(u^2-a_0^2)(u^2-b_0^2) - 2 \lambda u^3(2u^2-1),
$$
where $a_0$ is as in \eqref{P of r1} and $b_0$ is given in \eqref{E:b0 defn}.  Note that $a_0^2<\frac12<b_0^2$.  When $2^{-1/2}\leq u\leq b_0$ we see quickly that  $I(\lambda,u)<0$. On the other hand, $u\mapsto(u^5 - u^3 + \omega u)(2u^5-u^3)^{-1}$ is increasing on $0<u<2^{-1/2}$.  Recalling that $\Pw(r)$ is
a decreasing function bounded by $b_0$, see \eqref{E:b0 defn}, this proves the claim.

Next we claim that $v_\lambda$ has exactly one zero in $[0,r_1]$.  To see this note that $v_\lambda(0)>0$ and while $v_\lambda$ remains positive, it oscillates faster than $\delta$.
Thus $v_\lambda$ has at least one zero in $[0,r_1]$.  Let $r_0$ be the smallest such zero and note that $v_\lambda$ changes sign there. Indeed, if
$v_\lambda(r_0)=v_\lambda'(r_0)=0$, then $v_\lambda''(r_0)=I(\lambda,\Pw(r_0))<0$, which would be inconsistent with the fact that $v_\lambda$ is positive on $[0,r_0)$.  On the other hand, after changing sign, $v_\lambda$ oscillates more slowly than $\delta$ and so cannot vanish on $(r_0,r_1]$.

We now compare $\delta$ and $v_\lambda$ on $[r_1,\infty)$.  First note that $v_\lambda(r_1)<0$.   If $v_\lambda$ has no zeros on  $(r_1,\infty)$, then it oscillates faster than
$\delta$ because $I>0$ on this interval.  This is of course self-contradictory, since $\delta$ vanishes at both ends of this interval (this was our contradiction hypothesis).

To reach a contradiction, we will show that $v_\lambda$ does not have a zero in $(r_1,\infty)$.
To this end, suppose that $v_\lambda$ did vanish there and let $r_2$ denote the smallest
such zero.  A repetition of the second derivative argument above (now with reversed signs) shows that $v_\lambda$ must change sign at $r_2$.   However, by \eqref{E:P asym} we also
have $v_\lambda(r)<0$ for $r$ sufficiently large.  Thus there is a point $r_3\in(r_2,\infty)$ at which $v_\lambda$ also vanishes and such that $v_\lambda>0$ on $(r_2,r_3)$.
Now, on this smaller interval, $v_\lambda$ oscillates more slowly than $\delta$.  This yields a contradiction by forcing $\delta$ to have a zero in $(r_2,r_3)$.\end{proof}

\subsection{Smoothness of $\omega\mapsto\Pw$}\label{SS:C1}

Fix $0<\omega_0<\frac3{16}$.  First, we observe that $P_{\omega_0}$ belongs to a real-analytic branch $\omega\mapsto u\in H^1_\textup{rad}(\R^3)$ of solutions to
\begin{equation}\label{u branch}
-\Delta u + u^5 - u^3 + \omega u =0
\end{equation}
and subsequently identify this branch as consisting of ground state solitons $\Pw$.

As $\mathcal{L}_{\omega_0}$ is a relatively compact perturbation of $-\Delta+\omega_0$, Theorem~\ref{T:solitons}(i) together with the Fredholm Theorem guarantees that
it is an isomorphism of $H^{1}_\rad(\R^3)$ onto $H^{-1}_\rad(\R^3)$. (Recall that the subscript indicates radially symmetric functions.)
Thus, applying the Implicit Function Theorem to the mapping $u\mapsto\text{LHS\eqref{u branch}}$, we see that there is a real analytic curve $\omega\mapsto u(\omega)\in H^1_\rad$
of solutions to \eqref{u branch} defined in a neighbourhood of $\omega_0$ and obeying $u(\omega_0) = P_{\omega_0}$.

Note that the Implicit Function Theorem (or the contraction mapping argument used to prove it) also guarantees that $\omega\mapsto u(\omega)$ is smooth with values in $H^2(\R^3)$
and so (by Sobolev embedding) also smooth when viewed as a map into $L^\infty(\R^3)$.

To conclude that $u(\omega)=\Pw$, it suffices to show that $u(\omega)$ is non-negative, which we will do via a spectral theory argument.  For each $\omega$, the
Schr\"odinger operator
$$
H_\omega f:= -\Delta f + u(\omega)^4 f - u(\omega)^2 f + \omega f
$$
has a zero eigenvalue, as witnessed by $u(\omega)$ itself.  Moreover, when $\omega=\omega_0$ this is an isolated simple eigenvalue at the bottom of the spectrum (as follows,
for example, from the Weyl criterion and Sturm comparison).  Thus, by eigenvalue/vector perturbation theory, the zero eigenvalue of $H_\omega$ must also be at the bottom of the spectrum for $|\omega-\omega_0|$ small.  This in turn guarantees that $u(\omega)$ must be positive (cf. \cite[Theorem~XIII.46]{RS4}).

This completes the proof of part (ii) of Theorem~\ref{T:solitons}.\qed

\subsection{Further identities and inequalities}\label{SS:II}
Using just \eqref{E:Pdefn} and the definitions of mass and energy of $P_\omega$, we obtain
\begin{align}\label{E:dEdM}
\frac{dE}{d\omega} = \Bigl\langle\frac{d\Pw}{d\omega},\ -\Delta\Pw+\Pw^5-\Pw^3\Bigr\rangle_{L^2} = \Bigl\langle\frac{d\Pw}{d\omega},\ - \omega \Pw\Bigr\rangle_{L^2}
	= - \frac{\omega}{2}\frac{dM}{d\omega},
\end{align}
where $E(\omega)=E(\Pw)$ and $M(\omega)=M(\Pw)$.
Thus at any point $\omega_0$ where $\tfrac{dM}{d\omega}\neq0$,
\begin{align}\label{E:d2EdM2}
\frac{d^2E}{dM^2} = \Bigl(\frac{dM}{d\omega}\Bigr)^{-1} \frac{d\ }{d\omega}\Bigl[ \frac{dE}{d\omega} \div \frac{dM}{d\omega}\Bigr] =- \frac{1}{2} \Bigl(\frac{dM}{d\omega}\Bigr)^{-1}.
\end{align}

For much of what follows, it is convenient to introduce the notation
\begin{equation}\label{eq:Zbeta}
\beta (u):=\frac{\int |u|^6dx}{\int |\nabla u|^2 dx } \quad\text{for $u\in H^1(\R^3)$},
\end{equation}
which generalizes the notion introduced in Theorem~\ref{T:solitons}, namely, $\beta(\omega) = \beta(\Pw)$.

By exploiting the Pohozaev identities from Lemma~\ref{L:Poho}, we obtain very compact expressions
for the key quantities associated to our solitons:
\begin{alignat}{2}
\int P_{\omega}^2\, dx &=\tfrac{\beta(\omega)+1}{3\omega }\int |\nabla P_{\omega}|^2\,dx,
	  & \int P_{\omega}^4\,dx  &=\tfrac{4[\beta(\omega)+1]}{3}\int |\nabla P_{\omega}|^2\, dx, \label{eq:w24} \\
\int P_{\omega}^6\, dx&= \beta(\omega) \int |\nabla P_{\omega}|^2\,dx,
	\qtq{and} & E(P_{\omega}) &= \tfrac{1-\beta(\omega)}{6} \int |\nabla P_{\omega}|^2\,dx. \label{eq:w6E}
\end{alignat}

Numerical investigations give compelling evidence for the following:

\begin{conjecture}\label{Conj:beta}
The mapping $\omega\mapsto\beta(\omega)$ is injective; that is, ground state solitons are uniquely identified by the ratio $\beta(\omega)$.
\end{conjecture}

Several properties of solitons are singled out by their $\beta$ ratio.  For example, a soliton has zero energy if and only if it has $\beta=1$.  In subsequent
sections, several results have cumbersome formulations because we have not yet been able to verify this conjecture.

By Theorem~\ref{T:solitons}(iv,v), injectivity of $\beta(\omega)$ is equivalent to its strict monotonicity.
There is another equivalent formulation of Conjecture~\ref{Conj:beta}, relating it directly to the variational characterization \eqref{E:VP} of $\Pw$, namely,
$$
\log(\omega)\mapsto\log\Bigl[\tfrac16\!\int |\nabla \tilde P_\omega|^2\,dx \Bigr] = \tfrac23\log\Bigl[\tfrac16\!\int |\nabla P_\omega|^2\,dx \Bigr]
	\ \ \text{is a strictly convex function.}
$$
This equivalence is easily verified using \eqref{dGdw} and \eqref{eq:w24}.  The identity stated here follows immediately from \eqref{E:var char}.

We turn next to the proof of \eqref{dGdw}.  It is very simple: By \eqref{eq:w24} and \eqref{eq:w6E} we have $\int |\nabla P_\omega|^2\,dx = 3( E+\tfrac\omega2 M)$
and so \eqref{E:dEdM} yields
\begin{align}\label{E:dH1doemga}
  \frac{d\ }{d\omega} \int |\nabla P_\omega|^2\,dx = \frac{d\ }{d\omega} 3\bigl( E+\tfrac\omega2 M \bigr) = \tfrac32 M.
\end{align}

The proof of \eqref{dMdw} is rather more complicated.  We begin by computing the matrix elements of $\mathcal{L}_\omega$ given in Table~\ref{Table:Lw}.
Differentiating \eqref{E:Pdefn} with respect to $\omega$ yields $\mathcal{L}_\omega \tfrac{\partial \Pw}{\partial\omega} = - \Pw$ from which we immediately
deduce the top row of matrix elements in Table~\ref{Table:Lw}.  For the second row, we use the relation
\begin{align}
\mathcal{L}_\omega \Pw &= 4\Pw^5 - 2 \Pw^3,
\end{align}
which is a direct consequence of \eqref{E:Pdefn}, followed by \eqref{eq:w24} and \eqref{eq:w6E}.

\begin{table}[h]
\begin{center}
\def\strt{\vrule width 0em depth 1.2ex height 2.5ex}%
\begin{tabular}{|c||c|c|c|}
\hline
\strt$\langle\cdot,\mathcal{L}_\omega\,\cdot\rangle$ & $\tfrac{\partial \Pw}{\partial\omega}$ & $\Pw$ & $x\cdot \nabla \Pw + \tfrac32 \Pw$ \\
\hline&&&\\[-0.85\linespacing]\hline
\strt$\tfrac{\partial \Pw}{\partial\omega}$ & $-\tfrac12 M'(\omega)$ &  $-M(\omega)$ & $0$ \\
\hline
\strt$\Pw$  & $-M(\omega)$ & $\tfrac43[\beta(\omega)-2]G(\omega)$ & $2[\beta(\omega) - 1]G(\omega)$\\
\hline
\strt$x\cdot \nabla \Pw + \tfrac32 \Pw$ & $0$ & $2[\beta(\omega) - 1]G(\omega)$ & $[3\beta(\omega)-1]G(\omega)$ \\
\hline
\end{tabular}
\end{center}
\caption{Key matrix elements of $\mathcal{L}_\omega$.  Here $G(\omega):=\int |\nabla \Pw|^2\,dx$.}\label{Table:Lw}
\end{table}

For the last row in Table~\ref{Table:Lw} we use the basic identity
\begin{equation}\label{E:441}
\Delta (x\cdot \nabla u) = (x_k u_k)_{jj} = 2 \delta_{jk} u_{kj} + x_k u_{kjj} = 2 \Delta u + x\cdot\nabla \Delta u
\end{equation}
to deduce
\begin{align}
\mathcal{L}_\omega (x\cdot \nabla \Pw + \tfrac32 \Pw) &= - 2 \Delta \Pw + 6 \Pw^5 - 3 \Pw^3 = 4 \Pw^5 - \Pw^3 - 2\omega \Pw
\end{align}
and then apply \eqref{eq:w24} and \eqref{eq:w6E} to simplify the expressions for the resulting inner products.

Observe that the determinant of the bottom right $2\times2$ block in Table~\ref{Table:Lw} is
$$
 \tfrac43[\beta(\omega)-2][3\beta(\omega)-1]G(\omega)^2 - 4[\beta(\omega) - 1]^2G(\omega)^2 = -\tfrac{4}3[\beta(\omega)+1]G(\omega)^2,
$$
which is always negative.  From Subsection~\ref{SS:Morse} we know that $\mathcal{L}_\omega$ has exacly one negative eigenvalue and no zero eigenvalue.
Thus by Sylvester's law of inertia, the full $3\times 3$ determinant must be negative, that is,
\begin{align*}
\tfrac{2}3 M'(\omega) [\beta(\omega)+1]G(\omega)^2 - M(\omega)^2 [3\beta(\omega)-1]G(\omega) < 0.
\end{align*}
Employing \eqref{eq:w24} we may rewrite this as
\begin{align*}
\bigl\{2\omega M'(\omega) - [3\beta(\omega)-1]M(\omega)\bigr\} M(\omega) G(\omega) < 0,
\end{align*}
which proves \eqref{dMdw}.

One consequence of \eqref{dMdw} is that the mass is strictly decreasing for $\beta(\omega)\leq 1/3$.  Note that by Conjecture~\ref{Conj:beta} and the results of
Subsection~\ref{SS:0} below, this corresponds to an interval of the form $\omega\in(0,\omega_{1/3}]$.  The significance of this is heightened by the numerical
observation that the mass $M(\omega)=\int  \Pw^2\,dx$ at $\omega_{1/3}$ is extremely close to the minimal value of the mass; see the numerical data in Table~\ref{Table:num}.
Indeed, the masses differ by about one part in a thousand.

\begin{table}[h]
\begin{center}
\def\strt{\vrule width 0em depth 1.2ex height 2.5ex}%
\begin{tabular}{|c|c|c|c|c|}
\hline
\strt$\omega$ & $\beta(\omega)$ & $\int \Pw^2$ & $\int |\nabla\Pw|^2$  & Significance \\
\hline\hline
0.023926 & 0.33333 & 189.68  &  10.211 & $\beta=1/3$ \\
\hline
0.025544 & 0.36054 & 189.46  &  10.671 & Minimal M \\
\hline
0.054735 & 1.00000 & 240.45  &  19.741 & $E = 0$  \\
\hline
\end{tabular}
\end{center}
\caption{Selected numerical data}\label{Table:num}
\end{table}

The minimum of the mass of $\Pw$ as $\omega$ varies is represented by the cusp point in Figure~\ref{F:Pw}.  Note that \eqref{E:dEdM} combined with the analyticity
of $E(\omega)$ and $M(\omega)$ guarantees rigorously that this will indeed be a cusp.

In Section~\ref{SEC:Virial} we will find that $\beta=1/3$ has a further significance: the curves of solitons and rescaled solitons meet at this point.

From \eqref{E:dEdM} we see that the mass/energy curve (see Figure~\ref{F:Pw}) steepens as $\omega$ grows.  Our last result for this subsection captures this phenomenon
in a manner that will be used twice in the proof of Theorem~\ref{T:Rbdry}.

\begin{lemma}\label{L:cuspy}
Fix $0<\omega_0<\omega_1<3/16$ with $E(P_{\omega_0})=E(P_{\omega_1})$.  If $E(\Pw)\geq E(P_{\omega_0})$ for all $\omega\in(\omega_0,\omega_1)$,
then $M(\omega_1)<M(\omega_0)$.  If on the other hand, $E(\Pw)\leq E(P_{\omega_0})$ for all $\omega\in(\omega_0,\omega_1)$,
then $M(\omega_1)>M(\omega_0)$.
\end{lemma}

\begin{proof}
From \eqref{E:dEdM} and integration by parts, we have
\begin{align*}
-\!\int_{\omega_0}^{\omega_1} \!\tfrac{d\ }{d\omega} M(\Pw)\,d\omega = \int_{\omega_0}^{\omega_1}\! \tfrac{d\ }{d\omega} E(\Pw)\,\tfrac{2d\omega}{\omega}
&=\tfrac{2}{\omega_1} E(P_{\omega_1}) - \tfrac{2}{\omega_0}E(P_{\omega_0}) + \!\int_{\omega_0}^{\omega_1} E(P_{\omega}) \,\tfrac{2d\omega}{\omega^2}.
\end{align*}
Thus using $E(P_{\omega_0})=E(P_{\omega_1})$ we deduce
\begin{align*}
- M(\omega_1) + M(\omega_0) &=  \int_{\omega_0}^{\omega_1} \bigl[E(P_{\omega})-E(P_{\omega_0})\bigr] \,\tfrac{2d\omega}{\omega^2}.
\end{align*}
This immediately proves the lemma.  Note that we obtain strict inequalities because $\omega\mapsto E(\Pw)$ is a non-constant
analytic function; in particular, it is not constant on the interval $(\omega_0,\omega_1)$.
\end{proof}

\subsection{Asymptotics as $\omega\to0$}\label{SS:0}
The key idea in our analysis of this case is to change variables to
\begin{equation}\label{E:P/root}
u(x;\omega):= \tfrac{1}{\sqrt{\omega}}\Pw\bigl(\tfrac{x}{\sqrt{\omega}}\bigr).
\end{equation}
This new unknown obeys
\begin{equation}\label{E:P/root eq}
-\Delta u + \omega u^5 - u^3 + u =0,
\end{equation}
which we regard as a perturbation of the well-studied problem with $\omega=0$.  To this end, let $g$ denote the unique non-negative radially symmetric solution to
\begin{equation}\label{E:g eqn}
-\Delta g - g^3 + g =0.
\end{equation}
The function $g$ is strictly positive, smooth, and exponentially decaying.  Most important for us is the fact that the linearized operator
$$
L : f \mapsto -\Delta f - 3 g^2 f +f
$$
has been shown to be an isomorphism of $H^1_\rad$ onto $H^{-1}_\rad$; see \cite{Coffman}.  Thus, by the Implicit Function Theorem, there is a real-analytic family
$v(x;\omega)\in H^1_\rad$ of solutions to \eqref{E:P/root eq} defined in a neighbourhood of $\omega=0$  that obeys
\begin{equation}\label{u of om}
v(x;\omega) = g(x) - \omega [ L^{-1} g^5](x) + O(\omega^2) \qquad\text{in $H^1(\R^3)$ sense.}
\end{equation}

That $v(x;\omega)$ is non-negative and so must equal $u(x;\omega)$ follows by the ground state argument presented in Subsection~\ref{SS:C1}.

We now move on to studying asymptotics of the mass and energy.  The leading terms are almost immediate from the above; see below.
The second order terms require the observation that
\begin{equation}\label{E:L inv g}
L ( g + x\cdot \nabla g ) = -2 g,
\end{equation}
which follows from \eqref{E:g eqn}, \eqref{E:441}, and direct computation.  Indeed, using this, \eqref{E:P/root}, and \eqref{u of om} we have
\begin{align*}
M(\Pw) = \omega^{-1/2} \| u \|_{L^2_x}^2  &= \omega^{-1/2} \bigl[ \| g \|_{L^2_x}^2 - 2 \omega \langle g, L^{-1} g^5 \rangle + O(\omega^2 ) \bigr] \\
&= \omega^{-1/2} \| g \|_{L^2_x}^2  + \omega^{1/2} \langle g + x\cdot \nabla g , g^5 \rangle + O(\omega^{3/2} ) \\
&= \omega^{-1/2} \| g \|_{L^2_x}^2  + \tfrac12 \omega^{1/2} \| g \|_{L^6_x}^6  + O(\omega^{3/2} ).
\end{align*}
Note that this also gives asymptotics for the $L^4$-norm of $\Pw$ via the identity \eqref{eq:w4}.

Using the Pohozaev identities and the asymptotics for the mass we deduce
\begin{align*}
E(\Pw) &= \tfrac\omega2 M(\Pw) - \tfrac13 \| \Pw \|_{L^6_x}^6  \\
&= \tfrac12\omega^{1/2} \| g \|_{L^2_x}^2  + \tfrac14 \omega^{3/2} \| g \|_{L^6_x}^6
	- \tfrac13 \omega^{3/2} \| g \|_{L^6_x}^6 + O(\omega^{5/2} ) \\
&= \tfrac12 \omega^{1/2} \| g \|_{L^2_x}^2 - \tfrac{1}{12} \omega^{3/2} \| g \|_{L^6_x}^6 + O(\omega^{5/2} ).
\end{align*}
We also used here the leading term asymptotic for the $L^6$-norm of $\Pw$, which is evident from \eqref{u of om}.

A similar analysis yields $\beta(\omega)=\omega \beta(g) + O(\omega^2)$.

\subsection{Asymptotics as $\omega\to\frac3{16}$}\label{SS:3/16}
We will exploit the variational characterization \eqref{E:VP} of solitons. As there, we write $p(u)=\frac16|u|^6-\frac14|u|^4+\frac\omega2 |u|^2$.

First we consider an explicit trial function $v$ defined by $v(x) = \sqrt{3}/2$ when $|x|<R$,  $v(x) = 0$ when $|x|>R+h$, and $v(x)$ is the interpolating linear function
of $|x|$ for intermediate values of $|x|$.  The parameters $R$ and $h$ are chosen so that
$$
\int_{|x|<R} p(v(x)) \,dx = -2 \qtq{and then} \int_{\R^3} p(v(x)) \,dx = -1.
$$
As $p(\sqrt{3}/2) = -\tfrac38(\tfrac3{16}-\omega)$ we have $R\sim (\tfrac3{16}-\omega)^{-1/3}$; a slightly longer computation reveals $h\sim (\tfrac3{16}-\omega)^{2/3}$
and therefore $\|\nabla v\|_{L^2}^2 \sim (\tfrac3{16}-\omega)^{-4/3}$.

By the variational characterization of $\Pw$, via the rescaled function $\tilde P_\omega$, we then deduce that
$$
\|\nabla\tilde \Pw\|_{L^2}^2 \lesssim (\tfrac3{16}-\omega)^{-4/3}
	\qtq{and so} \|\nabla \Pw\|_{L^2}^2 \lesssim (\tfrac3{16}-\omega)^{-2}.
$$

The remainder of the argument is based on a direct analysis of the optimizer $\tilde \Pw$, which is necessarily positive, radial, and decreasing.

We define new lengths $R$ and $h$ so that
$$
|\tilde \Pw(x)| \geq \tfrac12  \qtq{when $|x|\leq R$ and}   |\tilde \Pw(x)| \leq \tfrac14  \quad\text{when $|x|\geq R+h$.}
$$
As $p(u)\gtrsim - (\tfrac3{16}-\omega)$ throughout $\C$ and $p(u)\gtrsim |u|^2 \geq 0$ when $|u|\leq\frac12$,  the constraint $\int p(\tilde \Pw) =-1$ guarantees that
$R \gtrsim (\tfrac3{16}-\omega)^{-1/3}$ and that
\begin{equation}\label{E:L2 beyond R}
 \int_{|x|>R} |\tilde \Pw(x)|^2\,dx \lesssim 1 + (\tfrac3{16}-\omega) R^3 \lesssim (\tfrac3{16}-\omega) R^3.
\end{equation}
In particular, $h R^2 \lesssim  (\tfrac3{16}-\omega) R^3$, which is to say, $h\lesssim (\tfrac3{16}-\omega) R$.

We can now get a lower bound on $\int|\nabla \tilde \Pw|^2$ by comparing it to the least value taken by any function that equals $\frac12$ on the sphere of
radius $R$ and equals $\frac14$ on the sphere of radius $R+h$.  This minimum is achieved by a harmonic function, specifically, by
$U(x)=\tfrac{R(R+h)}{4h|x|} - \tfrac{R-h}{4h}$. Thus,
\begin{equation*}
\int_{R<|x|<R+h} |\nabla \tilde \Pw(x)|^2\,dx \geq \tfrac{R(R+h)}{16h} \gtrsim R (\tfrac3{16}-\omega)^{-1} \gtrsim  (\tfrac3{16}-\omega)^{-4/3}.
\end{equation*}
From this and our earlier upper bound on $\int|\nabla \tilde \Pw|^2$, we deduce that
\begin{equation*}
\int_{\R^3} |\nabla \tilde \Pw(x)|^2\,dx \sim (\tfrac3{16}-\omega)^{-4/3} \qtq{and so} R \sim (\tfrac3{16}-\omega)^{-1/3}.
\end{equation*}
From the size of $R$, \eqref{E:L2 beyond R}, and \eqref{E:b0 defn}, we also obtain $\int |\tilde \Pw|^2 \sim (\tfrac3{16}-\omega)^{-1}$.

Rescaling we obtain
$$
\int_{\R^3} |\nabla \Pw(x)|^2\,dx \sim (\tfrac3{16}-\omega)^{-2} \qtq{and} M(\omega)=\int_{\R^3} |\Pw(x)|^2\,dx \sim (\tfrac3{16}-\omega)^{-3}
$$
and then by employing the Pohozaev identities,
$$
|E(\omega)|\sim\int_{\R^3} |\Pw(x)|^6\,dx \sim \int_{\R^3} |\Pw(x)|^4\,dx \sim (\tfrac3{16}-\omega)^{-3} \qtq{and} \beta(\omega) \sim (\tfrac3{16}-\omega)^{-1}.
$$

This completes the proof of the final part of Theorem~\ref{T:solitons}.\qed

\section{Solitons as Gagliardo--Nirenberg--H\"older optimizers}\label{SEC:GNH}

The appearance of solitons as optimizers in Sobolev and Gagliardo--Nirenberg inequalities has played a key
role in the analysis of the focusing NLS with a single power nonlinearity.  In particular, this is an essential
ingredient in the determination of thresholds for well-posedness and for scattering in these cases; see \cite{AkahoriNawa,DHR08,FXC11,KenigMerle,Weinstein}.

In this section we discuss an analogous characterization of solitons for the cubic-quintic problem in terms of a
one-parameter family of inequalities we dub Gagliardo--Nirenberg--H\"older inequalities; see \eqref{eq:GN} below.
We suggest this name because they interpolate between the Gagliardo--Nirenberg and H\"older inequalities
\begin{equation}\label{E:GN&H}
\|u\|_{L^4}^4 \lesssim \|u\|_{L^2} \|\nabla u\|_{L^2}^3 \qtq{and} \|u\|_{L^4}^4 \leq\|u\|_{L^2}\|u\|_{L^6}^{3}
\end{equation}
as the parameter $0<\alpha<\infty$ varies.  While the veracity of the inequality \eqref{eq:GN} below for some constant $C_\alpha>0$ is
easily seen by taking a (weighted) geometric mean of these classical inequalities, this argument tells us little about
the optimal constant or any optimizing functions.

\begin{proposition}[$\alpha$-Gagliardo--Nirenberg--H\"older inequality]\label{prop:GN}
\ Fix $0<\alpha<\infty$ and let $C_\alpha>0$ denote the optimal (i.e. infimal) constant so that
\begin{align}\label{eq:GN}
\|u\|_{L^4(\R^3)}^4\leq C_{\alpha}\|u\|_{L^2(\R^3)}\|u\|_{L^6(\R^3)}^{\frac{3\alpha}{1+\alpha}} \|\nabla u\|_{L^2(\R^3)}^{\frac{3}{1+\alpha}}
	\qquad\text{for all $u\in H^1(\R^3)$.}
\end{align}
Then $C_\alpha<\infty$ and the inequality admits optimizers; moreover, every optimizer $v$ is of the form
$v(x) = \lambda P_\omega (\rho(x-x_0))$ where $\lambda\in\C$, $x_0\in\R^3$, $\rho>0$, and $P_\omega$ is a ground state soliton with
$\beta(P_\omega)=\alpha$.
\end{proposition}

\begin{remark}
Theorem~\ref{T:solitons} shows that $\omega\mapsto\beta(\omega)$ is real-analytic with $\beta(\omega)\to0$ as $\omega\to 0$ and
$\beta(\omega)\to\infty$ as $\omega\to \frac3{16}$.  Thus for any $\alpha$, there are only finitely many $\omega$ so that $\beta(\omega)=\alpha$.
\end{remark}

One shortcoming in our result is that we are unable (at this time) to guarantee that all solitons appear as optimizers of \eqref{eq:GN}
nor are we able to preclude that two different solitons are both optimizers for the same inequality.  As will be apparent from what
follows, a positive resolution of Conjecture~\ref{Conj:beta} would obviate these concerns.

\begin{proof}
That $C_\alpha<\infty$ follows from \eqref{E:GN&H} as noted above.  It can also be deduced in a similar manner from Sobolev embedding
(specifically, $\dot{H}^1(\R^3)\subset L^6(\R^3)$).

Next, we prove that the optimal constant $C_\alpha$ is achieved.  To this end, let
\begin{align}\label{eq:ZGN1}
F(u):&=\frac{\|u\|_{L^2}\|u\|_{L^6}^{\frac{3\alpha}{1+\alpha}}\|\nabla u\|_{L^2}^{\frac{3}{1+\alpha}}}{\|u\|_{L^4}^4},
\qtq{so that} C_{\alpha}^{-1}=\inf_{u\in H^1(\R^3)\setminus\{0\}}F(u),
\end{align}
and let $\{u_n\}_{n\in\mathbb{N}}\subset H^1_x$ be a sequence realizing the infimum (i.e. $F(u_n)\to C_{\alpha}^{-1}$).
By the usual rearrangement inequalities, we may assume that $u_n$ are non-negative, radially symmetric, and non-increasing.  By
exploiting the fact that $F(u)$ is invariant under the rescaling $u(x)\mapsto \lambda u(x/r)$, we may also arrange that
$\|u_n\|_{L^2}=1$ and $\|\nabla u_n\|_{L^2}=1$ for all $n$.

Passing to a subsequence, if necessary, we may assume that $\{u_n\}$ converges to some $u_\infty \in H^1_x$ in both the
weak topology on $H^1_x$ and the norm topology on $L^4_x$.  The latter assertion relies on the fact that $u_n$ are
radially symmetric and the well-known compactness of the embedding $H^1_{\textup{rad}}(\R^3)\hookrightarrow L^4(\R^3)$;
see, for example, \cite[p.~570]{Weinstein}.  Sobolev embedding further guarantees that $u_n\rightharpoonup u_\infty$ weakly
in $L^6_x$.

Next we verify that $u_\infty\not\equiv 0$.  By the normalizations $\|u_n\|_{L^2}=\|\nabla u_n\|_{L^2}=1$ and H\"older's inequality,
we have
$$
C_\alpha^{-1} = \lim_{n\to\infty} F(u_n) = \lim_{n\to\infty}\frac{\|u_n\|_{L^2}^{\frac\alpha{1+\alpha}}\|u_n\|_{L^6}^{\frac{3\alpha}{1+\alpha}}}{\|u_n\|_{L^4}^4}
	\geq \lim_{n\to\infty} \frac{\|u_n\|_{L^4}^{\frac{4\alpha}{1+\alpha}}}{\|u_n\|_{L^4}^4}
	= \|u_\infty\|_{L^4}^{-\frac{4}{1+\alpha}},
$$
which shows that indeed $u_\infty\not\equiv 0$.

Recall that the norm on any Banach space is weakly lower semicontinuous.  Applying this in the numerator of $F(u_n)$ and using that
$\|u_n\|_{L^4}\to\|u_\infty\|_{L^4}\neq0$ in the denominator yields
$$
C_\alpha^{-1} = \lim_{n\to\infty} F(u_n) \geq F(u_\infty) \geq \inf_{u\in H^1(\R^3)\setminus\{0\}}F(u) = C_\alpha^{-1}.
$$
Thus equality holds throughout this line and $u_\infty$ is an optimizer for \eqref{eq:GN}.

It remains to demonstrate that all optimizers $v$ are ground state solitons up to the obvious symmetries.
We begin by treating the case that $v$ is non-negative and radially symmetric; at the end of the proof we will reduce
to this case via rearrangement inequalities.

As $v$ is a minimizer of $F$, it must satisfy the corresponding Euler-Lagrange equation:
$\frac{d}{d\eps}|_{\eps=0} F(v+\eps\phi)=0$ for all $\phi\in C^{\infty}_c(\R^3; \R)$.  Thus, direct computation
shows that $v$ is a distributional solution to the following equation:
\begin{align*}
-\Delta v +\alpha \, \frac{\int |\nabla v|^2dx}{\int v^6\,dx}v^5
-\frac{4(1+\alpha)}{3} \frac{\int |\nabla v|^2dx}{\int v^4\,dx}v^3
+\frac{1+\alpha}{3} \frac{\int|\nabla v|^2dx}{\int v^2\,dx}v=0.
\end{align*}

A little further computation shows that if we define $Q(x)=\lambda^{-1} v(x/\rho)$ where $\lambda,\rho>0$ are given by
\begin{equation*}
\lambda^2=\frac{4(1+\alpha)}{3\alpha}\frac{\int v^6\,dx}{\int v^4\,dx} \qtq{and} \rho^2=\frac{16(1+\alpha)^2}{9\alpha}\frac{\int v^6\,dx\int |\nabla v|^2\,dx}{\big(\int v^4\,dx\big)^2},
\end{equation*}
then,  $Q$ satisfies
\begin{align}\label{E:Q is P}
-\Delta Q+Q^5-Q^3 +\omega Q=0 \qtq{with} \omega=\frac{3\alpha}{16(1+\alpha)}\frac{\big(\int v^4\,dx\big)^2}{\int v^2\,dx\int v^6\,dx}.
\end{align}

As $Q$ is non-negative and radially symmetric, Theorem~\ref{T:solitons} shows that \eqref{E:Q is P} uniquely identifies $Q(x)$ as being $\Pw(x)$ for this value of $\omega$;
moreover,
$$
\beta(Q) = \tfrac{\rho^2}{\lambda^4} \beta(v) = \alpha.
$$

To complete the discussion of optimizers $v$ for \eqref{eq:GN} we now need to consider the case that $v$ is not non-negative and radial.  Standard rearrangement inequalities guarantee
that the symmetric decreasing rearrangement $v^*$ of $v$ will also be an optimizer.  By what we have just proven, $v^*$ must then be the rescaling of a ground state soliton $\Pw$.  Consequently, all level sets
of $v^*$ have zero measure.  Thus by the strict rearrangement inequalities of \cite{BrothersZiemer}, it follows that $v$ agrees with $v^*$ up to
translation and multiplication by a unimodular complex number.  This shows that $v(x)=\lambda P_\omega (\rho (x-x_0))$ for some $\lambda\in\C$, $x_0\in\R^3$, $\rho>0$, and $P_\omega$ a ground state soliton with
$\beta(P_\omega)=\alpha$, thereby completing the proof of Proposition \ref{prop:GN}.
\end{proof}

\begin{lemma}\label{lemma:varitions of M,E}
Fix $\alpha>0$ and let $Q_{\alpha}$ denote a ground state soliton that
optimizes the $\alpha$-Gagliardo--Nirenberg--H\"older inequality \eqref{eq:GN}.  Then, the optimal constant $C_\alpha$ in that inequality is given by
\begin{align}\label{eq:6}
C_{\alpha}&=\frac{4(1+\alpha)}{3\alpha^{\frac{\alpha}{2(1+\alpha)}}}\frac{1}{\|Q_{\alpha}\|_{L^2(\R^3)}
\|\nabla Q_{\alpha}\|_{L^2(\R^3)}^{\frac{1-\alpha}{1+\alpha}}}.
\end{align}
Moreover, if \ $0<\gamma<\alpha$ and $Q_\gamma$ is a ground state soliton that
optimizes the $\gamma$-Gagliardo-Nirenberg-H\"older inequality, then
\begin{align}
\frac{\|\nabla Q_{\alpha}\|_{L^2(\R^3)}}{\|\nabla Q_{\gamma}\|_{L^2(\R^3)}}&> \bigg(\frac{\alpha}{\gamma}\bigg)^{\frac14}  > 1, \label{eq:comparison_kinetic}\\
\frac{M(Q_{\gamma})}{M(Q_{\alpha})}&> \frac{(1+\gamma)^2\alpha^{\frac12}}{(1+\alpha)^2\gamma^{\frac12}},
\quad  \text{for } 0 <  \g < \alpha \leq 1,
\label{eq:comparison_mass}\\
\frac{M(Q_{\gamma})}{M(Q_{\alpha})}&< \frac{(1+\gamma)^2\alpha^{\frac12}}{(1+\alpha)^2\gamma^{\frac12}},
\quad
\text{for } 1 \leq  \g < \alpha,
\label{eq:comparison_mass2}\\
\frac{E(Q_{\alpha})}{E(Q_{\gamma})}&> \frac{( \alpha-1)\alpha^{\frac12}}{(\gamma-1)\gamma^{\frac12}},
\quad  \text{for }0< \g < \alpha <  1 \text{ or } 1 < \g < \alpha.
\label{eq:comparison_energy}
\end{align}
\end{lemma}

\begin{proof}
As $Q_\alpha$ is a ground state soliton with $\beta(Q_\alpha)=\alpha$, it obeys the relations \eqref{eq:w24} and \eqref{eq:w6E}; in particular,
$\|Q_\alpha\|_{L^4}^4 = \frac{4(\alpha+1)}3 \|\nabla Q_\alpha\|_{L^2}^2$ and $\|Q_\alpha\|_{L^6(\R^3)}^6=\alpha \|\nabla Q_\alpha\|_{L^2}^2$.
Substituting these into
$$
C_{\alpha} = \frac{\|Q_\alpha\|_{L^4(\R^3)}^4}{\|Q_\alpha\|_{L^2(\R^3)}\|Q_\alpha\|_{L^6(\R^3)}^{\frac{3\alpha}{1+\alpha}} \|\nabla Q_\alpha\|_{L^2(\R^3)}^{\frac{3}{1+\alpha}}}
$$
leads immediately to \eqref{eq:6}.

As $\gamma\neq\alpha$, Proposition~\ref{prop:GN} guarantees that $Q_\gamma$ is not an optimizer for the $\alpha$-Gagliardo-Nirenberg-H\"older inequality.
Using \eqref{eq:6} and the fact that $Q_\gamma$ obeys the relations \eqref{eq:w24} and \eqref{eq:w6E}, this yields
\begin{align}\label{B1}
\frac{\|Q_{\gamma}\|_{L^2}\|\nabla Q_{\gamma}\|_{L^2}^{\frac{1-\alpha}{1+\alpha}}}{\|Q_{\alpha}\|_{L^2}\|\nabla Q_{\alpha}\|_{L^2}^{\frac{1-\alpha}{1+\alpha}}}
>\frac{(1+\gamma)\gamma^{-\frac{\alpha}{2(1+\alpha)}}}{(1+\alpha)\alpha^{-\frac{\alpha}{2(1+\alpha)}}}.
\end{align}

Reversing the roles of $\alpha$ and $\gamma$, we also obtain
\begin{align}\label{B2}
\frac{\|Q_{\alpha}\|_{L^2}\|\nabla Q_{\alpha}\|_{L^2}^{\frac{1-\gamma}{1+\gamma}}}{\|Q_{\gamma}\|_{L^2}\|\nabla Q_{\gamma}\|_{L^2}^{\frac{1-\gamma}{1+\gamma}}}
>\frac{(1+\alpha)\alpha^{-\frac{\gamma}{2(1+\gamma)}}}{(1+\gamma)\gamma^{-\frac{\gamma}{2(1+\gamma)}}}.
\end{align}

Combining \eqref{B1} and \eqref{B2} gives
\begin{align}\label{eq:double_ineq_masses}
\frac{1+\gamma}{1+\alpha}
\bigg(\frac{\gamma}{\alpha}\bigg)^{\!-\frac{\alpha}{2(1+\alpha)}}
\bigg(\frac{\|\nabla Q_{\alpha}\|_{L^2}}{\|\nabla Q_{\gamma}\|_{L^2}}\bigg)^{\!\frac{1-\alpha}{1+\alpha}}
& <\frac{\|Q_{\gamma}\|_{L^2}}{\|Q_{\alpha}\|_{L^2}} \notag \\
& <\frac{1+\gamma}{1+\alpha}\bigg(\frac{\gamma}{\alpha}\bigg)^{\!-\frac{\gamma}{2(1+\gamma)}}
\bigg(\frac{\|\nabla Q_{\alpha}\|_{L^2}}{\|\nabla Q_{\gamma}\|_{L^2}}\bigg)^{\!\frac{1-\gamma}{1+\gamma}},
\end{align}
from which we will derive the remaining assertions of the lemma.

Skipping over the middle term in \eqref{eq:double_ineq_masses} and rearranging gives
\begin{align*}
\bigg(\frac{\|\nabla Q_{\alpha}\|_{L^2}}{\|\nabla Q_{\gamma}\|_{L^2}}\bigg)^{\frac{2(\alpha-\gamma)}{(1+\gamma)(1+\alpha)}}
>\bigg(\frac{\alpha}{\g}\bigg)^{\frac{\alpha - \gamma}{2(1+\gamma)(1+\alpha)}},
\end{align*}
which then implies \eqref{eq:comparison_kinetic} because $\alpha>\gamma>0$.

From \eqref{eq:comparison_kinetic} and the first inequality in \eqref{eq:double_ineq_masses} we deduce \eqref{eq:comparison_mass}
while \eqref{eq:comparison_mass2} follows from \eqref{eq:comparison_kinetic} and the second inequality in \eqref{eq:double_ineq_masses}.
Note that when $\alpha = 1$ or $\g = 1$, \eqref{eq:comparison_mass} and \eqref{eq:comparison_mass2}
follow directly from \eqref{eq:double_ineq_masses}. Finally, \eqref{eq:comparison_energy}
follows from  \eqref{eq:comparison_kinetic} and \eqref{eq:w6E}.
\end{proof}

In view of \eqref{eq:comparison_kinetic} we may say colloquially that
\begin{equation}\label{E:dGda}
\alpha \mapsto \|\nabla Q_{\alpha}\|_{L^2(\R^3)}^2 \qquad\text{is strictly increasing}.
\end{equation}
However, we have not proved that $\alpha$ uniquely determines a ground state soliton $Q_\alpha$.  Correspondingly, the proper formulation
is that \eqref{E:dGda} holds for \emph{any system} of optimizing solitons $\{Q_\alpha\}_{\alpha\in(0,\infty)}$.

As a counterpoint, recall from \eqref{dGdw} that $\omega\mapsto \|\nabla P_{\omega}\|_{L^2(\R^3)}^2$ is strictly increasing.
Combining this with \eqref{E:dGda} yields the following:

\begin{corollary}
If $P_\omega$ and $P_{\omega'}$ are optimizers for the Gagliardo--Nirenberg--H\"older inequalities with parameters $\alpha>\gamma$, respectively,
then $\omega > \omega'$.
\end{corollary}

This does not resolve Conjecture~\ref{Conj:beta} because we do not know that every $P_\omega$ occurs as an optimizer in \eqref{eq:GN}
for some $\alpha$.

As $\alpha\mapsto\alpha^{-1/2}(1+\alpha)^2$ is increasing for $\alpha >\frac{1}{3}$ and decreasing for $\alpha < \frac{1}{3}$, the inequalities
\eqref{eq:comparison_mass} and \eqref{eq:comparison_mass2} imply
\begin{equation}\label{E:dMda}
\alpha \mapsto M(Q_{\alpha}) \quad\text{is decreasing on $(0,\tfrac13]$ and increasing on $[1,\infty)$.}
\end{equation}
(This is to be interpreted in the same sense as \eqref{E:dGda}.)  Analogously,
\begin{equation}\label{E:dEda}
\alpha \mapsto E(Q_{\alpha}) \quad\text{is increasing on $(0,\tfrac13]$ and decreasing on $[1,\infty)$.}
\end{equation}
Note that the energy $E(Q_{\alpha})$ is negative for $\alpha>1$; see \eqref{eq:w6E}.

\section{Feasible mass/energy pairs}\label{SEC:ME}

The purpose of this section is to give an essentially complete description of the possible pairs $(M(u),E(u))$ for $u\in H^1(\R^3)$.
Recall that
\begin{equation*}
E(u)=\int_{\R^3} \tfrac12|\nabla u(x)|^2 - \tfrac14|u(x)|^4 + \tfrac16 |u(x)|^6 \,dx \qtq{and} M(u) = \int_{\R^3} |u(x)|^2\,dx.
\end{equation*}
Our results are summarized in Figure~\ref{F:ME}.  The symbol $Q_1$ appearing in this figure represents a ground state soliton that optimizes the Gagliardo-Nirenberg-H\"older
inequality with parameter $\alpha=1$, as in Lemma~\ref{lemma:varitions of M,E}.  We will continue to use this notation throughout the section.

Numerically, we find that $Q_1$ is unique; its vital statistics can be found in Table~\ref{Table:num}.  Although we do not have a proof that $Q_1$ is unique, the
masses of all such minimizers are identical.  Indeed, \eqref{eq:6} shows that the mass of $Q_1$ can be expressed in terms of the optimal constant $C_1$ in \eqref{eq:GN}
via $9 M(Q_1)=64 C_1^{-2}$.

We will describe the possible mass/energy pairs by characterizing the possible energies for each fixed value of the mass.  To this end, it is convenient to introduce
the notations
\begin{align}\label{eq:Zvar1}
\Fm(m) := \{E(u): \ u\in H^1(\R^3)\text{ and }  M(u)=m\} \ \ \text{and}\ \ \Em(m):=\inf \Fm(m).\!\!\!
\end{align}
This infimium is always finite; indeed, \eqref{E3/32} below shows that $E(u)\geq - \tfrac3{32} M(u)$.

\begin{figure}[bh]
\noindent
\begin{center}
\fbox{
\setlength{\unitlength}{1mm}
\begin{picture}(85,45)(-8,-20)
\put(0,-17){\vector(0,1){38}}\put(-5,19){$E$}
\put(-5,0){\vector(1,0){73}}\put(68,-3){$M$}
\put(-3,-4){$0$}
\linethickness{0.05mm}
\qbezier( 0, 3)( 9,12)(17,20)\qbezier( 0, 6)( 7,13)(14,20)\qbezier( 0, 9)( 6,15)(11,20)
\qbezier( 0,12)( 4,16)( 8,20)\qbezier( 0,15)( 3,18)( 5,20)\qbezier( 0,18)( 1,19)( 2,20)
\qbezier( 0, 0)(10,10)(20,20)\qbezier( 3, 0)(13,10)(23,20)\qbezier( 6, 0)(16,10)(26,20)
\qbezier( 9, 0)(19,10)(29,20)\qbezier(12, 0)(22,10)(32,20)\qbezier(15, 0)(25,10)(35,20)
\qbezier(18, 0)(28,10)(38,20)\qbezier(21, 0)(31,10)(41,20)\qbezier(24, 0)(34,10)(44,20)
\qbezier(27, 0)(37,10)(47,20)\qbezier(30, 0)(40,10)(50,20)\qbezier(33, 0)(43,10)(53,20)
\qbezier(36, 0)(46,10)(56,20)\qbezier(39, 0)(49,10)(59,20)\qbezier(42, 0)(52,10)(62,20)
\qbezier(45, 0)(55,10)(65,20)
\qbezier(47, -1)(58, 10)(65, 17)\qbezier(49, -2)(58,  7)(65, 14)\qbezier(51, -3)(59,  5)(65, 11)
\qbezier(52, -5)(59,  2)(65,  8)\qbezier(54, -6)(62,  2)(65,  5)\qbezier(56, -7)(63,  0)(65,  2)
\qbezier(57, -9)(63, -3)(65, -1)\qbezier(59,-10)(63, -6)(65, -4)\qbezier(60,-12)(63, -9)(65, -7)
\qbezier(62,-13)(63,-12)(65,-10)\qbezier(63,-15)(64,-14)(65,-13)
\linethickness{0.75mm}
\qbezier(45,0)(55,-5)(65,-17)
\multiput(3,0)(6,0){7}{\line(1,0){3}}
\multiput(0,3)(0,6){3}{\line(0,1){3}}
\put( 0,0){\circle*{1.5}}
\put(45,0){\circle*{1.5}}\put(37,-5){$M(Q_1)$}
\end{picture}
}
\end{center}
\caption{The shaded area indicates feasible mass/energy pairs.  Dots and solid lines indicate boundary points that are achieved; broken lines indicate
boundary points that are not achieved.}\label{F:ME}
\end{figure}
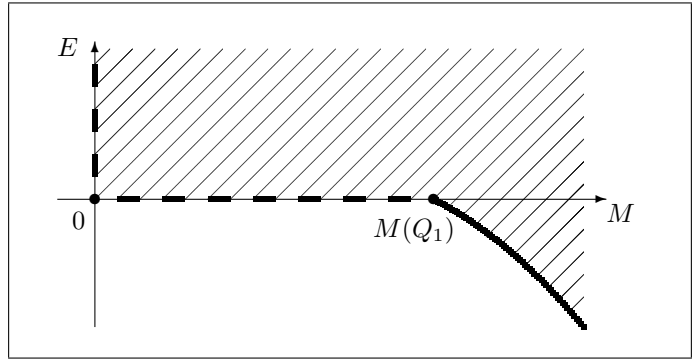

\begin{theorem}[Feasible mass/energy pairs]\label{T:feasible}
\hspace*{0em plus 0.5em minus 0em}$\Em(m)$ is continuous, concave, non-increasing, and  non-positive.  Furthermore,
\begin{SL}
\item If $m = 0$, then $\Fm(m)=\{0\}$ and $\Em(m)=0$.
\item If $0<m<M(Q_1)$, then $\Fm(m)=(0,\infty)$.  In particular, $\Em(m)=0$.
\item If $m = M(Q_1)$, then $\Fm(m)=[0,\infty)$.  Note that $\Em(m)=0$ is now achieved.
\item If $m > M(Q_1)$, then $\Fm(m)=[\Em(m),\infty)$ and $\Em(m)<0$.
\end{SL}

When $m \geq M(Q_1)$ we see that $\Em(m)$ is achieved.  All such minimizers are ground state solitons $P_\omega$ with $\beta(\omega)\geq 1$, up to translation
and multiplication by a unimodular constant.
\end{theorem}

\begin{remark}\label{R:strictly}
In the proof below, we also see that $\Em$ is strictly decreasing on the interval $[M(Q_1),\infty)$, with negative slope at $M(Q_1)$, as depicted in Figure~\ref{F:ME}.
This figure also shows $\Em$ to be smooth on this interval.  We do not currently have a proof of this; it would, however, follow from Conjectures~\ref{Conj:mass}
and~\ref{Conj:beta}.
\end{remark}

\begin{proof}[Proof of Theorem~\ref{T:feasible}]
As $M(u)=0$ enforces $u\equiv 0$, item (i) follows immediately.

Given a non-zero $\phi\in H^1(\R^3)$ and $\xi\in\R^3$, consider $\phi_\xi(x):=e^{i|\xi||x|}\phi(x)$.  Clearly $M(\phi_\xi)=M(\phi)>0$;
moreover, $E(\phi_\xi)=E(\phi)+\tfrac12|\xi|^2 M(\phi)$.  This shows that for $m>0$, the set $\Fm(m)$ of possible energies is a semi-infinite
interval.

It remains only to investigate the curve $\Em(m)$ and to determine when this infimium is achieved.  The analysis below will show that minimizing sequences (and minima,
when they occur) can be made radially symmetric.  We chose to use the multiplier $e^{i|\xi||x|}$ above, rather than $e^{i\xi\cdot x}$, because it preserves this radial symmetry.
As a consequence, we see that all feasible mass/energy pairs can be realized by radial functions and hence, by functions with zero total momentum.

Given $\phi\in H^1(\R^3)$ and $\lambda>0$, let $\phi^{\ld}(x):=\ld^{3/2}\phi (\ld x)$.  Then
\begin{align*}
M(\phi^{\ld})=M(\phi) \qtq{and} E(\phi^{\ld})=\int \tfrac{\ld^2}{2}|\nabla \phi|^2-\tfrac{\ld^3}{4}|\phi|^4+\tfrac{\ld^6}{6}|\phi|^6\,dx.
\end{align*}
In particular, $E(\phi^{\ld})\to 0$ as $\ld\to 0$, which shows that $\Em(m)\leq 0$ for all $m\geq 0$.

The Gagliardo--Nirenberg--H\"older inequality \eqref{eq:GN} with parameter $\alpha=1$ gives
$$
\|u\|_{L^4_x}^4 \leq \tfrac83 \bigl(\tfrac{M(u)}{M(Q_1)}\bigr)^{1/2} \|u\|_{L^6_x}^{3/2} \|\nabla u\|_{L^2_x}^{3/2}
$$
when the optimal constant is written in the form \eqref{eq:6}.  Consequently, using Young's inequality (with powers $\tfrac43$ and~$4$),
\begin{equation}\label{E:Ea1}
\begin{aligned}
E(u) &\geq \tfrac12 \|\nabla u\|_{L^2_x}^2 + \tfrac{1}{6} \|u\|_{L^6_x}^{6}
		- \tfrac23 \bigl(\tfrac{M(u)}{M(Q_1)}\bigr)^{1/2} \|\nabla u\|_{L^2_x}^{3/2} \|u\|_{L^6_x}^{3/2} \\
&\geq \Bigl[ 1 - \bigl(\tfrac{M(u)}{M(Q_1)}\bigr)^{1/2} \Bigr] \Bigl[ \tfrac12 \|\nabla u\|_{L^2_x}^2 + \tfrac{1}{6} \|u\|_{L^6_x}^{6} \Bigr].
\end{aligned}
\end{equation}

From \eqref{E:Ea1} we see that $E(u)>0$ whenever $0<M(u)<M(Q_1)$.  Notice that combined with the preceding analysis, this
completes the proof of part (ii) of the theorem.

By the same analysis we see that $E(u)\geq 0$ when $M(u)=M(Q_1)$.  Moreover, $E(Q_1)=0$, as can be seen from \eqref{eq:w6E}.
This settles part (iii) of the theorem.

Next we show that $\Em(m)<0$ whenever $m>M(Q_1)$.  Observe that if $u\in H^1(\R^3)$ and $u_\lambda(x):=\lambda^{-1/2} u(x/\lambda)$, then
\begin{equation}\label{E:rescl}
M(u_\lambda) = \lambda^2 M(u) \qtq{and} E(u_\lambda) = E(u) - \tfrac{\lambda-1}4 \|u\|_{L^4_x}^4.
\end{equation}
Choosing $u=Q_1$, $\lambda^2 = m/M(Q_1) > 1$, and using that $E(Q_1)=0$, we deduce that $\Em(m)<0$.

Applying the same rescaling argument to a generic trial function shows that $m\mapsto\Em(m)$ is non-increasing, as stated in the theorem.

Incidentally, if the infimum $\Em(m)$ is achieved for some $m>0$, then applying the rescaling argument to an optimizer shows that the right derivative of $\Em$ at $m$
is strictly negative.
(This derivative exists by virtue of the concavity proved below.)  When combined with the results on the existence of extremizers proved below, this justifies
the claims made in Remark~\ref{R:strictly} .

Having shown that $\Em(m)<0$ for $m>M(Q_1)$, we will now be able to prove that the minimum is achieved for these values of $m$.  To this end,
let $\{u_n\}$ be a minimizing sequence for this variational problem (i.e., $M(u_n)=m$ and $E(u_n)\to\Em(m)$).
Using rearrangement inequalities, we see that $u_n$ may be taken radially symmetric, while the identity
\begin{equation}\label{E3/32}
E(u)+\tfrac{3}{32}M(u) = \int \tfrac12|\nabla u|^2 + \tfrac16|u|^2\bigl(|u|^2-\tfrac34\bigr)^2 \, dx
\end{equation}
shows that $u_n$ is bounded in $H^1_x$.   Thus, passing to a subsequence, we may guarantee that $u_n$ converges to some $v\in H^1(\R^3)$
weakly in $H^1_\rad$ and $L^6$ as well as strongly in $L^4$.  From weak lower semicontinuity of norms we deduce that
\begin{equation}\label{E:drpdwn}
E(v) \leq \Em(m) \qtq{and} M(v) \leq m.
\end{equation}
Note that this guarantees $E(v)<0$ and hence $v\not\equiv 0$.

Actually, equality must hold in both parts of \eqref{E:drpdwn}.  If $M(v)<m$, then by rescaling as in \eqref{E:rescl} we could produce
a function $v_\lambda$ with $E(v_\lambda)< \Em(m)$ and $M(v_\lambda)=m$, which is clearly in violation of the definition of $\Em(m)$.
The same violation would arise if $M(v)=m$ but $E(v)<\Em(m)$.  In conclusion, this $v$ demonstrates that when $m>M(Q_1)$ the infimal
energy value $\Em(m)$ is actually achieved.  This completes the verification of part (iv) of the theorem.

Having finished the proof of the numbered parts of the theorem, it remains to show that optimizers are solitons (when $m\neq 0$) and to demonstrate
the concavity and continuity of $\Em(m)$.  (We have already proved that $\Em(m)$ is non-positive and non-increasing.)

That optimizers for $\Em$ are solitons follows immediately from the Euler--Lagrange equation, $dE+\tfrac{\omega}{2}dM=0$.  Indeed, having chosen to write the
Lagrange multiplier as $\omega/2$, we precisely recover \eqref{eq:soliton}.  By the sharp rearrangement inequalities of \cite{BrothersZiemer} we see that
optimizers must be non-negative (after multiplication by a unimodular constant) and spherically symmetric (about some point).  Thus, by Lemma~\ref{L:Poho} and Theorem~\ref{T:solitons}, optimizers must
be ground state solitons, up to symmetries.  That these solitons have $\beta(\Pw)\geq 1$ follows from \eqref{eq:w6E} and the  fact that $\Em\leq 0$.

As we have shown that $\Em$ is monotone, continuity follows from mid-point concavity, which then implies concavity.  Thus it remains only to show that
\begin{equation}\label{mid pt}
\tfrac{1}{2}\Em(m_1)+\tfrac{1}{2}\Em(m_2) \leq \Em\big(\tfrac12 m_1 + \tfrac12 m_2 \big) \qquad\text{for all $m_1, m_2\geq 0$.}
\end{equation}
This will be effected by a moving hyperplane argument.

Given $u\in H^1(\R^3)$ with  $M(u)=\tfrac12 m_1 + \tfrac12 m_2$, choose a plane that partitions the mass into parts $m_1/2$ and $m_2/2$.
Reflecting one side or the other through the chosen plane yields two functions $v_1$ and $v_2\in H^1(\R^3)$ each of which is even
with respect to reflection in the plane and which obey $M(v_1)=m_1$ and $M(v_2)=m_2$.

Directly from the construction we have $\tfrac12 E(v_1)+\tfrac12 E(v_2) = E(u)$.  Taking the infimum over all $u\in H^1(\R^3)$ with  $M(u)=\tfrac12 m_1 + \tfrac12 m_2$
now yields \eqref{mid pt}.
\end{proof}

\section{Rescaled solitons and the virial}\label{SEC:Virial}

The key new player in this section is the \emph{virial},
\begin{equation}\label{E:V defn}
V(u) := \int |\nabla u(x)|^2 + |u(x)|^6 - \tfrac34  |u(x)|^4 \,dx,
\end{equation}
to be regarded as a functional on $H^1(\R^3)$.  While not in perfect accord with the historical meaning (see \cite{Clausius}),
our use of the term `virial' is consistent with the modern usage as can be seen from the following \emph{virial identity}:
\begin{equation}\label{dynam V}
\frac{d\ }{dt} \ 2\Im\int_{\R^3} \overline{u(t,x)}\,x\cdot \nabla u(t,x)\,dx = 4 V(u(t))
\end{equation}
for any solution $u(t)$ to \eqref{3-5}.

Taking a linear combination of the Pohozaev identities \eqref{E:Poh1} and \eqref{E:Poh2} shows that $V(u)=0$ whenever $u$ is a soliton
solution to \eqref{3-5}.  Looking back to the manner in which the Pohozaev identities were derived yields an alternate expression of
the vanishing of the virial for solitons, namely, for any $u\in H^1(\R^3)$,
\begin{equation}\label{V=dE}
V(u) = dE\bigr|_u ( x\cdot \nabla u + \tfrac32 u) \qtq{and} dM\bigr|_u ( x\cdot \nabla u + \tfrac32 u) =0.
\end{equation}
As solitons are precisely the points $u\in H^1(\R^3)$ where $dE+\tfrac\omega2 dM=0$, we deduce that solitons have zero virial.

Although the early sections of this paper focus on variational and elliptic problems, it is important to remember that Theorem~\ref{thm:  main} is a dynamical
statement; we wish to describe the long-time behaviour of solutions to \eqref{3-5} for an open set of initial data.  This focuses our attention on \eqref{dynam V}
as a dynamical substitute for the Pohozaev identities.

For solutions that remain well-localized around a fixed point, such as non-moving solitons,
LHS\eqref{dynam V} will be the time derivative of a bounded function; thus, by integrating over long time intervals, we deduce that $V(u(t))$ vanishes
in an averaged sense.  We will show below that $V>0$ in a certain region $\RR$ in the mass-energy plane, which then precludes this sort of soliton-like
behaviour.

It is far from immediate to see why solutions in the region $\RR$ that merely fail to scatter should remain both in one place and sufficiently well localized
that one may apply the argument just sketched.  Indeed, much of the second half of the paper is devoted to showing that the existence of a non-scattering
solution in the region $\RR$ would guarantee the existence of \emph{another} solution in the region $\RR$ to which this argument can be applied.
This other solution will be constructed as a minimal blowup solution.

\begin{figure}
\noindent
\begin{center}
\fbox{
\setlength{\unitlength}{1mm}
\begin{picture}(95,37)(-7,-7)
\put(0,-5){\vector(0,1){32}}\put(-5,25){$E$}
\put(-5,0){\vector(1,0){87}}\put(82,-3){$M$}
\put(65,0){\line(0,-1){2}}\put(65,-5){\hbox to 0mm{\hss$M(Q_1)$\hss}}
\put(30,0){\line(0,-1){2}}\put(30,-5){\hbox to 0mm{\hss$\tfrac{4}{3\sqrt{3}}M(Q_1)$\hss}}
\qbezier(40,10)(60,5)(65,0)
\qbezier(30,15)(30,10)(40,10)
\put(45,15){\vector(-1,-1){4}}\put(45.5,14){kink}
\put(15.7,7){$\RR$}
\linethickness{0.05mm}
\qbezier(03,0)(16,13)(30,27)\qbezier(06,0)(18,12)(30,24)
\qbezier(09,0)(13,4)(15.6,6.6)\qbezier(18.7,9.7)(25,16)(30,21)
\qbezier(12,0)(21,09)(30,18)\qbezier(15,0)(22,07)(30,15)\qbezier(18,0)(23,05)(30.7,12.7)\qbezier(21,0)(26,05)(32.3,11.3)
\qbezier(24,0)(29,05)(34.5,10.5)\qbezier(27,0)(32,05)(37.1,10.1)\qbezier(30,0)(35,05)(40,10)\qbezier(33,0)(38,05)(42.3,9.3)
\qbezier(36,0)(41,05)(44.7,8.7)\qbezier(39,0)(44,05)(47.1,8.1)\qbezier(42,0)(45,03)(49.4,7.4)\qbezier(45,0)(49,04)(51.7,6.7)
\qbezier(48,0)(51,03)(53.9,5.9)\qbezier(51,0)(54,03)(56.1,5.1)\qbezier(54,0)(56,02)(58.2,4.2)\qbezier(57,0)(59,02)(60.2,3.2)
\qbezier(60,0)(61,1)(62.2,2.2)\qbezier(63,0)(63.5,0.5)(63.9,0.9)
\qbezier(0,00)(15,15)(27,27)\qbezier(0,03)(12,15)(24,27)\qbezier(0,06)(10,16)(21,27)
\qbezier(0,09)(9,18)(18,27)\qbezier(0,12)(07,19)(15,27)\qbezier(0,15)(06,21)(12,27)\qbezier(0,18)(04,22)(09,27)\qbezier(0,21)(03,24)(06,27)
\qbezier(0,24)(01,25)(03,27)\qbezier( 0, 3)( 9,12)(17,20)\qbezier( 0, 6)( 7,13)(14,20)\qbezier( 0, 9)( 6,15)(11,20)
\end{picture}
}
\end{center}
\caption{Schematic depiction of the open set $\RR\subset\R^2$, based on numerics.  Note: the scale of the various features has been drastically altered in order to make them all visible on one plot.}\label{F:R}
\end{figure}
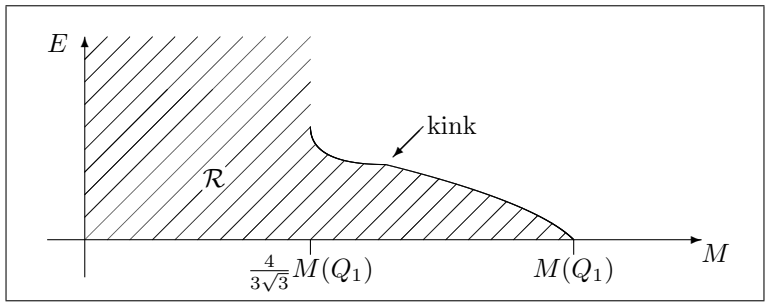

It will take considerable effort to shoehorn our problem into an application of~\eqref{dynam V}.  However, we have little choice.
This is not a small-data problem; the nonlinear effects are as strong as the linear effects.  Indeed, there are soliton
solutions on the boundary of the region $\RR$; see Theorem~\ref{T:Rbdry}.  As a consequence, we must exploit some fundamentally nonlinear information.  The virial
identity is the only truly pertinent tool that we know.

With this motivating discussion complete, we now turn to the analysis of the region $\RR$ in the mass/energy plane where the virial is strictly positive,
beginning with some definitions.

\begin{definition}\label{D:RR defn}
For $m>0$, we define
\begin{equation}\label{EV defn}
\Et(m):=\inf\bigl\{ E(u) : u\in H^1(\R^3),\ M(u)=m, \text{ and } V(u)=0\bigr\},
\end{equation}
with the understanding that $\Et(m)=\infty$ if no function $u\in H^1(\R^3)$ obeys both $M(u)=m$ and $V(u)=0$.
We then define
\begin{equation}\label{R defn}
\RR:=\bigl\{ (m,e) : 0 < m < M(Q_1) \text{ and } 0 < e < \Et(m) \bigr\},
\end{equation}
where $M(Q_1)$ is the common mass of all soliton optimizers of the $(\alpha=1)$-Gagliardo--Nirenberg--H\"older inequality, as in Section~\ref{SEC:ME}.
\end{definition}

This definition does not make it immediately apparent that functions with mass/energy belonging to the region $\RR$ have positive virial; this will be shown
in Theorem~\ref{T:MMR}.

Theorem~\ref{T:feasible} shows that that the restrictions $m<M(Q_1)$ and $e>0$ appearing in \eqref{R defn} are purely for expository clarity;
the set of $u\in H^1_x$ with $(M(u),E(u))\in\RR$ would be the same if they were omitted.  Indeed, if $m\geq M(Q_1)$, then the minimal feasible energy
with mass $m$ is achieved by a soliton, which has zero virial.  Theorem~\ref{T:feasible} also shows that $E(u)>0$ whenever $0<M(u)< M(Q_1)$.

The general features of $\RR$ are shown in Figure~\ref{F:R}.  In Subsection \ref{SS: EV}, we study the variational problem
\eqref{EV defn} and verify some of the features depicted in Figure~\ref{F:R}.
Subsection~\ref{SS: exhaustion of R} is devoted to constructing an exhaustion of $\RR$
that will allow us to prove Theorem~\ref{thm: main} by performing induction on a \emph{single} variable.

\subsection{Description of the region $\RR$}\label{SS: EV}
Some of the basic structural properties of the region $\RR$ are rigorously justified in Theorem~\ref{T:MMR} below.
Later in this subsection, we will discuss the boundary of $\RR$ more thoroughly; see Theorem~\ref{T:Rbdry}.

\begin{theorem}\label{T:MMR}
If $(M(u),E(u))\in\RR$ for some $u\in H^1(\R^3)$, then $V(u)>0$.
\begin{alignat}{3}
\Et(m)&=\infty & \qquad&\text{when}\qquad & 0&<m<\tfrac{4}{3\sqrt{3}}M(Q_1), \label{E:MMR1}\\
0<\Et(m)&<\infty & \qquad&\text{when}\qquad & \tfrac{4}{3\sqrt{3}}M(Q_1)&\leq m<M(Q_1), \quad \text{and} \label{E:MMR2}\\
\Et(m)&=\Em(m) & \qquad&\text{when}\qquad & M(Q_1)&\leq m. \label{E:MMR3}
\end{alignat}

For $m\geq \tfrac{4}{3\sqrt{3}}M(Q_1)$, the infimum $\Et(m)$ is achieved and is both strictly decreasing and lower semicontinuous as a function of $m$.
\end{theorem}

Recall that $\Em(m)$ denotes the infimal energy that is possible for $H^1$ functions with mass $m$.  From Theorem~\ref{T:feasible} we know that $\Em(m)=0$
for $0\leq m\leq M(Q_1)$ and $\Em(m)<0$ for $m>  M(Q_1)$.

Before beginning the proof of Theorem~\ref{T:MMR}, we first give two lemmas.  Both are based on straight-forward, but rather cumbersome, computations
of the effect of rescaling on the mass, energy, and virial of a function.

\begin{lemma}\label{L:energy+virial rescaling}
Suppose $u\in H^1(\R^3)$ is not identically zero and that either\\
\hbox to 1.5em{\hss\textup{(i)}\hss} $V(u)<0;$ or\\
\hbox to 1.5em{\hss\textup{(ii)}\hss} $V(u)=0$ and~$\beta(u) <\frac{1}{3}$.\\
Then there exists $\ld>1$ so that $u^{\ld}(x):=\ld^{\frac32}u(\ld x)$ obeys $V(u^{\ld})=0$, $\beta(u^{\ld})\geq\frac13$, and $E(u^{\ld})<E(u)$.
Note that $M(u^{\ld})=M(u)$ and $\|\nabla u^\lambda\|_{L^2}^2 =  \lambda^2 \|\nabla u\|_{L^2}^2 > \|\nabla u\|_{L^2}^2 $.
\end{lemma}

\begin{proof}
Direct computation shows that
\begin{align}\label{eq:deriv E(u^ld)}
V(u^{\ld}) = \ld \frac{d\ }{d\ld}E(u^{\ld}) = \ld^2 \!\!\int |\nabla u|^2\,dx - \tfrac{3}{4}\ld^3 \!\! \int |u|^4\,dx + \ld^6 \!\!\int |u|^6\,dx.
\end{align}

Suppose $V(u)<0$.  From \eqref{eq:deriv E(u^ld)} we see that $\lim_{\ld\to\infty}V(u^{\ld})=\infty$.  Thus, there is at least one $\ld>1$ so that $V(u^{\ld})=0$.
Let $\ld_0$ denote the smallest such $\ld$, which allows us to infer that $V(u^{\ld})<0$ for all $\ld\in [1,\ld_0)$.  It then follows from \eqref{eq:deriv E(u^ld)}
that $E(u^{\ld})$ is decreasing on the interval $[1,\ld_0)$, and so $E(u^{\ld_0})<E(u)$.

By construction, $\partial_\lambda V(u^\lambda)\geq 0$ at $\lambda=\lambda_0$.  Combining this with $V(u^{\lambda_0})=0$ yields
$$
0\leq 2\lambda_0 \int |\nabla u|^2 - \tfrac94 \lambda_0^2 \int |u|^4 + 6 \lambda_0^5 \int |u|^6 = - \lambda_0 \int |\nabla u|^2 + 3\lambda_0^5 \int |u|^6.
$$
Thus $\beta(u^{\lambda_0})=\lambda_0^4\beta(u)\geq \frac13$.  This proves part~(i) of the lemma.

We turn now to part~(ii) and so suppose that $V(u)=0$ and $\beta(u)<1/3$.  Then $\frac{3}{4}\!\int |u|^4\,dx=[1+\beta(u)] \int |\nabla u|^2\,dx$ and so \eqref{eq:deriv E(u^ld)} can be rewritten as
\begin{align}\label{eq:deriv E(u^ld) formula}
V(u^{\ld}) = \ld \frac{d\ }{d\ld}E(u^{\ld}) = \ld^2(1-\ld)\bigl(1-\beta(u) \ld-\beta(u)\ld^2-\beta(u)\ld^3\bigr)\!\int\! |\nabla u|^2\,dx.\! \!
\end{align}

As $\beta(u)<1/3$, RHS\eqref{eq:deriv E(u^ld) formula} has negative derivative at $\lambda=1$.  Thus for $\lambda>1$ sufficiently close to $1$, we have
$V(u^{\lambda}) < 0$ and $E(u^{\lambda}) < E(u)$.  Thus part~(ii) now follows by applying part~(i).
\end{proof}

\begin{lemma}\label{L:rescaling2}
Given $m>0$ and  $u\in H^1(\R^3)$ with $V(u)=0$ and $0<M(u)<m$, there exists $v\in H^1(\R^3)$ obeying
\begin{equation}\label{E:rescaling2}
M(v)=m, \quad E(v)\leq E(u) - \tfrac{m-M(u)}{6 M(u)}\!\int |\nabla u(x)|^2\,dx, \qtq{and} V(v)=0.
\end{equation}
\end{lemma}

\begin{proof}
We will prove the lemma under the additional assumption that $\beta(u)\geq\tfrac13$.  This suffices because any $u$ not obeying this hypothesis can be replaced
by the function $u^\lambda$ provided by Lemma~\ref{L:energy+virial rescaling}(ii).  The relations \eqref{E:rescaling2} remain
true with the original function $u$ because $E(u^\lambda)<E(u)$ and $\|\nabla u^\lambda\|_{L^2} > \|\nabla u\|_{L^2} $.

Little effort is required to check that for $\beta(u)\geq\tfrac13$ and $1<\rho<\infty$,
\begin{equation*}
0 <  \frac{d\ }{d\rho} \; \frac{(1+\rho^2\beta(u))^2}{\rho(1+\beta(u))^2} = \frac{(1+\rho^2\beta(u))(3\beta(u)\rho^2-1)}{\rho^2(1+\beta(u))^2} \leq 3\beta(u)\rho^2-1.
\end{equation*}
Using this, we see that there is a unique $1<\rho<\infty$ so that
\begin{equation}\label{E:rho M}
m = M(u) \frac{(1+\rho^2\beta(u))^2}{\rho(1+\beta(u))^2}; \qtq{moreover} \frac{m-M(u)}{M(u)} \leq \beta(u)(\rho^3-1) - (\rho-1).
\end{equation}
Indeed, existence of such a $\rho$ follows from the intermediate value theorem (and the values at $\rho=1$ and as $\rho\to\infty$).  Uniqueness follows from the positivity of the derivative, while the second part of \eqref{E:rho M} follows from the upper bound on the derivative and the fundamental theorem of calculus.

Using this value of $\rho$ we then define
$$
v(x) := \sqrt{\rho\lambda}\,u(\lambda x) \qtq{with} \lambda=\frac{\rho(1+\beta(u))}{1+\rho^2\beta(u)}.
$$

Elementary computations show that $M(v)=m$ and $V(v)=0$; indeed, this is precisely how the parameters $\rho$ and $\lambda$ were chosen.  To verify $V(v)=0$,
we exploit the fact that $V(u)=0$, which implies $\tfrac34\int |u|^4\,dx = [1+\beta(u)]\int |\nabla u|^2\,dx$.

Similar computations give
\begin{align*}
E(v) &= E(u) - \tfrac1{6}\bigl[\beta(u)(\rho^3-1)-(\rho-1)\bigr]\int |\nabla u(x)|^2\,dx
\end{align*}
The energy bound stated in \eqref{E:rescaling2} now follows by combining this estimate with the second part of \eqref{E:rho M}.
\end{proof}

\begin{proof}[Proof of Theorem~\ref{T:MMR}]
Suppose $u\in H^1(\R^3)$ and $(M(u),E(u))\in\RR$.  We wish to show that $V(u)>0$.  That $V(u)\neq 0$ follows from the definition of $\RR$.  If, on the other hand,
$V(u)<0$, then by Lemma~\ref{L:energy+virial rescaling}(i) there is a rescaling $u^\lambda$ of $u$ with the same mass, $E(u^\lambda)<E(u)$ and $V(u^\lambda)=0$.
This implies $\Et(M(u))<E(u)$, so contradicting the original assumption $(M(u),E(u))\in\RR$.

Next we show that $\Et(m)=\infty$ when $0<m<\tfrac{4}{3\sqrt{3}}M(Q_1)$.  To do this, we define
$$
R(x) := \tfrac1{\sqrt{2}} Q_1\bigl(\tfrac{\sqrt{3}}{2}\, x\bigr).
$$
Later we will introduce a class of rescaled solitons that include this example; see Lemma~\ref{L:Rw}.  Direct computation (cf. the proof of Lemma~\ref{L:Rw}) shows that
\begin{equation}\label{E:R relations}
\beta(R)=\tfrac13, \quad V(R)=0, \qtq{and} M(R) = \tfrac{4}{3\sqrt{3}}M(Q_1).
\end{equation}

As $Q_1$ is an optimizer for the ($\alpha=1$)-Gagliardo--Nirenberg--H\"older inequality, so too is $R$.  This allows us to express the optimal constant as
$$
C_1 = \frac{\|R\|_{L^4}^4}{\|R\|_{L^2}\|R\|_{L^6}^{3/2}\|\nabla R\|_{L^2}^{3/2}} = \frac{16\cdot 3^{1/4}}{9\|R\|_{L^2}},
$$
where the second equality follows by exploiting \eqref{E:R relations}.  Thus
$$
\int |u|^4\,dx \leq \frac{16\|u\|_{L^2}}{9\|R\|_{L^2}} \biggl(3\!\!\int |u|^6\,dx\biggr)^{\!\!\frac14} \biggl(\int |\nabla u|^2\,dx\biggr)^{\!\!\frac34} \qquad\text{for all $u\in H^1(\R^3)$,}
$$
and so Young's inequality (with powers $4$ and $\frac43$) yields
\begin{equation}\label{E:V vs R}
\tfrac34\!\int |u|^4\,dx \leq \frac{\|u\|_{L^2}}{\|R\|_{L^2}} \int |\nabla u|^2 + |u|^6\,dx  \qquad\text{for all $u\in H^1(\R^3)$.}
\end{equation}
This inequality immediately shows that $V(u)>0$ whenever $0<M(u)<M(R)$.

We have now completed the discussion of $m<M(R)$.  When $m\geq M(Q_1)$ all assertions follow from Theorem~\ref{T:feasible}.  This theorem shows that for such
$m$, the minimal energy under the constraint $M(u)=m$ is achieved by a soliton, and hence by a function with zero virial.  Thus $\Et(m)=\Em(m)$.  Moreover, the proof of this
theorem shows that $\Em(m)$ is continuous and strictly decreasing for $m\geq M(Q_1)$.

We are now left to consider the middle interval $M(R)\leq m<M(Q_1)$.

Applying Lemma~\ref{L:rescaling2} with $u=R$ shows that $\Et(m)\leq E(R)<\infty$ for all such~$m$.  In particular, it is possible to achieve $V(u)=0$ with $M(u)=m$.

Using the classical Gagliardo--Nirenberg inequality from \eqref{E:GN&H}, we observe that
$$
V(u)=0 \ \implies\  \|\nabla u\|_{L^2}^2 \leq \tfrac34 \|u\|_{L^4}^4 \lesssim \|u\|_{L^2} \|\nabla u\|_{L^2}^3
\ \implies\  \|u\|_{L^2} \|\nabla u\|_{L^2} \gtrsim 1,
$$
whenever $u\not\equiv 0$.  Combining this with \eqref{E:Ea1} we deduce
$$
E(u) \gtrsim \Bigl[ 1 - \bigl(\tfrac{M(u)}{M(Q_1)}\bigr)^{1/2} \Bigr] \tfrac{1}{M(u)} \qtq{whenever}   V(u)=0 \qtq{and} M(u) < M(Q_1).
$$
This proves that $\Et(m)>0$ when $M(R)\leq m<M(Q_1)$.

Next we show that $\Et(m)$ is strictly decreasing when $M(R)\leq m< M(Q_1)$.  Given $M(Q_1)>m_1>m_2\geq M(R)$, let $\{u_n\}$ be a sequence with $M(u_n)=m_2$, $V(u_n)=0$, and $E(u_n)\to \Et(m_2)$.
As shown above, $V(u_n)=0$ guarantees that $\|\nabla u_n\|_{L^2}^2 \geq c/m_2$, for an absolute constant $c>0$.  Thus by Lemma~\ref{L:rescaling2}, there is a corresponding
sequence $\{v_n\}$ with $M(v_n)=m_1$, $V(v_n)=0$, and $E(v_n)\leq E(u_n) - c\tfrac{m_1-m_2}{6m_2^2}$.  Sending $n\to\infty$, we deduce that $\Et(m_1)<\Et(m_2)$.

Next we show that the infimum is actually achieved for $M(R)\leq m < M(Q_1)$.  To this end, let $u_n$ be a minimizing sequence at mass $m$.  Replacing $u_n$ by their symmetric decreasing rearrangement does
not affect the mass constraint, does not increase the energy, but may result in $V(u_n)<0$.  However, by applying Lemma~\ref{L:energy+virial rescaling} we may still
conclude that there is a minimizing sequence $\{u_n\}$ comprising radially symmetric functions.

By \eqref{E3/32}, our optimizing sequence $\{u_n\}$ is bounded in $H^1(\R^3)$.  Passing to a subsequence (if necessary) we may ensure weak convergence
in $H^1(\R^3)$ and $L^6(\R^3)$, as well as norm convergence in $L^4(\R^3)$, due to radial symmetry.   Let us denote the limit by $u_\infty$.

From these modes of convergence it follows that $M(u_\infty)\leq m$, $E(u_\infty)\leq \Et(m)$, and $V(u_\infty)\leq 0$.  We will show that equality holds in all three cases;
consequently, $u_\infty$ is the sought-after optimizer.

If $M(u_\infty)<m$, we could conclude that $\Et(M(u_\infty))\leq E(u_\infty) \leq \Et(m)$, which is inconsistent with the fact that $\Et$ is a strictly decreasing function, as proved above.
(Here we also invoke Lemma~\ref{L:energy+virial rescaling} if it happens that $V(u_\infty) < 0$.)  Thus $M(u_\infty)=m$.  It follows then from
Lemma~\ref{L:energy+virial rescaling} and the definition of $\Et$ as an infimium that $V(u_\infty)= 0$ and $E(u_\infty)=\Et(m)$.

The only assertion of Theorem~\ref{T:MMR} that remains unproven is that of lower semicontinuity.  This will follow readily from the ideas already presented.

As $\Et(m)$ is decreasing, lower semicontinuity is equivalent to right continuity.  In this way, we see that failure of lower semicontinuity at a point $m$ would guarantee the
existence of a sequence $\{m_n\}$ decreasing to $m$ with $\Et(m) > \lim \Et(m_n)$.  Now let $u_n$ be radially symmetric
functions with $M(u_n)=m_n$, $V(u_n)=0$, and $E(u_n)=\Et(m_n)$.  (We just proved the existence of such optimizers above.)  Passing to a subsequence we may assume that $u_n\to u_\infty$ weakly in $H^1(\R^3)$ and strongly
in $L^4(\R^3)$.  Using the lower semicontinuity of $M$, $E$, and $V$ under such convergence together with the strict monotonicity of $\Et(m)$, we then deduce that $M(u_\infty)=m$,
$V(u_\infty)=0$, and $E(u_\infty) = \lim \Et(m_n)$,
just as in the proof of the existence of optimizers.  This then falsifies the assertion that $\Et(m) > \lim \Et(m_n)$.
\end{proof}

The appearance of $R(x)$, a particular rescaling of $Q_1(x)$, in the proof of Theorem~\ref{T:MMR} suggests that we will need to consider functions other than
solitons if we wish to understand the boundary of $\RR$.  This is indeed the case, as will be borne out by Theorem~\ref{T:Rbdry}.  Indeed, it will be shown
that an optimizer in the variational description of $\Et$ is either a soliton or a special type of rescaled soliton; moreover both cases do occur.
The statement and proof of Theorem~\ref{T:Rbdry} are simpler if we present the basic properties of these special rescaled solitons first.  This is the purpose of the next lemma.

\begin{lemma}[Rescaled solitons]\label{L:Rw}
Fix $0<\omega<\tfrac3{16}$.  Among all rescalings $u(x)=a\Pw(\lambda x)$ of $\Pw$ with $a>0$ and $\lambda >0$ there is exactly one that obeys $\beta(u)=\tfrac13$ and
$V(u)=0$, namely,
\begin{equation}\label{E:Rw defn}
\Rw(x) := \sqrt{\tfrac{1+\beta(\omega)}{4\beta(\omega)}} \, \Pw\Bigl(\tfrac{3[1+\beta(\omega)]}{4\sqrt{3\smash[b]{\beta(\omega)}}} \, x\Bigr) .
\end{equation}
Moreover,
\begin{gather}
E(\Rw) = \frac{1}{9\sqrt{3\smash[b]{\beta(\omega)}}}\int |\nabla\Pw(x)|^2\,dx,
	\quad M(\Rw) = \frac{16\sqrt{3\smash[b]{\beta(\omega)}}}{9[1+\beta(\omega)]^2} M(\Pw), \label{ER:me}\\[0.5ex]
\text{ $M(\Rw)\leq M(\Pw)$ with equality iff $\beta(\omega)=\tfrac13$}, \label{ER:m}\\
\text{ $E(\Rw)\geq E(\Pw)$ with equality iff $\beta(\omega)=\tfrac13$}, \label{ER:e}\\
\tfrac{d\ }{d\omega} E(\Rw) >0, \qquad [\beta(\omega)-1] \tfrac{d\ }{d\omega} M(\Rw) \geq 0 \ \text{with equality iff $\beta(\omega)=1$,} \label{ER:dE>0} \\
\lim_{\omega\searrow 0} E(\Rw) = \tfrac{1}{3\sqrt{3\beta(g)}} M(g), \qtq{and} \lim_{\omega\searrow 0} M(\Rw) = \tfrac{16 \sqrt{3\beta(g)}}{9} M(g), \label{ER:w0}
\end{gather}
where $g$ is the unique positive radial solution to $-\Delta g - g^3 + g =0$, as in Theorem~\ref{T:solitons}.
\end{lemma}

\begin{proof}
Recall that $\beta(u)$ was defined in \eqref{eq:Zbeta} and that $\beta(\omega):=\beta(\Pw)$.  After a few manipulations, this reveals
that $u(x)=a\Pw(\lambda x)$ obeys
\begin{equation}\label{vR0}
\beta(u) = \tfrac{a^4}{\lambda^2} \beta(\omega).
\end{equation}
Proceeding similarly, but also exploiting \eqref{eq:w24} and \eqref{eq:w6E}, yields
\begin{equation}\label{ER:V}
V(u) = \tfrac{a^2}{\lambda} \bigl\{ 1 + \tfrac{a^4}{\lambda^2} \beta(\omega) - \tfrac{a^2}{\lambda^2} [\beta(\omega)+1]\bigr\} \int |\nabla\Pw(x)|^2\,dx.
\end{equation}
It is now an elementary matter to verify that \eqref{E:Rw defn} gives the only solution to the system $\beta(u)=\frac13$ and $V(u)=0$.

The formula for $M(\Rw)$ given in \eqref{ER:me} follows directly from \eqref{E:Rw defn}, while the formula for $E(\Rw)$ is most easily found by mimicking \eqref{ER:V}.

Equality in \eqref{ER:m} and \eqref{ER:e} when $\beta(\omega)=\frac13$ is immediate as $\Rw\equiv\Pw$ when $\beta(\omega)=\frac13$.  That strict inequality holds in all other cases then follows from
$$
\frac{d\ }{d\beta} \frac{16\sqrt{3\smash[b]{\beta}}}{9(1+\beta)^2} = \frac{8(1-3\beta)}{3\sqrt{3\smash[b]{\beta}}(1+\beta)^3}
    \qtq{and} \frac{d\ }{d\beta} \frac{2}{3(1-\beta)\sqrt{3\smash[b]{\beta}}} = \frac{(3\beta-1)}{(3\beta)^{3/2}(\beta-1)^2}.
$$
Note that there is no need to use the second formula when $\beta\geq 1$, for if $\beta(\omega)\geq 1$, then \eqref{eq:w6E} shows that $E(\Pw)\leq 0$,
while \eqref{ER:me} shows that $E(\Rw)>0$.

It remains only to verify \eqref{ER:dE>0} and \eqref{ER:w0}.  Differentiating the first relation in \eqref{eq:w24} with respect to $\omega$ and then using \eqref{dGdw} gives
\begin{equation}\label{ER:db}
  \tfrac13 \bigl[\tfrac{d\;}{d\omega} \beta(\omega)\bigr] \!\int |\nabla \Pw|^2\,dx = - \tfrac{\beta-1}2 M(\Pw) + \omega\tfrac{d\;}{d\omega} M(\Pw).
\end{equation}
From this, \eqref{ER:me}, \eqref{dGdw}, and finally \eqref{dMdw}, we then deduce that
$$
\tfrac{d\;}{d\omega} E(\Rw) = \tfrac{1}{6\beta\sqrt{3\smash[b]{\beta}}}\Bigl[\tfrac{3\beta-1}2 M(\Pw) - \omega \tfrac{d\;}{d\omega} M(\Pw) \Bigr] > 0.
$$

Using \eqref{eq:w24} we may rewrite \eqref{ER:db} as follows:
$$
M(\Pw) \tfrac{d\;}{d\omega} \beta(\omega) = \tfrac{1+\beta(\omega)}{\omega} \Bigl[- \tfrac{\beta-1}2 M(\Pw) + \omega\tfrac{d\;}{d\omega} M(\Pw)\Bigr].
$$
Proceeding as above, we deduce that
\begin{align*}
\tfrac{d\;}{d\omega} M(\Rw) &= \tfrac{16\sqrt{3\smash[b]{\beta}}}{9(1+\beta)^2} \tfrac{d\;}{d\omega} M(\Pw)
	+ \tfrac{8(1-3\beta)}{3\sqrt{3\smash[b]{\beta}}(1+\beta)^2} \Bigl[- \tfrac{\beta-1}{2\omega} M(\Pw) + \tfrac{d\;}{d\omega} M(\Pw)\Bigr] \\
&= \tfrac{8(1-\beta)}{3\sqrt{3\smash[b]{\beta}}(1+\beta)^2} \Bigl[- \tfrac{3\beta-1}{2\omega} M(\Pw) + \tfrac{d\;}{d\omega} M(\Pw)\Bigr].
\end{align*}
That $[\beta(\omega)-1] \tfrac{d\ }{d\omega} M(\Rw) \geq 0$, as well as the cases of equality, now follow from \eqref{dMdw}, which shows that the final expression in square brackets is negative.

As we have now shown that $E(\Rw)$ is increasing in $\omega$, the limiting value as $\omega\to0$ actually provides the infimum of the energies $E(\Rw)$.
The stated value for this and for the limiting mass follow from \eqref{ER:me}, \eqref{eq:w24}, and Theorem~\ref{T:solitons}(iv).
\end{proof}

\begin{theorem}\label{T:Rbdry}
Suppose $m\geq \tfrac{4}{3\sqrt{3}}M(Q_1)$ and $u\in H^1_x\setminus\{0\}$ obeys
\begin{equation}\label{E:Rbdry}
M(u)=m, \quad E(u)=\Et(m), \qtq{and} V(u)=0.
\end{equation}
\textup{(}Theorem~\ref{T:MMR} guarantees that such a $u$ exists precisely for these values of $m$.\textup{)}
Then either $u(x)=e^{i\theta}\Rw(x+x_0)$ or $u(x)=e^{i\theta}\Pw(x+x_0)$ for some $\theta\in[0,2\pi)$, some $x_0\in \R^3$, and some $0<\omega<\tfrac3{16}$
that obeys $\beta(\omega)>\frac13$.  Recall from Theorem~\ref{T:MMR} that for all such $m$ there is a function $u$ obeying \eqref{E:Rbdry}.

Furthermore, both cases necessarily occur on the boundary of $\RR$.  Specifically, there exists $\delta>0$ so that
no soliton has mass $<\smash[b]{\tfrac{4}{3\sqrt{3}}}M(Q_1) + \delta$, while $u$ cannot be a rescaled soliton when $m> M(Q_1) - \delta$.
\end{theorem}

Before we begin the proof of this theorem, we discuss the fuller picture provided by numerics.  See Figures~\ref{F:numerics1} and~\ref{F:numerics2}.  We find that there is an intermediate mass
$$
\frac{4}{3\sqrt{3}} M(Q_1) < m_0 < M(Q_1)
$$
so that $\Et(m)$ is achieved solely by rescaled solitons $\Rw$ when $m<m_0$, and solely by solitons $\Pw$ when $m>m_0$.  (Numerics give $m_0=189.48$.)
The transition in the type of minimizer is marked by a discontinuity in the derivative of $\Et$ and so is depicted as a kink in Figure~\ref{F:R}.
Note that this kink is essentially invisible in Figure~\ref{F:numerics1}; this is why we have included Figure~\ref{F:numerics2}.

The mass/energy curve of $\Rw$ crosses the mass/energy curve of $\Pw$ whenever $\beta(\omega)=1/3$; indeed, $\Rw=\Pw$ whenever $\beta(\omega)=1/3$.
Our numerics show that the mass/energy curves intersect exactly twice, once at $m_0$ and once at the (unique) point where $\beta(\omega)=1/3$. The intersection at $m_0$ is transverse.  At the point where $\beta(\omega)=1/3$, the two curves cross; however, the formulas in Lemma~\ref{L:Rw} can be used to show that both have the same tangent and curvature.

Excepting the point $m_0$, we observe $\Et(m)$ to be a smooth curve.  Let us briefly describe what is needed to achieve a rigorous proof of this.  It follows from Theorem~\ref{T:solitons} that both
$\Pw$ and $\Rw$ are analytic $H^1(\R^3)$-valued functions of $\omega$; thus $\Et(m)$ is smooth except where $\frac{d\;}{d\omega} M(\Pw)$ or $\frac{d\;}{d\omega} M(\Rw)$ vanish or where the two mass/energy curves touch.

From the proof of Lemma~\ref{L:Rw} and Conjecture~\ref{Conj:beta}, it would follow that $\frac{d\;}{d\omega} M(\Rw)$ vanishes only once,
namely, when $\Pw=Q_1$, which corresponds to $m=\frac{4}{3\sqrt{3}} M(Q_1)=185.10$.  Recall that Conjecture~\ref{Conj:beta} implies uniqueness of the optimizer $Q_1$.

On the other hand Conjecture~\ref{Conj:mass} implies that $\frac{d\;}{d\omega} M(\Pw)$ vanishes only once, namely, at the minimal mass soliton.  This corresponds to the cusp in Figures~\ref{F:Pw}, \ref{F:numerics1} and~\ref{F:numerics2}.  In this way, we see that the claimed shape of $\Et(m)$ follows from our two conjectures along with the following
assertion, which is supported by numerics: The mass/energy curve of $\Rw$ intersects the \emph{lower} branch of solitons at exactly one point (which then serves to define $m_0$)
and there is no intersection with the upper branch of solitons with $m<m_0$.

\begin{figure}
\noindent\hbox to \textwidth{\hss\hspace*{-5mm}\includegraphics[scale=0.2]{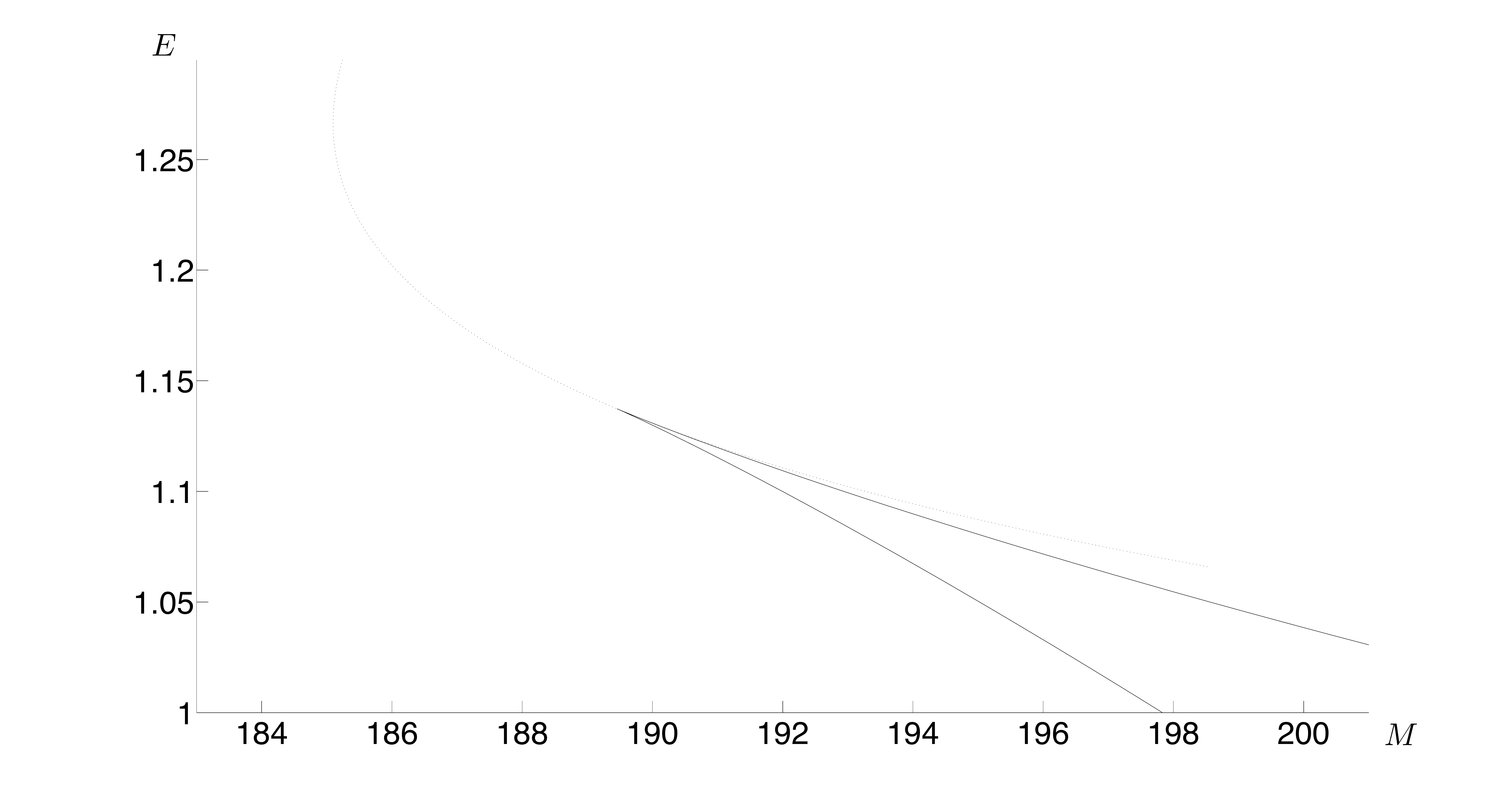}\hss}
\vspace*{-10mm}
\caption{Detail of the mass/energy curves of $\Pw$, shown solid, and $\Rw$, shown dotted.}
\label{F:numerics1}
\end{figure}

\begin{figure}
\noindent\hbox to \textwidth{\hss\hspace*{-5mm}\includegraphics[scale=0.2]{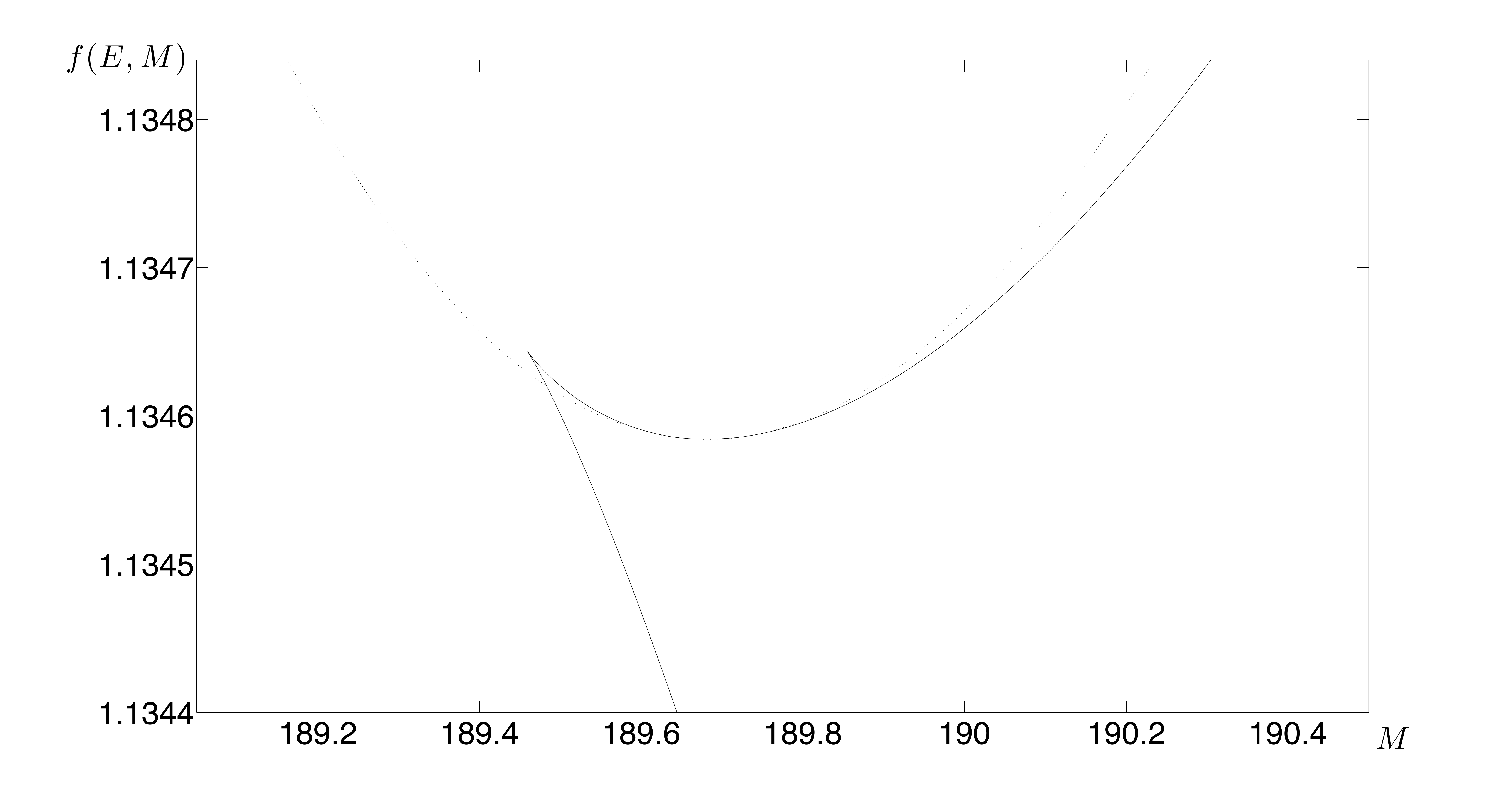}\hss}
\vspace*{-10mm}
\caption{Closer view of the mass/energy curves of $\Pw$, shown solid, and $\Rw$, shown dotted.
The horizontal axis denotes mass.  In order to reveal the fine detail of the intersection,
a shear transformation has been applied, namely, $f(E,M)=E+\tfrac12\omega_{1/3}(M-M(P_{\omega_{1/3}}))$, where $\omega_{1/3}$ is determined numerically so that
$\beta(\omega_{1/3})=1/3$ (cf. Table~\ref{Table:num}).}\label{F:numerics2}
\end{figure}

\begin{proof}[Proof of Theorem~\ref{T:Rbdry}]
If $u$ obeys \eqref{E:Rbdry}, then it follows that the radially symmetric rearrangement $u^*$ of $u$ also obeys \eqref{E:Rbdry}.  To see this, we first observe that $M(u^*)=M(u)$, $E(u^*)\leq E(u)$ and $V(u^*)\leq V(u)=0$.
By Lemma~\ref{L:energy+virial rescaling}, we see that $V(u^*)<0$ would be inconsistent with the definition of $\Et(m)$ as an infimum.  Similarly, we must have $E(u^*)= E(u)$.

Below we will show that non-negative radially symmetric functions obeying \eqref{E:Rbdry} must agree with some $\Pw$ or $\Rw$.  In particular, this can be applied to $u^*$.  As all level sets of $\Pw$ and $\Rw$ have zero (volume) measure, combining the strict rearrangement inequalities of \cite{BrothersZiemer} with the arguments in the previous paragraph, we may conclude that any $u\in H^1(\R^3)$ obeying \eqref{E:Rbdry} must agree with some $\Pw$ or some $\Rw$ up to a translation and a phase rotation.

For the remainder of the proof we consider only $u$ that are non-negative and radially symmetric.

As $u$ is a minimizer of the energy with constrained virial and mass, $u$ must satisfy the corresponding Euler--Lagrange equation,  provided
$dV$ and $dM$ are linearly independent at the point $u$.  We first consider the possibility that they are linearly dependent.  As it turns
out, this is not an illusory scenario; a little thought shows that this must occur at $m=4 M(Q_1)/(3\sqrt{3})$, which is the minimal mass at which it is possible
to achieve zero virial.

As $u\not\equiv0$, so $dM|_u\neq 0$.  Thus linear dependence of $dV$ and $dM$ at $u$ guarantees
\begin{equation}\label{E:lindep}
-\Delta u + 3u^5 - \tfrac{3}{2} u^3 = cu,
\end{equation}
for some $c\in\R$.  Pairing this equation with $x\cdot\nabla u + \tfrac32 u$ yields
$$
\int |\nabla u|^2 + 3 u^6 - \tfrac98 u^4 \,dx=0, \qtq{while} V(u):= \int |\nabla u|^2 + u^6 - \tfrac34 u^4 \,dx = 0,
$$
by hypothesis.  Combining these immediately shows $\beta(u)=\tfrac13$.

Making the change of variables $u(x)=\frac{1}{\sqrt{2}}v(\frac{\sqrt{3}x}{2})$ in \eqref{E:lindep}, we find that $v$ is a non-negative radially symmetric solution to \eqref{E:Pdefn}.
Thus by Theorem~\ref{T:solitons}, it follows that $v(x)=\Pw(x)$ for some $\omega\in (0,\frac{3}{16})$.

As $u$ is a rescaling of a soliton $\Pw$ with $\beta(u)=\frac13$ and $V(u)=0$, it follows from Lemma~\ref{L:Rw} that $u=\Rw$, with the same value
of $\omega$.  From the rescaling used above we deduce from \eqref{E:Rw defn} that $\beta(\omega)=1$.  In particular, $\beta(\omega)>\frac13$.

We now turn to the case that $dV|_u$ and $dM|_u$ are linearly independent.
In this case, $u$ satisfies the Euler--Lagrange equation $dE|_u + \mu dV|_u + \nu dM|_u =0$, that is,
\begin{equation}\label{eq: E-L rewritten}
-\Delta u+u^5-u^3 + \mu (-2\Delta u+6u^5-3u^3) + 2\nu u =0.
\end{equation}

If the Lagrange multiplier $\mu=0$, then \eqref{eq: E-L rewritten} is precisely the defining equation for a soliton (cf. \eqref{E:Pdefn}) with $\omega=2\nu$.
As $u$ is non-negative and radially symmetric, Theorem~\ref{T:solitons} guarantees that $u=P_{\omega_0}$ for some $\omega_0\in (0,\frac{3}{16})$.
That $\beta(\omega_0)\geq\frac13$ follows from Lemma~\ref{L:energy+virial rescaling} and the fact that $u$ minimizes the energy subject to $M(u)=m$ and $V(u)=0$.

To show that actually $\beta(\omega_0)>\tfrac13$, we will use Lemma~\ref{L:cuspy}.  Suppose instead that $\beta(\omega_0)=\frac13$.
Combining \eqref{dMdw} and \eqref{E:dEdM} we see that $\omega\mapsto E(\Pw)$
is increasing at $\omega_0$; however, this function converges to $-\infty$ as $\omega\to3/16$.  Thus we may choose $\omega_1>\omega_0$ so that
$E(P_{\omega_1})=E(P_{\omega_0})$ and $E(\Pw)\geq E(P_{\omega_0})$ on the intervening interval.  By Lemma~\ref{L:cuspy} it then follows that
$M(P_{\omega_1})<M(P_{\omega_0})$.  This contradicts the minimality of $u=P_{\omega_0}$ because $\Et$ is strictly decreasing, as shown in Theorem~\ref{T:MMR}.

This leaves the case $\mu\neq 0$.  We first show that $\mu>0$, arguing by contradiction.

As $dV$ and $dM$ are linearly independent at $u$, there exists $w\in H^1(\R^3)$ so that $dV|_u(w)<0$ and $dM|_u(w)=0$.  Without loss of generality,
we may assume that $M(w)=M(u)$.  If $\mu<0$ then \eqref{eq: E-L rewritten} implies $dE|_u(w)<0$.  From these values of the differentials, we see that
$
u_\theta:=\cos(\theta) u + \sin(\theta) w
$
obeys $M(u_\theta) =M(u)$, $V(u_\theta) < 0$, and $E(u_\theta)<E(u)$ for small $\theta>0$.  In view of Lemma~\ref{L:energy+virial rescaling}, this
contradicts the minimality of $u$.  Thus, either $\mu=0$, which was treated above, or $\mu>0$, which we address next.

Under the change of variables $u(x)=av(\ld x)$ with
\begin{equation}\label{eq: a,ld}
a^2=\frac{1+3\mu}{1+6\mu}, \quad \ld^2=\frac{(1+3\mu)^2}{(1+2\mu)(1+6\mu)}, \qtq{and} \omega=\frac{2\nu (1+6\mu)}{(1+3\mu)^2},
\end{equation}
the Euler--Lagrange equation \eqref{eq: E-L rewritten} becomes $-\Delta v+v^5-v^3 + \omega v=0$.  Theorem~\ref{T:solitons} then allows us to conclude that
$v=\Pw$ for some $0<\omega<\frac3{16}$.  By Lemma~ref{L:Rw}, to see that $u=\Rw$ we need to show that $\beta(u)=\frac13$.  This follows by pairing \eqref{eq: E-L rewritten} with
$x\cdot\nabla u + \tfrac32 u$ and using $V(u)=0$, as in the treatment of \eqref{E:lindep}.  Lastly, a short computation shows
\begin{equation}\label{ZZZbeta}
\beta (\omega):= \beta(\Pw) =\beta(u) \frac{1+6\mu}{1+2\mu}>\frac{1}{3}.
\end{equation}

Finally, we turn to the assertions of the second paragraph in the theorem.

Choose $\omega_*\in(0,3/16)$ so that $P_{\omega_*}$ has the least mass possible among solitons (Theorem~\ref{T:solitons} guarantees that this is possible).
By \eqref{dMdw} we know that $\beta(\omega_*)>1/3$ and consequently $M(R_{\omega_*})< M(P_{\omega_*})$, by \eqref{ER:m}.  This shows that
$M(P_{\omega_*})>4 M(Q_1)/(3\sqrt{3})$ and consequently, there exists $\delta>0$ so that
no soliton has mass smaller than $4M(Q_1)/(3\sqrt{3}) + \delta$.  Thus minimizers in this regime must be rescaled solitons.

Next we exploit the fact that $\eta:=\inf\{E(\Rw):0<\omega<3/16\}$ is positive (see \eqref{ER:w0}) to show that the possibility
$u=\Pw$ does indeed occur.

Let $P_{\omega_1}$ denote an optimizer for \eqref{eq:GN} in the case $\alpha=1$.  Thus $M(P_{\omega_1})=M(Q_1)$ and $E(P_{\omega_1})=0$.  We claim that
there exists $\eps>0$ so that $M(\Pw)<M(P_{\omega_1})$ for $\omega_1-\eps<\omega<\omega_1$.  Taking this for granted for a moment, we see from the continuity of
$\omega\mapsto \Pw$ that $\Et(m)\to 0$ as $m\nearrow M(Q_1)$.  In particular, there is some $\delta>0$ so that $\Et(m)< \eta$ for $m> M(Q_1) - \delta$,
which guarantees that optimizers $u$ of $\Et(m)$ must be solitons (rather than rescaled solitons) on this interval.

It remains to verify the claim $M(\Pw)<M(P_{\omega_1})$ for $\omega_1-\eps<\omega<\omega_1$ and some $\eps>0$.  Note that since $\omega\mapsto M(\Pw)$ is
real analytic, failure of this claim would actually imply that $\frac{d\;}{d\omega} M(\Pw)<0$ on some interval of this type.  By \eqref{E:dEdM}, this in turn
implies that $E(\Pw)$ is increasing (and so negative) on this interval.  Recall that $E(\Pw)\to 0$ from above as $\omega\to 0$.  In this way, the assumption
that our claim fails leads to the existence of an interval $[\omega_0,\omega_1]$ so that $E(\Pw)$ vanishes at the endpoints and is non-positive on the interior.
This allows us to apply Lemma~\ref{L:cuspy} and deduce that $M(P_{\omega_0})<M(P_{\omega_1})=M(Q_1)$.  This contradicts Theorem~\ref{T:feasible}, which showed
that $M(Q_1)$ is the least mass at which zero energy is feasible.
\end{proof}

\subsection{Foliation of the region $\RR$}\label{SS: exhaustion of R}
As noted earlier, this subsection is devoted to the construction of a function $D:H^1(\R^3)\to[0,\infty]$
whose sublevel sets exhaust $\RR$.  There are many ways to do this; we choose one that is convenient for the subsequent analysis.
In particular, to ensure that it is conserved by the flow, we will choose $D(u)=D(M(u),E(u))$.
Note that property (v) below relies inherently on the shape of $\RR$, specifically, the monotonicity of the function $\Et(m)$
proved in Theorem~\ref{T:MMR}.  This property cannot be relaxed; it is precisely how our induction on $D(u)$ manifests a simultaneous
induction on the quantities $M(u)$ and $E(u)$ that are conserved by the flow.

\begin{proposition}[Key properties of $D(u)$]\label{P:all about D}
There is a continuous function $D:H^1(\R^3)\to[0,\infty]$ with the following properties:
\begin{SL}
\item $D(u)=0$ if and only if $u\equiv 0$.
\item $0<D(u)<\infty$ if and only if $(M(u),E(u))\in\RR$.
\item $D$ is conserved under the flow of the cubic-quintic NLS.
\item If $0<D(u)<\infty$, then $V(u)> 0$.
\item If $M(u_1)\leq M(u_2)$ and $E(u_1)\leq E(u_2)$, then $D(u_1)\leq D(u_2)$.
\item If $M(u_n)\leq M_0$, $E(u_n)\leq E_0$, and $D(u_n)\to D(M_0,E_0)$, then actually $M(u_n)\to M_0$ and $E(u_n)\to E_0$.
\item Given $0<D_0<\infty$, we have
\begin{equation}\label{E:norm equiv}
\|u\|_{\dot H^1_x}^2 \sim_{D_0} E(u) \qtq{and} \|u\|_{H^1_x}^2 \sim_{D_0} E(u) + M(u) \sim_{D_0} D(u)
\end{equation}
uniformly for all $u\in H^1(\R^3)$ with $D(u)\leq D_0$.
\end{SL}
\end{proposition}

\begin{proof}
As noted above, $D(u)$ will be chosen to depend only on $M(u)$ and $E(u)$, which justifies (iii).  Specifically, we define
\begin{equation}\label{eq: def RRc}
\RR^c := \{(m,e)\in\R^2 : m \geq \tfrac{4}{3\sqrt{3}}M(Q_1) \text{ and } e\geq \Et(m) \bigr\}.
\end{equation}
and then set
\begin{equation}\label{eq: def D}
D(u):=D\bigl(M(u),E(u)\bigr):=E(u)+\frac{M(u)+E(u)}{\dist\, \Big(\big(M(u),E(u)\big), \RR^c\Big)}
\end{equation}
when $(M(u),E(u))\not\in\RR^c$ and $D(u)=\infty$ otherwise. Note that $\RR^c$ is not, strictly speaking, the complement of $\RR$.  Rather, it is precisely
the set of feasible mass/energy pairs, other than $(0,0)$ and those in $\RR$.

From Theorems~\ref{T:feasible} and~\ref{T:MMR} we see that
\begin{equation}\label{E:E>0}
E(u) \geq 0 \qtq{whenever} D(u)<\infty.
\end{equation}
Properties~(i) and~(ii) follow immediately.  To obtain~(iv), we then need merely exploit the fact that $V(u)>0$ on $\RR$; see Theorem~\ref{T:MMR}.

Next we consider property (vii).  Combining \eqref{E:E>0} and $\Em(M(Q_1))=0$ shows that
\begin{equation}\label{ehd0}
D(u) \geq \tfrac{M(u)}{M(Q_1)-M(u)} \qtq{and so} 1 - \sqrt{\tfrac{M(u)}{M(Q_1)}} \geq \tfrac{1}{2D(u)+2}.
\end{equation}
From the second relation and \eqref{E:Ea1} we then deduce
\begin{equation}\label{ehd1}
\tfrac{1}{4D(u)+4} \|u\|_{\dot H^1_x}^2 \leq E(u) \leq D(u),
\end{equation}
while by this and Sobolev embedding,
\begin{equation}\label{ehd2}
E(u) \lesssim \|u\|_{\dot H^1_x}^2 + \|u\|_{\dot H^1_x}^6 \lesssim \|u\|_{\dot H^1_x}^2 \bigl[ 1 + D(u)^4 \bigr].
\end{equation}
Taken together, \eqref{ehd1} and \eqref{ehd2} show that $E(u)$ and $\|u\|_{\dot H^1_x}^2$ are comparable, uniformly for $D(u)\leq D_0$.

Next observe that \eqref{E3/32} and \eqref{ehd2} yield
$$
\|u\|_{H^1_x}^2 \leq 2E(u) + \tfrac{19}{16} M(u) \lesssim \|u\|_{H^1_x}^2 \bigl[ 1 + D(u)^4 \bigr],
$$
which shows the analogous comparability of $E(u)+M(u)$ and $\|u\|_{H^1}^2$.  To complete the proof of \eqref{E:norm equiv}, we need to show that these
are comparable to $D(u)$.  From \eqref{ehd0} and \eqref{ehd1} we have
$$
E(u)+ M(u) \leq [1+M(Q_1)] D(u) .
$$
On the other hand, if $D(u)\leq D_0 < \infty$, then
$$
D(u) \leq \frac{4 D_0 M(u)}{M(Q_1)}  + E(u) + \frac{4[E(u)+M(u)]}{M(Q_1)}.
$$
To see this we divide into two cases: If $M(u)\geq M(Q_1)/4$, then this follows from $D(u)\leq  D_0$ and $E(u)\geq 0$.  If $M(u)\leq M(Q_1)/4$, we need only observe that
$$
\dist\bigl((M(u),E(u)), \RR^c\bigr) \geq \tfrac{4}{3\sqrt{3}}M(Q_1) - M(u) \geq \tfrac14 M(Q_1).
$$
This completes the proof of part~(vii).

Lastly, we discuss the (strict) monotonicity relations (v) and (vi).  These are elementary from the structure of $D(u)$ and the fact that
\begin{equation}\label{dist mono}
\dist\big((m',e'), \RR^c\big) \geq \dist\big((m,e), \RR^c\big) \qtq{whenever} m'\leq m\qtq{and} e'\leq e,
\end{equation}
which follows from the monotonicity of $\Et(m)$.  Note that we do not claim strict monotonicity in \eqref{dist mono}.  It is not true; part of the boundary of $\RR^c$ is vertical.

Incidentally, lower semicontinuity of $\Et(m)$, which was shown in Theorem~\ref{T:MMR}, guarantees that $\RR^c$ is a closed set.  Thus the minimal distance is
actually achieved by a point in $\RR^c$.
\end{proof}


\section{Small data scattering  and perturbation theory}\label{SEC:4}

As discussed in the introduction, Zhang \cite{Zhang} proved global well-posedness of \eqref{3-5} for initial data in $H^1(\R^3)$; moreover, she proved scattering for solutions whose mass is sufficiently small depending on their kinetic energy.  The goal of this section is to develop several results of a similar flavour that will be needed to run the induction on mass and energy argument of this paper.  We begin with a more explicit version of small data scattering:

\begin{proposition}[Small data scattering]\label{thm: Small data scattering}
There exists $\eta>0$ such that for $u_0\in H^1(\R^3)$ with $D(u_0)\leq \eta$, there exists a unique global solution $u\in C(\R; H^1(\R^3))$ to \eqref{3-5} with initial condition $u(0)=u_0$.  Moreover, the solution satisfies the global spacetime bound
\begin{equation}\label{Zsmall1}
\|u\|_{L^{10}_{t,x}(\R\times \R^3)}\les \|\nabla u_0\|_{L^2(\R^3)}.
\end{equation}
Consequently, the solution $u$ scatters, that is, there exist asymptotic states $u_{\pm}\in H^1(\R^3)$ such that
\begin{align*}
\lim_{t\to\pm\infty}\|u(t)-e^{it\Delta}u_{\pm}\|_{H^1(\R^3)}=0.
\end{align*}
\end{proposition}

\begin{proof}
The proof of existence involves a standard contraction mapping argument.  The main observation is that for $\eta$ sufficiently small, the hypothesis $D(u_0)\leq \eta$
together with Proposition~\ref{P:all about D} parts~(iii) and~(vii) implies
$$
\|u\|_{L_t^\infty L_x^2}^2 +  \|\nabla u\|_{L_t^\infty L_x^2}^2 \les \eta.
$$

With
\begin{equation*}
\Phi(u)(t):=e^{it\Delta}u_0-i\int_0^te^{i(t-t')\Delta}\big(|u|^4u-|u|^2u\big)(t')\, dt',
\end{equation*}
it is not difficult to see that $\Phi$ is a contraction on the ball $B\subset L^{\infty}_t{H}^1_x\cap L^2_t{H}^{1,6}_x\cap L^{10}_{t,x}$ of radius $C_0\eta^{1/2}$ for a sufficiently large absolute constant $C_0>0$.  The key estimates appear already in the proof that $\Phi$ maps this ball to itself: for $u\in B$,
\begin{align*}
\|\Phi (u) &\|_{L^{\infty}_t{H}^1_x\cap L^2_t{H}^{1,6}_x\cap L^{10}_{t,x}} \\
&\les \|\langle\nabla\rangle u_0\|_{L^2}+\|u\|^{4}_{L^{10}_{t,x}}\|\langle\nabla\rangle u\|_{L^2_tL^6_x}
	+\|u\|_{L^{\infty}_tL^2_x} \|u\|_{L^{10}_{t,x}}\|\langle\nabla\rangle u\|_{L^2_tL^6_x}\\
&\les \eta^{1/2}+(C_0\eta^{1/2})^5+ \eta^{1/2} (C_0\eta^{1/2})^2\leq C_0\eta^{1/2},
\end{align*}
provided $\eta$ is sufficiently small and $C_0$ is sufficiently large.

Proving scattering from finite global spacetime norms is standard; we omit the details.  Lastly, uniqueness in the larger class $C(\R; H^1(\R^3))$ can be proved by exploiting endpoint Strichartz estimates; see, for example, \cite{CKSTT, KOPV2011}.
\end{proof}

We will also need the following persistence of regularity result:

\begin{lemma}[Persistence of regularity]\label{LEM:small2}
Let $u\in C(\R; H^1(\R^3))$ be a solution to \eqref{3-5} such that $S:=\|u\|_{L^{10}_{t, x}(\R\times\R^3)} < \infty$.  Then for any $t_0\in \R$ and $k=0,1$ we have
$$
\big\| |\nabla|^k u\big\|_{S^0(\R)}\leq C\big(S, M(u)\big) \big \| |\nabla|^k u(t_0)\big\|_{L^2_x}.
$$
\end{lemma}
	
\begin{proof}
We divide $\R$ into $O\big((\frac{S}{\eps})^{10}\big)$ many subintervals $I_j=[t_j, t_{j+1}]$ such that
$$
\|u\|_{L^{10}_{t, x}(I_j\times\R^3)} \leq \eps,
$$
where $\eps > 0$ is a small constant to be chosen later.
Then by repeating the computation in the proof of Proposition~\ref{thm: Small data scattering}, we have
\begin{align*}
\big\||\nabla|^k u\|_{S^0(I_j)}
& \les \big\| |\nabla|^k u(t_j)\|_{L_x^2} + \|u\|_{L^{10}_{t, x}}^4 \big\||\nabla|^k u\|_{L^{2}_tL^{6}_x} + \| u\|_{L^\infty_tL^2_x}\| u\|_{L^{10}_{t, x}}
\big\||\nabla|^k u\|_{L^{2}_t L^{6}_x} \\
& \les \big\| |\nabla|^k u(t_j)\|_{L_x^2} + \eps \big(\eps^3 + M(u)^{1/2} \big) \big\||\nabla|^k u\|_{S^0(I_j)},
\end{align*}
where all spacetime norms are over $I_j\times\R^3$.  Choosing $\eps = \eps(M(u)) > 0$ sufficiently small, we obtain
$$
\big\||\nabla|^k u\|_{S^0(I_j)}\les  \big\| |\nabla|^k u(t_j)\|_{L_x^2} .
$$
Combining the subintervals $I_j$ yields the claim.
\end{proof}
	
Our last result in this section is a stability result for \eqref{3-5}.  This is essential for any induction on mass/energy argument.

\begin{proposition}[Perturbation theory]\label{Perturbation of {3-5}}
Let $\tilde{u}:I\times\R^3\to\C$ be a solution to the perturbed cubic-quintic NLS
\begin{equation*}
i\partial_t\tilde{u}+\Delta \tilde{u}=|\tilde{u}|^4\tilde{u}-|\tilde{u}|^2\tilde{u}+e
\end{equation*}
for some function $e$.  Suppose  that
\begin{align}
\|\tilde{u}\|_{L^{\infty}_tH^1_x(I\times\R^3)}& \leq E \label{eq: perturb1}\\
\|\tilde{u}\|_{L^{10}_{t,x}(Id\times\R^3)}& \leq L \label{eq: perturb2}
\end{align}
for some $E, L>0$.  Let $u_0\in H^1(\R^3)$ with $\|u_0\|_{L^2(\R^3)}\leq M$ for some $M>0$ and let $t_0\in I$.  There exists $\eps_0=\eps_0(E,L,M)>0$
such that if
\begin{align}
\| u_0 -\tilde u(t_0)\|_{\dot H^1_x} \leq\eps \label{eq: perturb3}\\
\|\nabla e\|_{L^{\frac{10}{7}}_{t,x}(I\times\R^3)}\leq\eps \label{eq: perturb4}
\end{align}
for $0<\eps<\eps_0$, then the unique global solution $u$ to \eqref{3-5} with $u(t_0)=u_0$ satisfies
\begin{align}\label{Zperturb3}
\|\nabla(u-\tilde{u})\|_{S^0(I)}\leq C(E,L,M)\eps,
\end{align}
where $C(E, L, M)$ is a non-decreasing function of $E, L$, and $M$.
\end{proposition}

\begin{proof}
Without loss of generality, we may assume $t_0=0\in I$.  By the usual combinatorial argument (see, for example, \cite{CKSTT}) and Lemma~\ref{LEM:small2}, it suffices to prove the proposition under the additional hypothesis
\begin{equation} \label{Zperturb1}
\|\tilde{u}\|_{L^{10}_{t,x}}+\|\nabla\tilde{u}\|_{L^2_tL^6_x}\leq \eta,
\end{equation}
where $\eta=\eta (M,E) > 0$ is a small constant to be chosen later.  From Theorem~\ref{T:Matador} we know that $u$ exists globally; the point is to prove the estimate
\eqref{Zperturb3}.

Let $w:=u-\tilde{u}$ and let $A(t):=\|w\|_{\dot S^1(I\cap [-t,t])}$.  Then by Strichartz, H\"older, \eqref{eq: perturb3},  \eqref{eq: perturb4}, and \eqref{Zperturb1}, we obtain
\begin{align*}
A(t)&\les \| u_0 -\tilde u(t_0)\|_{\dot H^1_x} + \big\|\nabla\big(|u|^4u  -|\tilde{u}|^4\tilde{u}\big)\big\|_{L^{\frac{10}{9}}_tL^{\frac{30}{17}}_x} + \big\| \nabla\big(|u|^2u-|\tilde{u}|^2\tilde{u}\big)\big\|_{L^{\frac{5}{3}}_tL^{\frac{30}{23}}_x} \\
&\quad +\|\nabla e\|_{L^{\frac{10}{7}}_{t,x}}\\
&\les \eps+ \|\nabla w\|_{L^2_tL^6_x} \big[\|\tilde{u}\|_{L^{10}_{t,x}}^4+ \|w\|_{L^{10}_{t,x}}^4\big]
+ \|\nabla\tilde{u}\|_{L^2_tL^6_x}\big[\|\tilde{u}\|_{L^{10}_{t,x}}^3\|w\|_{L^{10}_{t,x}} + \|w\|_{L^{10}_{t,x}}^4\big]\\
&\quad +\|\nabla w\|_{L^2_tL^6_x} \big[\|\tilde{u}\|_{L^{\infty}_tL^2_x} \|\tilde{u}\|_{L^{10}_{t,x}}+ \|w\|_{L^{\infty}_tL^2_x}\|w\|_{L^{10}_{t,x}}\big]\\
&\quad+ \|\nabla\tilde{u}\|_{L^2_tL^6_x}\|w\|_{L^{10}_{t,x}}\big[\|\tilde u\|_{L^{\infty}_tL^2_x}+\|w\|_{L^{\infty}_tL^2_x}\big]\\
&\les \eps + \eta^4 A(t) + \eta A(t)^4 + A(t)^5 + \eta (M +E) A(t) + (M+E) A(t)^2,
\end{align*}
where all spacetime norms are over $(I\cap [-t,t])\times\R^3$.  To obtain the last inequality above we also used \eqref{eq: perturb1} and the conservation of mass for \eqref{3-5} to estimate
$$
\|w\|_{L^{\infty}_tL^2_x}\leq\|\tilde{u}\|_{L^{\infty}_tL^2_x}+\|u\|_{L^{\infty}_tL^2_x}\leq E+\|u_0\|_{L^2}\leq E+M.
$$
Therefore, taking $\eta$ sufficiently small depending on $M$ and $E$ and then using the standard continuity argument to remove the restriction to $[-t,t]$, we derive \eqref{Zperturb3}.
\end{proof}


\section{Linear profile decomposition in $H^1(\R^3)$}
\label{SEC:5}

The goal of this section is to derive a concentration-compactness principle for the linear Schr\"odinger propagator for bounded sequences of initial data in the \emph{inhomogeneous} Sobolev space $H^1(\R^3)$.  Specifically, we prove that any bounded sequence in $H^1(\R^3)$ can be written as an asymptotically decoupling superposition of translations and rescalings of linear evolutions of fixed profiles plus a remainder whose linear evolution converges to zero in the $L^{10}_{t,x}$-norm.  The fact that $H^1(\R^3)$ is not a scale-invariant space makes the argument more complicated; in particular, limiting profiles may not belong to $H^1(\R^3)$.  On the other hand, working in $H^1(\R^3)$ guarantees that the profiles cannot live at arbitrarily large length scales; cf. $\ld_{\infty}<\infty$ in Proposition~\ref{propo: Inverse Strichartz}.

A linear profile decomposition for the linear Schr\"odinger propagator for bounded sequences in the \emph{homogeneous} Sobolev space $\dot H^1(\R^3)$ was first obtained by Keraani \cite{keraani-h1}, relying on an improved Sobolev inequality proved by G\'erard, Meyer, and Oru \cite{GerardMO}.  We should also note the influential precursor \cite{bahouri-gerard} which treated the wave equation.  A linear profile decomposition for the Schr\"odinger propagator for bounded sequences in $L^2(\R^d)$ can be found in \cite{BegoutVargas, carles-keraani, merle-vega}.

The approach we follow here mirrors that in \cite{ClayNotes, Visan:Oberwolfach}, rather than \cite{keraani-h1}.  In particular, we will build our linear profile decomposition based on an improved Strichartz inequality, which states that if the linear evolution is large in the  $L^{10}_{t,x}$-norm, then at least one frequency annulus contributes non-trivially.  The improved Strichartz inequality is used to find the correct scaling parameters in the linear profile decomposition.

\begin{lemma}[Improved Strichartz estimate]\label{lemma: improved Strichartz}
For $f\in\dot{H}^1(\R^3)$ we have
\begin{align}\label{Zimp1}
\|e^{it\Delta}f\|_{L^{10}_{t,x}(\R\times \R^3)}\lesssim \|f\|_{\dot{H}^1(\R^3)}^{\frac15}\sup_{N\in 2^{\Z}}\|e^{it\Delta}P_Nf\|_{L^{10}_{t,x}(\R\times \R^3)}^{\frac45}.
\end{align}
\end{lemma}

\begin{proof}
By the square function estimate followed by the H\"older, Bernstein, Strichartz, and Cauchy--Schwarz inequalities, we obtain
\begin{align*}
\|e^{it\Delta}  f \|_{L^{10}_{t,x}}^{10}
&\sim \bigg\|\Big(\sum_{N\in 2^{\Z}}|e^{it\Delta}f_N|^2\Big)^{\frac12}\bigg\|_{L^{10}_{t,x}}^{10}\\
&\lesssim \sum_{N_1\leq\dots\leq N_5}\int_{\R}\int_{\R^3}\prod_{j = 1}^5 |e^{it\Delta}f_{N_j}|^2\,dx\,dt\\
&\lesssim \sum_{N_1\leq\dots\leq N_5} \|e^{it\Delta}f_{N_1}\|_{L^{\infty}_{t,x}}\|e^{it\Delta}f_{N_1}\|_{L^{10}_{t,x}}\bigg(\prod_{j = 2}^4 \|e^{it\Delta}f_{N_j}\|_{L^{10}_{t,x}}^2\bigg)\\
&\hphantom{XXXXX} \times \|e^{it\Delta}f_{N_5}\|_{L^{10}_{t,x}} \|e^{it\Delta}f_{N_5}\|_{L^{5}_{t,x}}\\
&\lesssim \sup_{N\in 2^{\Z}}\|e^{it\Delta}f_N\|_{L^{10}_{t,x}}^8 \sum_{N_1\leq\dots\leq N_5}
N_1^{\frac32}\|e^{it\Delta}f_{N_1}\|_{L^{\infty}_tL^2_x}N_5^{\frac12}\|e^{it\Delta}f_{N_5}\|_{L^5_tL^{\frac{30}{11}}_x}\\
&\lesssim \sup_{N\in 2^{\Z}}\|e^{it\Delta}f_N\|_{L^{10}_{t,x}}^8\sum_{N_1\leq N_5}\Big( \log \frac{N_5}{N_1}\Big)^3 N_1^{\frac32}\|f_{N_1}\|_{L^2_x}
N_5^\frac{1}{2}\|f_{N_5}\|_{L^2_x}\\
&\lesssim \sup_{N\in 2^{\Z}}\|e^{it\Delta}f_N\|_{L^{10}_{t,x}}^8\sum_{N_1\leq N_5}\Big(\frac{N_1}{N_5}\Big)^{\frac12-\eps}\|f_{N_1}\|_{\dot{H}_x^1}\|f_{N_5}\|_{\dot{H}_x^1}\\
&\lesssim \sup_{N\in 2^{\Z}}\|e^{it\Delta}f_N\|_{L^{10}_{t,x}(\R\times\R^3)}^8\|f\|_{\dot{H}_x^1}^2,
\end{align*}
for any $\eps>0$.  This completes the proof of the lemma.
\end{proof}

Our next result is an inverse Strichartz inequality, which says that if the linear evolution is large, it must in fact contain a bubble of concentration at some point in spacetime.

\begin{proposition}[Inverse Strichartz inequality]\label{propo: Inverse Strichartz}
Let $\{f_n\}_{n\in\mathbb{N}}\subset H^1(\R^3)$ be a sequence such that
\begin{align*}
\limsup_{n\to\infty}\|f_n\|_{H^1_x}=A<\infty \qtq{and} \liminf_{n\to\infty}\|e^{it\Delta}f_n\|_{L^{10}_{t,x}(\R\times\R^3)}=\eps>0.
\end{align*}
Passing to a subsequence if necessary, there exist
$$
\phi\in\dot{H}^1(\R^3),\quad \{\ld_n\}_{n\in\mathbb{N}}\subset (0,\infty), \quad \{t_n\}_{n\in\mathbb{N}}\subset\R, \qtq{and}  \{x_n\}_{n\in\mathbb{N}}\subset \R^3
$$
such that all of the following hold:  First $\ld_n\to\ld_{\infty}\in [0, \infty)$, and if $\ld_{\infty}>0$ then additionally $\phi\in L^2(\R^3)$.  Secondly,
\begin{align}\label{eq:weak conv}
\ld_n^{\frac12}\big(e^{it_n\Delta}f_n\big)(\ld_nx+x_n)\rightharpoonup \phi(x) \qtq{weakly in}
\begin{cases}
H^1(\R^3), &  \qtq{if} \ld_\infty > 0, \\
\dot{H}^1(\R^3), &  \qtq{if} \ld_\infty = 0.
\end{cases}
\end{align}
Furthermore, letting
\begin{align}
\phi_n(x):=
\begin{cases}
\ld_n^{-\frac12}e^{-it_n\Delta}\big[\phi\big(\tfrac{x-x_n}{\ld_n}\big) \big], & \qtq{if}\ld_{\infty}>0,\\
\ld_n^{-\frac12}e^{-it_n\Delta}  \big[\big(P_{\geq \lambda_n^\theta}\phi\big)\big(\tfrac{x-x_n}{\ld_n}\big)\big], & \qtq{if} \ld_{\infty}=0,
\end{cases}
\label{Zinv0a}
\end{align}
with $0<\theta<1$ fixed, the following decoupling statements hold:
\begin{align}\label{eq:decoupling dotH1}
& \lim_{n\to\infty}\Bigl[\|f_n\|_{\dot{H}^1_x}^2-\|f_n-\phi_n\|_{\dot{H}^1_x}^2\Bigr] =\|\phi\|_{\dot{H}^1_x}^2\gtrsim A^2\big(\tfrac{\eps}{A}\big)^{\frac{15}{4}}, \\
& \lim_{n\to\infty}\Bigl[\|f_n\|_{L^2_x}^2-\|f_n-\phi_n\|_{L^2_x}^2 -\|\phi_n\|_{L^2_x}^2\Bigr]=0. \label{eq:decoupling L2}
\end{align}
\end{proposition}

\begin{proof}
Passing to a subsequence, we may assume that
\begin{align}
\|e^{it\Delta}f_n\|_{L^{10}_{t,x}}\geq \tfrac{\eps}{2} \quad \text{and} \quad \|f_n\|_{H^1_x}\leq 2A
\label{Zinv1}
\end{align}
for all $n\in\mathbb{N}$.  By Lemma~\ref{lemma: improved Strichartz}, we have
\begin{align*}
\eps\les \|e^{it\Delta}f_n\|_{L^{10}_{t,x}} \les \|f_n\|_{\dot{H}^1_x}^{\frac15} \sup_{N}\|e^{it\Delta}P_Nf_n\|_{L^{10}_{t,x}}^{\frac45}
\les A^{\frac15}\sup_{N}\|e^{it\Delta}P_Nf_n\|_{L^{10}_{t,x}}^{\frac45}
\end{align*}
for all $n\in\mathbb{N}$.  Thus, for each $n \in \mathbb{N}$ there exists $N_n\in 2^\Z$ such that
\begin{equation}\label{Zinv2}
\|e^{it\Delta}P_{N_n}f_n\|_{L^{10}_{t,x}}\gtrsim \eps^{\frac54}A^{-\frac14}.
\end{equation}
On the other hand, by H\"older, Strichartz, and \eqref{Zinv1},
\begin{align}
\|e^{it\Delta}P_{N_n}f_n\|_{L^{10}_{t,x}}
\leq \|e^{it\Delta}P_{N_n}f_n\|_{L^{\infty}_{t,x}}^{\frac23}\|e^{it\Delta}P_{N_n}f_n\|_{L^{\frac{10}3}_{t,x}}^{\frac13}
&\les \|e^{it\Delta}P_{N_n}f_n\|_{L^{\infty}_{t,x}}^{\frac23}\|P_{N_n}f_n\|_{L^2_x}^{\frac13}\notag \\
&\les \|e^{it\Delta}P_{N_n}f_n\|_{L^{\infty}_{t,x}}^{\frac23}N_{n}^{-\frac13}A^{\frac13}.\label{Zinv3}
\end{align}
Combining \eqref{Zinv2} and \eqref{Zinv3}, we get
\begin{align} \label{Zinv4'}
N_n^{-\frac12}\|e^{it\Delta}P_{N_n}f_n\|_{L^{\infty}_{t,x}}\gtrsim \eps^{\frac{15}{8}}A^{-\frac{7}{8}}.
\end{align}
Therefore, there exist $t_n\in\R$ and $x_n\in\R^3$ such that
\begin{align} \label{Zinv4}
N_n^{-\frac12}\bigl|\bigl[e^{it_n\Delta}P_{N_n}f_n\bigr](x_n)\bigr|\gtrsim \eps^{\frac{15}{8}}A^{-\frac{7}{8}}.
\end{align}

We choose the spatial scales to be $\ld_n:=N_n^{-1}$.  We note that  $\{\ld_n\}_{n\in\mathbb{N}}$ is a bounded sequence; indeed, by \eqref{Zinv4'}, Bernstein's inequality, and \eqref{Zinv1},
\begin{align*}
\ld_n^{-\frac12}\eps^{\frac{15}{8}}A^{-\frac{7}{8}} \les \|e^{it_n\Delta}P_{\ld_n^{-1}}f_n\|_{L^{\infty}_x}\les \ld_n^{- \frac{3}{2}}\|e^{it_n\Delta}P_{\ld_n^{-1}}f_n\|_{L^2_x}
\les  \ld_n^{-\frac32}\|f_n\|_{L^2_x}\les\ld_n^{-\frac32}A.
\end{align*}
Thus, passing to a subsequence we may assume that $\lambda_n\to \lambda_\infty$, where either $0<\lambda_\infty\les (\frac A\eps)^{\frac{15}8}$ or $\lambda_\infty=0$.

It remains to find the profile $\phi$.  To this end, let $h_n(x):=\ld_n^{\frac12}(e^{it_n\Delta}f_n)(\ld_nx+x_n)$ as in LHS\eqref{eq:weak conv}.  By \eqref{Zinv1}, we have
\begin{align}\label{hn bdd}
\|h_n\|_{\dot{H}^1_x}=\|f_n\|_{\dot{H}^1_x}\leq 2A \qtq{for all} n\in\mathbb{N}.
\end{align}
Thus, passing to a subsequence, we can find $\phi\in\dot{H}^1(\R^3)$ such that $h_n\rightharpoonup \phi$ in $\dot{H}^1(\R^3)$.  Moreover, if $\ld_n\to\ld_{\infty}>0$, then $$
\|h_n\|_{L^2_x}=\ld_n^{-1}\|f_n\|_{L^2_x}\les\ld_{\infty}^{-1}A
$$
for all large enough $n$.  Thus by the uniqueness of weak limits and weak lower semicontinuity of the norm, we get that in this case $\phi\in H^1(\R^3)$.  This proves \eqref{eq:weak conv}.

As $\dot{H}^1_x$ is a Hilbert space, the weak convergence of $\{h_n\}$ to $\phi$ in $\dot H^1_x$ implies that
\begin{align}\label{Zinv5}
\|h_n\|_{\dot{H}^1_x}^2-\|h_n-\phi\|_{\dot{H}^1_x}^2-\|\phi\|_{\dot{H}^1_x}^2 =2\Re\jb{\phi, h_n-\phi}_{\dot H^1_x}\to 0.
\end{align}
Combining this with the fact that
$$
\|\phi- P_{\geq \lambda_n^\theta}\phi\|_{\dot H^1_x} \to 0 \qtq{when} \lambda_n\to 0
$$
and performing a change of variables, we obtain the equality in \eqref{eq:decoupling dotH1}.

To complete the proof of \eqref{eq:decoupling dotH1}, we must prove the lower bound on $\|\phi\|_{\dot{H}^1_x}$.  To this end, let $k_N$ denote the kernel of the operator $P_N$.  As $\hat{k}_{N_n}(\xi)= \hat{k}_1\big(\frac{\xi}{N_n}\big)$ and $\ld_n = N_n^{-1}$, we have $k_{N_n}(x) = \ld_n^{-3}k_1\big(\frac{x}{\ld_n}\big)$.  Thus
\begin{align*}
\jb{\phi,k_1}&=\lim_{n\to\infty}\int \ld_n^{\frac12}\big(e^{it_n\Delta}f_n\big)(\ld_n x+x_n)k_1(x)\,dx\\
&=\lim_{n\to\infty}\int \ld_n^{\frac12}\big(e^{it_n\Delta}f_n\big)(y)k_{N_n}(y - x_n)\, dy
=\lim_{n\to\infty}\ld_n^{\frac12}\big[e^{it_n\Delta}P_{N_n}f_n\big](x_n).
\end{align*}
Invoking \eqref{Zinv4}, we derive
\begin{equation}\label{eq: bound <phi,k1>}
|\jb{\phi,k_1}|\ges \eps^{\frac{15}8}A^{-\frac{7}{8}}.
\end{equation}
On the other hand,
\begin{align}\label{Zinv6}
|\jb{\phi,k_1}|\leq \|\phi\|_{L_x^6}\|k_1\|_{L_x^{6/5}}\les \|\phi\|_{\dot H^1_x}.
\end{align}
Combining \eqref{eq: bound <phi,k1>} and \eqref{Zinv6} completes the proof of \eqref{eq:decoupling dotH1}.

We now turn to \eqref{eq:decoupling L2}.  When $\ld_n\to\ld_{\infty}>0$ and $\phi\in H^1(\R^3)$, proceeding as in \eqref{Zinv5} we obtain
\begin{align*}
\|h_n\|_{L^2_x}^2-\|h_n-\phi\|_{L^2_x}^2-\|\phi\|_{L^2_x}^2\to 0;
\end{align*}
a simple change of variables now yields the claim.  When $\ld_n\to 0$, we write
\begin{align*}
\|f_n\|_{L^2_x}^2-\|f_n-\phi_n\|_{L^2_x}^2-\|\phi_n\|_{L^2_x}^2
&=2\Re \langle f_n-\phi_n,\phi_n\rangle_{L^2_x}\\
&=2\Re\big\langle h_n -P_{\geq \lambda_n^\theta}\phi, \ld_n^2 P_{\geq \lambda_n^\theta}\phi\big\rangle_{L^2_x}\\
&=2\Re\big\langle h_n-P_{\geq \lambda_n^\theta}\phi, \ld_n^2 |\nabla|^{-2}P_{\geq \lambda_n^\theta}\phi\big\rangle_{\dot H^1_x}.
\end{align*}
As $\ld_n\to 0$ and $0<\theta<1$, we have
\begin{align*}
\big\||\nabla|^{-2}\ld_n^2 P_{\geq \lambda_n^\theta}\phi\big\|_{\dot{H}^1_x} \les \ld_n^{2(1-\theta)}\|P_{\geq \lambda_n^\theta}\phi\|_{\dot{H}^1_x}
	\les \ld_n^{2(1-\theta)}\|\phi\|_{\dot{H}^1_x}\to 0,
\end{align*}
and so \eqref{eq:decoupling L2} follows from the boundedness of $\{h_n -P_{\geq \lambda_n^\theta} \phi\}_{n \in \mathbb{N}}$ in $\dot H^1(\R^3)$.

This completes the proof of the proposition.
\end{proof}

The following corollary allows us to assume that the scaling parameters $\ld_n$ and the temporal parameters $t_n$ in Proposition \ref{propo: Inverse Strichartz} satisfy simple dichotomies.

\begin{corollary} \label{COR:Zinv}
Passing to a further subsequence if necessary, we may choose the parameters $\{\ld_n\}_{n\in \mathbb N}$ and $\{t_n\}_{n\in \mathbb N}$ in Proposition~\ref{propo: Inverse Strichartz} such that
\begin{gather}\label{Zinv11}
\begin{aligned}
\textup{(i)} \qquad &\ld_n\equiv 1 \quad \text{or} \quad \ld_n\to 0 \\
\textup{(ii)} \qquad &\, t_n\equiv 0 \quad \text{or} \quad \frac{t_n}{\ld_n^2}\to\pm\infty.
\end{aligned}
\end{gather}
\end{corollary}

\begin{proof}
Let $\ld_n$ and $t_n$ denote the scale and temporal parameters from Proposition~\ref{propo: Inverse Strichartz}.  To ease notation, for $\lambda>0$ let $D_{\ld}$ denote the dilation operator
$$
(D_{\ld}f)(x):=\ld^{-\frac12}f(\ld^{-1}{x}).
$$

We first consider claim (i) in \eqref{Zinv11}.  If $\ld_n\to 0$, then there is nothing to prove.

Suppose therefore that $\ld_n\to\ld_{\infty}>0$.  In this case we may redefine $\ld_n\equiv 1$, provided we also replace the profile $\phi$ by
$\tilde \phi:=D_{\ld_\infty}\phi$.  Indeed, as $D_{\ld_n}$ and $D_{\ld_n}^{-1}$ converge strongly to $D_{\ld_{\infty}}$ and $D_{\ld_{\infty}}^{-1}$, respectively, as operators on $H^1(\R^3)$, all the requisite conclusions carry over.

We now turn to (ii).  Passing to a subsequence, we may assume that $\frac{t_n}{\ld_n^2}\to \tau_\infty\in [-\infty,\infty]$.  If $\tau_\infty = \pm\infty$,
then there is nothing to prove.  If $\tau_\infty\in\R$, then we may redefine $t_n\equiv 0$, provided we also replace the profile $\phi$ by
$\tilde \phi:= e^{-i\tau_\infty \Delta}\phi$.  Indeed, this is easily seen using the identity
\begin{equation*}
e^{-it_n\Delta}D_{\ld_n}\phi= D_{\ld_n}\big[e^{-i\frac{t_n}{\ld_n^2}\Delta}\phi\big]
\end{equation*}
and the strong continuity of $e^{it\Delta}$ in $H^1_x$.  To see \eqref{eq:decoupling L2} when $\lambda_n\to 0$, we also use Bernstein's inequality:
\begin{align*}
\big\|e^{-it_n\Delta}D_{\ld_n} P_{\geq \ld_n^\theta} \phi- D_{\ld_n}P_{\geq \ld_n^\theta} e^{-i\tau_\infty\Delta}\phi\big\|_{L^2_x}
& = \big\| D_{\ld_n}\big[e^{-i\frac{t_n}{\ld_n^2}\Delta} - e^{-i\tau_\infty\Delta}\big] P_{\geq \ld_n^\theta} \phi\big\|_{L^2_x}\\
& \les\ld_n^{1-\theta}\big\|\big[e^{-i\frac{t_n}{\ld_n^2}\Delta} - e^{-i\tau_\infty\Delta}\big] P_{\geq \ld_n^\theta} \phi\big\|_{\dot{H}^1_x} \\
&\lesssim \ld_n^{1-\theta} \|\phi\|_{\dot{H}^1_x}\to 0 \quad\text{as $n\to\infty$.}
\end{align*}

This completes the proof of the corollary.
\end{proof}

As explained in the introduction, to prove Theorem~\ref{thm: main} we will run an induction argument on both the mass and the energy.  To this end, we need to show that the masses and the energies corresponding to two distinct profiles decouple asymptotically.  The asymptotic decoupling of the kinetic energies and the masses will follow from \eqref{eq:decoupling dotH1} and \eqref{eq:decoupling L2}, respectively.  The asymptotic decoupling of the potential energies will follow from the following lemma.

\begin{lemma}\label{Cor: decoupling L4,L6}
Under the hypotheses of Proposition~\ref{propo: Inverse Strichartz}, we have
\begin{align}
&\lim_{n \to \infty}\Big[\|f_n\|_{L^6_x}^6-\|f_n- \phi_n\|_{L^6_x}^6-\|\phi_n\|_{L^6_x}^6\Big] = 0\label{eq:decoupling L6}, \\
&\lim_{n \to \infty}\Big[\|f_n\|_{L^4_x}^4-\|f_n- \phi_n\|_{L^4_x}^4-\|\phi_n\|_{L^4_x}^4\Big] = 0\label{eq:decoupling L4}.
\end{align}
\end{lemma}

\begin{proof}
Passing to a subsequence, we may assume that the conclusions of Corollary~\ref{COR:Zinv} hold.  As above, we will write $D_{\ld}$ for the dilation operator
$$
(D_{\ld}f)(x):=\ld^{-\frac12}f(\ld^{-1}{x}).
$$

We start by proving the decoupling of the $L^6_x$-norms.  First, we consider the case when $\frac{t_n}{\ld_n^2}\to\pm\infty$.
Let $\eps>0$ and choose $\psi\in\mathcal{S}(\R^3)$ such that
$$
\|\psi-\phi\|_{\dot H^1_x}\leq \eps.
$$
Let
\begin{align*}
\psi_n(x):=
\begin{cases}
e^{-it_n\Delta}\psi(x-x_n), & \qtq{if}\ld_n\equiv 1,\\
e^{-it_n\Delta}  [D_{\ld_n}P_{\geq \lambda_n^\theta}\psi](x-x_n), & \qtq{if} \ld_n\to 0.
\end{cases}
\end{align*}

By the dispersive estimate,
\begin{align*}
\|\psi_n\|_{L^6_x}
&\les \tfrac{1}{|t_n|}\big\|D_{\ld_n}\psi\big\|_{L^{\frac65}_x}\les \tfrac{\ld_n^2}{|t_n|}\|\psi\|_{L^{\frac65}_x}\to 0 \qtq{as} n\to \infty.
\end{align*}
On the other hand, by Sobolev embedding,
\begin{align*}
\big\|\phi_n-\psi_n\big\|_{L^6_x}\les \|\phi-\psi\|_{\dot{H}^1_x} \les \eps \quad \text{uniformly for } n\in\mathbb N.
\end{align*}
Thus, for $n$ large enough,
$$
\|\phi_n\|_{L_x^6} \les \eps
$$
and so, by the triangle inequality,
\begin{align*}
\Big|\|f_n\|_{L^6_x}-\|f_n- \phi_n\|_{L^6_x}\Big|\les \eps.
\end{align*}
As $\eps>0$ was arbitrary, we obtain \eqref{eq:decoupling L6} when $\frac{t_n}{\ld_n^2}\to\pm\infty$.

Next, we consider the case when $t_n\equiv 0$.  By \eqref{eq:weak conv}, $h_n(x):=[D_{\ld_n}^{-1}f_n](x+x_n)\rightharpoonup \phi(x)$ in $\dot H^1_x$.  As $\|h_n\|_{\dot H^1_x}= \|f_n\|_{\dot H^1_x}\les A$, invoking the Rellich--Kondrashov Theorem and passing to a subsequence we conclude that $h_n\to\phi$ in $L^2_x(K)$ for any compact set $K\subset\R^3$.  Using a diagonal argument and passing to a further subsequence, we deduce that $h_n \to\phi$ almost everywhere on  $\R^3$.  Thus by Lemma~\ref{LEM:B-L},
\begin{align*}
\|h_n\|_{L^6_x}^6 -\|h_n-\phi\|_{L^6_x}^6-\|\phi\|_{L^6_x}^6\to 0 \qtq{as} n\to \infty.
\end{align*}
Using also the fact that
$$
\|P_{\geq \lambda_n^\theta} \phi-\phi\|_{L_x^6} \to 0 \qtq{when} \ld_n\to 0
$$
and performing a simple change of variables, we obtain \eqref{eq:decoupling L6}.

Next we consider the decoupling of the $L^4_x$-norms.  When $\ld_n\to 0$, the proof is very simple; indeed, the claim follows from
\begin{align*}
\|\phi_n\|_{L^4_x}&\les \|\phi_n\|_{L_x^2}^{\frac14} \|\phi_n\|_{\dot H_x^1}^{\frac34}
\les \ld_n^{\frac14} \|P_{\geq \ld_n^\theta}\phi\|_{L_x^2}^{\frac14}\|\phi\|_{\dot H_x^1}^{\frac34}\les \lambda_n^{\frac{1-\theta}4}\|\phi\|_{\dot H_x^1} \to 0 \qtq{as} n\to \infty.
\end{align*}
When $\ld_n\equiv 1$, we argue as for \eqref{eq:decoupling L6}.  One small difference appears in the case when $t_n/\ld_n^2\to\pm\infty$; in this case, we choose $\psi\in\mathcal{S}(\R^3)$ to approximate $\phi$ in $H^1_x$ (rather than $\dot H^1_x$) and use the fact that by Gagliardo--Nirenberg, the $L^4_x$-norm is controlled by the $H^1_x$-norm.

This completes the proof of the lemma.
\end{proof}

We are now ready to prove the main result of this section.

\begin{theorem}[Linear profile decomposition]\label{thm:profile decomposition}
Let $\{f_n\}_{n\in\mathbb{N}}$ be a bounded sequence in $H^1(\R^3)$.  Passing to a subsequence if necessary, there exists $J^{\ast}\in\{0,1,2,\dots\}\cup\{\infty\}$ such that for each finite $1\leq j\leq J^{\ast}$ there exist
$$
\phi^j\in\dot{H}^1_x\setminus \{0\}, \quad \{\ld_n^j\}_{n\in\mathbb{N}}\subset (0,1], \quad \{t_n^j\}_{n\in\mathbb{N}}\subset \R, \quad \text{and} \quad \{x_n^j\}_{n\in\mathbb{N}}\subset\R^3,
$$
satisfying
\begin{align*}
\ld_n^j\equiv 1 \qtq{or} \ld_n^j\to 0  \qquad \qtq{and} \qquad t_n^j\equiv 0 \qtq{or} t_n^j\to\pm\infty.
\end{align*}
If $\ld_n^j\equiv 1$, then additionally $\phi^j\in L^2_x$.  Choosing $0<\theta<1$ and defining
\begin{align}\label{Zpro-1}
\phi_n^j(x):=
\begin{cases}
\big[e^{it_n^j\Delta}\phi^j\big](x-x_n^j), & \qtq{if}\ld_n^j\equiv 1,\\
(\ld_n^j)^{-\frac12}  \big[e^{it_n^j\Delta}P_{\geq (\lambda_n^j)^\theta}\phi^j\big]\big(\tfrac{x-x_n^j}{\ld_n^j}\big), & \qtq{if} \ld_n^j\to 0,
\end{cases}
\end{align}
for each finite $1\le J\leq J^{\ast}$ we have the decomposition
\begin{align}\label{E:LPD}
f_n=\sum_{j=1}^{J}\phi_n^j+w_n^J
\end{align}
and the following statements hold:
\begin{gather}
\lim_{J\to J^{\ast}}\limsup_{n\to\infty}\|e^{it\Delta}w_n^J\|_{L^{10}_{t,x}(\R\times \R^3)}=0, \label{eq: profile reminder}\\
e^{-it_n^j\Delta}\bigl[ (\ld_n^j)^{\frac12}w_n^J(\ld_n^jx+x_n^j) \bigr] \rightharpoonup 0 \quad \text{in } \dot H^1_x \text{ for all }1\leq j\leq J, \label{eq: profile weak conv}\\
\sup_{J}\lim_{n\to\infty}\Big[M(f_n)- \sum_{j=1}^JM(\phi_n^j)-M(w_n^J)\Big]=0, \label{eq: profile decoupling L2}\\
\sup_{J}\lim_{n\to\infty}\Big[E(f_n)- \sum_{j=1}^JE(\phi_n^j)-E(w_n^J)\Big]=0, \label{eq: profile decoupling H1}\\
\lim_{n\to \infty}\bigg[\frac{\ld_n^j}{\ld_n^k}+\frac{\ld_n^k}{\ld_n^{j}}+\frac{|x_n^j-x_n^k|^2}{\ld_n^j\ld_n^k}+\frac{\big|t_n^j(\ld_n^j)^2-t_n^k(\ld_n^k)^2\big|}{\ld_n^j\ld_n^k}\bigg] =  \infty \quad \text{for all  } j\neq k. \label{eq: profile orthogonality}
\end{gather}
\end{theorem}

\begin{proof}
To keep our formulas within margins, we will use operators $g_n^j$ defined by
$$
(g_n^j f)(x):= (\lambda_n^j)^{-\frac12}  f\bigl(\tfrac{x-x_n^j}{\lambda_n^j}\bigr) \qtq{or} [(g_n^j)^{-1} f](x):= (\lambda_n^j)^{\frac12}  f\bigl(\lambda_n^jx+x_n^j\bigr).
$$
With this notation, we have
\begin{align*}
\phi_n^j= \begin{cases}
g_n^je^{it_n^j\Delta}\phi^j, & \qtq{if}\ld_n^j\equiv 1,\\
g_n^j e^{it_n^j\Delta}P_{\geq (\lambda_n^j)^\theta}\phi^j, & \qtq{if} \ld_n^j\to 0.
\end{cases}
\end{align*}

To prove the theorem we will proceed inductively, extracting one bubble at a time.  To start, we set $J=0$ and $w_n^0:=f_n$.  Now suppose we have a decomposition up to level $J\geq 0$ obeying \eqref{eq: profile decoupling L2}, \eqref{eq: profile decoupling H1}, and the $j=J$ case of \eqref{eq: profile weak conv}.  (Conditions
\eqref{eq: profile reminder}, \eqref{eq: profile weak conv} with $j<J$, and \eqref{eq: profile orthogonality} will be verified at the end.)  Passing to a subsequence if necessary, we set
\begin{align*}
A_J:=\lim_{n\to\infty} \|w_n^J\|_{H^1_x} \qtq{and} \eps_J:=\lim_{n\to \infty} \|e^{it\Delta}w_n^J\|_{L_{t,x}^{10}}.
\end{align*}
If $\eps_J=0$, we stop and set $J^*=J$.  If not, we apply Proposition~\ref{propo: Inverse Strichartz} to $w_n^J$.  Thus, passing to a subsequence in $n$ we find $\phi^{J+1}$, $\{\lambda_n^{J+1}\}$, $\{t_n^{J+1}\}$, and $\{x_n^{J+1}\}$, with
$$
\ld_n^{J+1}\equiv 1 \qtq{or} \ld_n^{J+1}\to 0  \qquad \qtq{and} \qquad t_n^{J+1}\equiv 0 \qtq{or} t_n^{J+1}\to\pm\infty.
$$
Note that we renamed the time parameters given by Proposition~\ref{propo: Inverse Strichartz} so that $t_n^{J+1} := - \lambda_n^{-2} t_n$.
According to Proposition~\ref{propo: Inverse Strichartz}, the profile $\phi^{J+1}$ is defined as the weak limit
\begin{align*}
\phi^{J+1}&=\wlim_{n\to\infty}(g_n^{J+1})^{-1} \bigl[ e^{-it_n^{J+1}(\lambda_n^{J+1})^2\Delta}w_n^J\bigr]
=\wlim_{n\to\infty}e^{-it_n^{J+1}\Delta}[(g_n^{J+1})^{-1}w_n^J].
\end{align*}

Now define $w_n^{J+1}:=w_n^J-\phi_n^{J+1}$.  By the definition of $\phi^{J+1}$, we obtain \eqref{eq: profile weak conv} for $j=J+1$.  Moreover, from Proposition~\ref{propo: Inverse Strichartz} and Lemma~\ref{Cor: decoupling L4,L6} we also have
\begin{gather*}
\lim_{n\to\infty}\Bigl[M(w_n^J) - M(w_n^{J+1}) - M(\phi_n^{J+1})\Bigr]=0,\\
\lim_{n\to\infty}\Bigl[E(w_n^J) - E(w_n^{J+1}) - E(\phi_n^{J+1})\Bigr]=0.
\end{gather*}
Combining this with the inductive hypothesis gives \eqref{eq: profile decoupling L2} and \eqref{eq: profile decoupling H1} at the level $J+1$.

Passing to a further subsequence and using Proposition~\ref{propo: Inverse Strichartz}, we obtain
\begin{equation}\label{new a,eps}
A_{J+1}^2=\lim_{n\to\infty}\|w_n^{J+1}\|_{ H^1_x}^2\leq A_J^2.
\end{equation}
Passing to a further subsequence if necessary, let
$$
\eps_{J+1}:=\lim_{n\to \infty} \|e^{it\Delta}w_n^{J+1}\|_{L_{t,x}^{10}}.
$$
If $\eps_{J+1}=0$ we stop and set $J^*=J+1$; in this case, \eqref{eq: profile reminder} is automatic.  If $\eps_{J+1}>0$ we continue the induction.  If the algorithm does not terminate in finitely many steps, we set  $J^*=\infty$.  In this case, \eqref{eq: profile reminder} follows from the fact that $\eps_J\to 0$ as $J\to \infty$.  To demonstrate that $\eps_J\to 0$, we combine
\eqref{new a,eps} and \eqref{eq:decoupling dotH1} to obtain
$$
 \sum_J A_0^2\big(\tfrac{\eps_J}{A_0}\big)^{\frac{15}{4}}  \les  \sum_j \|\phi^j\|_{\dot H^1_x} ^2 \leq A_0^2.
$$
(Recall that $A_0=\limsup_n \|f_n\|_{H^1_x} <\infty.$)

Next we verify the asymptotic orthogonality condition \eqref{eq: profile orthogonality}; claim \eqref{eq: profile weak conv} with $j<J$ follows from a similar argument.  We argue by contradiction.  Assume \eqref{eq: profile orthogonality} fails to be true for some pair $(j,k)$.  Without loss of generality, we may assume that this is the first pair for which \eqref{eq: profile orthogonality} fails, that is, $j<k$ and \eqref{eq: profile orthogonality} holds for all pairs $(j,l)$ with $j<l<k$.  Passing to a subsequence, we may assume
\begin{align}\label{cg}
\tfrac{\lambda_n^j}{\lambda_n^k}\to \lambda_0\in (0,\infty), \quad \tfrac{x_n^j-x_n^k}{\sqrt{\lambda_n^j\lambda_n^k}}\to x_0, \qtq{and}
\tfrac{t_n^j(\lambda_n^j)^2-t_n^k(\lambda_n^k)^2}{\lambda_n^j\lambda_n^k}\to t_0.
\end{align}
From the inductive relation
\begin{align*}
w_n^{k-1}=w_n^j-\sum_{l=j+1}^{k-1}\phi_n^l
\end{align*}
and the definition of $\phi^k$, we obtain
\begin{align}\label{tp}
\phi^k(x)&=\wlim_{n\to\infty}e^{-it_n^k\Delta}[(g_n^k)^{-1}w_n^{k-1}]\notag\\
&=\wlim_{n\to\infty}e^{-it_n^k\Delta}[(g_n^k)^{-1}w_n^j]-\sum_{l=j+1}^{k-1} \wlim_{n\to \infty}e^{-it_n^k\Delta}[(g_n^k)^{-1}\phi_n^l].\end{align}
We will prove that these weak limits are all zero and so obtain a contradiction to the nontriviality of $\phi^k$.

We write
\begin{align*}
e^{-it_n^k\Delta}[(g_n^k)^{-1}w_n^j]
&=e^{-it_n^k\Delta}(g_n^k)^{-1}g_n^je^{it_n^j\Delta}[e^{-it_n^j\Delta}(g_n^j)^{-1}w_n^j]\\
&=(g_n^k)^{-1}g_n^je^{i\bigl(t_n^j-t_n^k\tfrac{(\lambda_n^k)^2}{(\lambda_n^j)^2}\bigr)\Delta}[e^{-it_n^j\Delta}(g_n^j)^{-1}w_n^j].
\end{align*}
Note that by \eqref{cg},
\begin{align*}
t_n^j-t_n^k\tfrac{(\lambda_n^k)^2}{(\lambda_n^j)^2}=\tfrac{t_n^j(\lambda_n^j)^2-t_n^k(\lambda_n^k)^2}{\lambda_n^j\lambda_n^k}\cdot\tfrac{\lambda_n^k}
{\lambda_n^j}\to \tfrac{t_0}{\lambda_0}.
\end{align*}
Using this together with \eqref{cg} we conclude that the adjoints of the operators
$$
(g_n^k)^{-1}g_n^je^{i\bigl(t_n^j-t_n^k\tfrac{(\lambda_n^k)^2}{(\lambda_n^j)^2}\bigr)\Delta}
$$
converge strongly in $\dot H^1_x$.  Combining this with the $J=j$ case of \eqref{eq: profile weak conv}, we obtain that the first term on RHS\eqref{tp} is zero.

To complete the proof of \eqref{eq: profile orthogonality}, it remains to show that the second term on RHS\eqref{tp} is zero.  For all $j<l<k$ we write
\begin{align*}
e^{-it_n^k\Delta}(g_n^k)^{-1}\phi_n^l
=(g_n^k)^{-1}g_n^je^{i\bigl(t_n^j-t_n^k\tfrac{(\lambda_n^k)^2}{(\lambda_n^j)^2}\bigr)\Delta}[e^{-it_n^j\Delta}(g_n^j)^{-1}\phi_n^l].
\end{align*}
Using the fact that
$$
\|\phi^l- P_{\geq (\ld_n^l)^\theta}\phi^l\|_{\dot H^1_x} \to 0 \qtq{when} \ld_n^l\to 0
$$
and arguing as for the first term on RHS\eqref{tp}, it suffices to show that
\begin{align*}
e^{-it_n^j\Delta}(g_n^j)^{-1}g_n^l e^{it_n^l\Delta}\phi^l\rightharpoonup 0 \quad\text{weakly in } \dot H^1_x.
\end{align*}
Using a density argument, this reduces to
\begin{align}\label{need11}
I_n:=e^{-it_n^j\Delta}(g_n^j)^{-1}g_n^le^{it_n^l\Delta}\phi\rightharpoonup 0 \quad\text{weakly in } \dot H^1_x,
\end{align}
for all $\phi\in C_c^\infty(\R^3)$.  Note that we can rewrite $I_n$ as follows:
\begin{align*}
I_n=\Bigl(\tfrac{\lambda_n^j}{\lambda_n^l}\Bigr)^{\frac12}\Bigl[e^{i\bigl(t_n^l-t_n^j\bigl(\tfrac{\lambda_n^j}
{\lambda_n^l}\bigr)^2\bigr)\Delta}\phi\Bigr]\Bigl(\tfrac{\lambda_n^j x+x_n^j- x_n^l}{\lambda_n^l}\Bigr).
\end{align*}

Recalling that \eqref{eq: profile orthogonality} holds for the pair $(j,l)$, we first prove \eqref{need11} when the scaling parameters are not comparable, that is,
\begin{align}\label{E:divg lambda}
\lim_{n\to\infty}\tfrac{\lambda_n^j}{\lambda_n^l}+\tfrac{\lambda_n^l}{\lambda_n^j}=\infty.
\end{align}
By the Cauchy--Schwarz inequality,
\begin{align*}
\bigl|\langle I_n, \psi\rangle_{\dot H^1_x}\bigr|
&\les \min\Bigl\{\|\Delta I_n\|_{L^2_x}\|\psi\|_{L^2_x}, \|I_n\|_{L^2_x}\|\Delta\psi\|_{L^2_x}\Bigr\}\\
&\les \min\Bigl\{\tfrac{\lambda_n^j}{\lambda_n^l}\|\Delta\phi\|_{L^2_x}\|\psi\|_{L^2_x}, \tfrac{\lambda_n^l}{\lambda_n^j}\|\phi\|_{L^2_x}\|\Delta\psi\|_{L^2_x}\Bigr\},
\end{align*}
which converges to zero as $n\to \infty$, for all $\psi\in C_c^\infty(\R^3)$.  This establishes \eqref{need11} when \eqref{E:divg lambda} holds.

Henceforth we may assume
\begin{align*}
\lim_{n\to \infty}\tfrac{\lambda_n^j}{\lambda_n^l}=\lambda_1\in (0,\infty).
\end{align*}
We now suppose that the time parameters diverge, that is,
\begin{align*}
\lim_{n\to \infty}\tfrac{|t_n^j(\lambda_n^j)^2-t_n^l(\lambda_n^l)^2|}{\lambda_n^j\lambda_n^l}=\infty;
\end{align*}
then we also have
\begin{align*}
\Bigl|t_n^l-t_n^j\Bigl(\tfrac{\lambda_n^j}{\lambda_n^l}\Bigr)^2\Bigr|
=\tfrac{|t_n^l(\lambda_n^l)^2-t_n^j(\lambda_n^j)^2|}{\lambda_n^l\lambda_n^j}\cdot\tfrac{\lambda_n^j}{\lambda_n^l}\to \infty \qtq{as} n\to\infty.
\end{align*}
Under this condition, \eqref{need11} follows from
\begin{align*}
\lambda_1^{\frac 12}\Bigl[e^{i\bigl(t_n^l-t_n^j\bigl(\tfrac{\lambda_n^j}{\lambda_n^l}\bigr)^2\bigr)\Delta}\phi\Bigr]\Bigl(\lambda_1 x+\tfrac{x_n^j-x_n^l}{\lambda_n^l}\Bigr)\rightharpoonup 0 \quad \text{weakly in } \dot H^1_x,
\end{align*}
which is an immediate consequence of the dispersive estimate.

Finally, we deal with the situation when
\begin{align*}
\tfrac{\lambda_n^j}{\lambda_n^l}\to \lambda_1\in(0, \infty), \quad \tfrac{t_n^l(\lambda_n^l)^2-t_n^j(\lambda_n^j)^2}{\lambda_n^j\lambda_n^l}\to t_1,
\qtq{but} \tfrac{|x_n^j-x_n^l|^2}{\lambda_n^j\lambda_n^l}\to \infty.
\end{align*}
Then we also have $t_n^l-t_n^j(\lambda_n^j)^2/(\lambda_n^l)^2\to \lambda_1t_1$.  In this case, the desired weak convergence follows from the easily proved assertion
\begin{align*}
\lambda_1^{\frac 12}e^{it_1\lambda_1\Delta}\phi(\lambda_1x+y_n)\rightharpoonup0  \quad \text{weakly in } \dot H^1_x,
\end{align*}
where
\begin{align*}
y_n:=\tfrac{x_n^j-x_n^l}{\lambda_n^l}=\tfrac{x_n^j-x_n^l}{\sqrt{\lambda_n^l\lambda_n^j}}\sqrt{\tfrac{\lambda_n^j}{\lambda_n^l}},
\end{align*}
which diverges to infinity as $n\to\infty$.

This completes the proof of the theorem.
\end{proof}


\section{Embedding the quintic NLS inside the cubic-quintic NLS}\label{SEC:6}

The next major milestone in proving the existence of minimal blowup solutions is the development of a nonlinear profile decomposition.  To do this, we must associate to each profile in \eqref{E:LPD} a solution of the cubic-quintic NLS.  In this section, we consider those profiles for which $\ld_n^j\to 0$ as $n\to \infty$.  What is special about these profiles is that the
long-time behaviour of the associated solutions to \eqref{3-5} can be deduced from the main result in \cite{CKSTT}, namely:

\begin{theorem}[Spacetime bounds for the defocusing quintic NLS, \cite{CKSTT}]\label{T:Iteam}
Let $u_0\in \dot H^1(\R^3)$.  Then there exists a unique global solution $u$ to
\begin{align}\label{QNLS}
i\partial_t u + \Delta u= |u|^4 u
\end{align}
with initial data $u(0)=u_0$.  Moreover, the solution $u$ satisfies
$$
\|\nabla u\|_{S^0(\R)}\leq C\big(\|u_0\|_{\dot H^1_x}\big).
$$
\end{theorem}

\begin{remark}
An easy consequence of Theorem~\ref{T:Iteam} is global spacetime bounds for solutions defined by their scattering states.  Specifically, given $u_+\in \dot H^1_x$ there exists a unique global solution $u$ to \eqref{QNLS} such that
$$
\|u(t) - e^{it\Delta}u_+\|_{\dot H^1_x}\to 0 \qtq{as} t\to \infty;
$$
moreover, the solution $u$ satisfies
$$
\|\nabla u\|_{S^0(\R)} \leq C\big( \|u_+\|_{\dot H^1_x}\big).
$$
A similar statement holds backward in time.
\end{remark}

This result is pertinent to us because profiles with $\ld_n^j\to 0$ are highly concentrated for $n$ large and so the cubic nonlinearity is relatively weak.  Indeed, this precisely corresponds to the energy-subcriticality of the cubic nonlinearity.  In this section we use a perturbative argument to show that this heuristic can be made rigorous.  The exact information we will extract is dictated by the needs of the proof of the Palais--Smale condition Proposition~\ref{prop: Palais-Smale}.

\begin{proposition}\label{thm: embedding}
Let $\{\ld_n\}_{n\in\mathbb{N}}\subset (0,\infty)$ be such that $\ld_n\to 0$, $\{t_n\}_{n\in \mathbb N}\subset \R$ such that either $t_n\equiv 0$ or
$t_n\to\pm\infty$, and let $\{x_n\}_{n\in\mathbb{N}}\subset \R^3$.  Given $\phi\in\dot{H}^1_x$, let
$\phi_n(x):=\ld_n^{-\frac12}[e^{it_n\Delta}P_{\geq \ld_n^\theta}\phi]\big(\frac{x-x_n}{\ld_n}\big)$ for some $0 < \theta < 1$.  Then for $n$ sufficiently large, the unique global solution $u_n$ to \eqref{3-5} with initial data $u_n(0)=\phi_n$ satisfies
\begin{align}\label{ZNLS0}
\|\nabla u_n\|_{S^0(\R)}\leq C\big(\| \phi\|_{\dot H^1_x}\big).
\end{align}
Furthermore, for every $\eps>0$, there exist $N_{\eps}\in \mathbb{N}$ and $\phi_\eps, \psi_{\eps}\in C^{\infty}_c(\R\times\R^3)$ such that
\begin{align}
\Big\|u_n(t,x)-\ld_n^{-\frac{1}{2}}\phi_{\eps}\big(\tfrac{t}{\ld_n^2}+t_n,\tfrac{x-x_n}{\ld_n}\big)\Big\|_{L^{10}_{t,x}(\R\times\R^3)}&<\eps,\label{eq: embedding}\\
\Big\|\nabla u_n(t,x)-\ld_n^{-\frac{3}{2}}\psi_{\eps}\big(\tfrac{t}{\ld_n^2}+t_n,\tfrac{x-x_n}{\ld_n}\big)\Big\|_{L^{\frac{10}{3}}_{t,x}(\R\times \R^3)}&<\eps, \label{eq: embedding_nabla}
\end{align}
for all $n\geq N_{\eps}$.
\end{proposition}

\begin{proof}
As \eqref{3-5} is space-translation invariant, without loss of generality we may assume that $x_n\equiv0$.

As noted previously, to prove the proposition we will use a perturbative argument.  Specifically, we will construct a solution $\tilde u_n$ to the defocusing energy-critical NLS that
is an approximate solution to \eqref{3-5} with initial data asymptotically matching $\phi_n$.  This solution will have good spacetime bounds thanks to Theorem~\ref{T:Iteam}.  Then using the stability result Proposition~\ref{Perturbation of {3-5}}, we will deduce that for $n$ sufficiently large, the solution $u_n$ inherits the spacetime bounds of $\tilde u_n$, thus proving~\eqref{ZNLS0}.

If $t_n\equiv 0$, we define $w_n$ and $w$ to be the solutions to \eqref{QNLS} with initial data $w_n(0)=P_{\geq\ld_n^\theta}\phi$ and $w(0) = \phi$, respectively.  If $t_n\to \pm\infty$, we define $w_n$ and $w$ to be the solutions to \eqref{QNLS} which scatter in $\dot{H}^1_x$ to $e^{it\Delta}P_{\geq\ld_n^\theta}\phi$ and $e^{it\Delta}\phi$, respectively, as $t\to\pm \infty$.  By Theorem~\ref{T:Iteam}, we have
\begin{align*}
\|\nabla w_n\|_{S^0(\R)} + \|\nabla w\|_{S^0(\R)}\leq C\big(\|\phi\|_{\dot H^1_x}\big) \quad\text{uniformly in } n\in \mathbb N.
\end{align*}

The stability result \cite[Lemma~3.10]{CKSTT} allows us to deduce that
\begin{align*}
\lim_{n\to \infty}\|\nabla(w_n-w)\|_{S^0(\R)}=0.
\end{align*}
Indeed, if $t_n\equiv 0$ this follows from the observation
\begin{equation*}
\|w_n(0)-w(0)\|_{\dot{H}^1_x}=\|P_{\geq\ld_n^{\theta}}\phi-\phi\|_{\dot{H}^1_x}\to 0 \qtq{as} n\to \infty.
\end{equation*}
If instead $t_n\to\pm\infty$, to be able to apply \cite[Lemma~3.10]{CKSTT}, it suffices to observe that
$$
\lim_{n\to \infty}\lim_{t\to \pm\infty} \|w_n(t) - w(t)\|_{\dot H^1_x} =0,
$$
which follows from the construction of $w_n$ and $w$ and the fact that $\ld_n\to 0$:
\begin{align*}
&\lim_{n\to \infty}\lim_{t\to \pm\infty} \|w_n(t) - w(t)\|_{\dot H^1_x}\\
&\leq \lim_{n\to \infty}\lim_{t\to \pm\infty} \Bigl[\|w_n(t) - e^{it\Delta}P_{\geq \ld_n^\theta} \phi \|_{\dot H^1_x} +\|w(t) - e^{it\Delta}\phi \|_{\dot H^1_x}
+\|P_{\geq\ld_n^{\theta}}\phi-\phi\|_{\dot{H}^1_x}\Bigr]=0.
\end{align*}

Next, by the Bernstein inequality,
$$
\|P_{\geq \ld_n^\theta} \phi \|_{L_x^2}\les \ld_n^{-\theta}\|\phi \|_{\dot H^1_x}
$$
and so the persistence of regularity result \cite[Lemma~3.12]{CKSTT} implies that
\begin{equation}\label{wnL2}
\|w_n\|_{S^0(\R)}\leq C\big(\| \phi\|_{\dot H^1_x}\big)\ld_n^{-\theta}.
\end{equation}

We are now in a position to introduce the approximate solutions $\tilde u_n$ to \eqref{3-5}.  For $n\in \mathbb N$, we define
\begin{equation*}
\tilde{u}_n(t, x):=\ld_n^{-\frac{1}{2}}w_n\big(\tfrac{t}{\ld_n^2}+t_n,\tfrac{x}{\ld_n}\big).
\end{equation*}
As \eqref{QNLS} is time-translation and scaling invariant, $\tilde u_n$ is also a solution to \eqref{QNLS} and
$$
\|\nabla\tilde u_n\|_{S^0(\R)} = \|\nabla w_n\|_{S^0(\R)} \leq C\big(\|\phi\|_{\dot H^1_x}\big).
$$
Moreover, $\tilde u_n(0)$ asymptotically matches the initial data $u_n(0)=\phi_n$; indeed,
$$
\|\tilde{u}_n(0)-\phi_n\|_{\dot{H}^1_x}= \|w_n(t_n)-e^{it_n\Delta}P_{\geq \ld_n^\theta}\phi\|_{\dot{H}^1_x}\to 0 \qtq{as} n\to \infty.
$$
To invoke the stability result Proposition~\ref{Perturbation of {3-5}} and deduce claim \eqref{ZNLS0}, it thus remains to show that $\tilde u_n$ has bounded mass (so \eqref{eq: perturb1} holds) and that $\tilde u_n$ is an approximate solution to \eqref{3-5}, both as $n\to \infty$.  By \eqref{wnL2},
\begin{equation*}
\|\tilde{u}_n\|_{S^0(\R)}= \ld_n \|w_n\|_{S^0(\R)}\leq C\big(\| \phi\|_{\dot H^1_x}\big)\ld_n^{1-\theta}\to 0 \qtq{as} n\to \infty.
\end{equation*}
This bounds the mass of $\tilde u_n$ and provides the key tool for bounding $e:=|\tilde{u}_n|^2\tilde{u}_n$:
\begin{align*}
\|\nabla e\|_{L^{\frac{10}{7}}_{t,x}(\R\times\R^3)}
&\les \|\nabla \tilde{u}_n\|_{L^{\frac{10}{3}}_{t,x}(\R\times\R^3)}
\|\tilde{u}_n\|_{L^{10}_{t,x}(\R\times\R^3)}\|\tilde{u}_n\|_{L^{\frac{10}{3}}_{t,x}(\R\times\R^3)}\leq C\big(\|\phi\|_{\dot H^1_x})\ld_n^{1-\theta},
\end{align*}
which converges to zero as $n\to\infty$.  This completes the proof of \eqref{ZNLS0}.

Finally, we turn to \eqref{eq: embedding} and \eqref{eq: embedding_nabla}.  For $\eps>0$, approximate $w$ by $\phi_{\eps}, \psi_{\eps} \in C_c^{\infty}(\R\times\R^3)$ such that
\begin{align*}
\|w-\phi_{\eps}\|_{L^{10}_{t,x}(\R\times\R^3)} < \tfrac{\eps}{3} \qtq{and} \|\nabla w-\psi_{\eps}\|_{L^{\frac{10}{3}}_{t,x}(\R\times\R^3)}< \tfrac{\eps}{3}
\end{align*}
and take $n$ sufficiently large so that
\begin{align*}
\|u_n-\tilde u_n\|_{L_{t,x}^{10}\cap L_t^{\frac{10}3}H^{1,\frac{10}3}_x} < \tfrac{\eps}{3} \qtq{and} \|w_n-w\|_{L_{t,x}^{10}\cap L_t^{\frac{10}3}H^{1,\frac{10}3}_x} < \tfrac{\eps}{3}.
\end{align*}
The two claims now follow easily from the triangle inequality.
\end{proof}


\section{Existence of a minimal blowup solution} \label{SEC:7}

As mentioned in the introduction, Theorem~\ref{thm: main} is equivalent to $L(D)<\infty$ for all $0<D<\infty$, where $L(D)$ is defined as in \eqref{QL}.
From the small data theory Proposition~\ref{thm: Small data scattering}, we know that $L(D)<\infty$ for $D\ll 1$.  Indeed, using also Lemma~\ref{LEM:small2} and
\eqref{E:norm equiv} we have
\begin{align}\label{E:small data bounds}
\|u\|_{L_{t,x}^{10}(\R\times\R^3)}+\|\nabla u\|_{L_{t,x}^{\frac{10}3}(\R\times\R^3)}\les E(u)^{1/2}
\end{align}
for all global solutions $u$ with $D(u)\ll 1$.  Thus if Theorem~\ref{thm: main} failed, there would exist a critical $D_c<\infty$ such that
\begin{align}\label{E:induct}
L(D)<\infty \qtq{for} D<D_c \qquad \qtq{and} \qquad L(D)=\infty \qtq{for} D>D_c.
\end{align}

The goal of this section is to prove the existence of a minimal blowup solution to \eqref{3-5}, that is, a solution $u$ to \eqref{3-5} such that
\begin{equation*}
D_c(u) = D_c \quad \text{and}\quad \|u\|_{L^{10}_{t,x}([0,\infty)\times\R^3)}=\|u\|_{L^{10}_{t,x}((-\infty,0]\times\R^3)}=\infty.
\end{equation*}
To extract this minimal counterexample to Theorem~\ref{thm: main}, we will prove a Palais--Smale condition for an increasingly bad sequence of almost counterexamples to our theorem.

\begin{proposition}[A Palais--Smale condition]\label{prop: Palais-Smale}
Let $\{u_n\}_{n\in\mathbb{N}}$ be a sequence of global solutions to \eqref{3-5} and $\{t_n\}_{n\in\mathbb{N}}\subset\R$ be such that
$\lim_{n\to\infty} D(u_n)=D_c$ and
\begin{align}\label{ZPS0}
\lim_{n\to\infty}\|u_n\|_{L^{10}_{t,x}([t_n,\infty)\times\R^3)}=\lim_{n\to\infty}\|u_n\|_{L^{10}_{t,x}((-\infty,t_n]\times\R^3)}=\infty.
\end{align}
Passing to a subsequence, there exists $\{x_n\}_{n\in\mathbb{N}}\subset\R^3$ such that $\{u_n(t_n, \, \cdot\,  +x_n)\}_{n\in\mathbb{N}}$ converges in $H^1(\R^3)$.
\end{proposition}

\begin{proof}
By time-translation invariance, we may take $t_n \equiv 0$.

By Proposition~\ref{P:all about D}, we know that $u_n(0)$ is bounded in $H^1_x$ and consequently we may apply the linear profile decomposition Theorem~\ref{thm:profile decomposition}
to write
\begin{equation}\label{eq: linear profile decomp}
u_n(0)=\sum_{j=1}^J\phi_n^j+w_n^J.
\end{equation}
To prove the proposition, we must show that there exists a single profile (i.e., $J^*=1$), that $\lambda_n^1\equiv 1$, $t_n^1\equiv 0$, and that the error $w_n^1$ converges to zero in $H^1_x$.

Passing to a further subsequence if necessary, we may assume $M(u_n)\to M_0$ and $E(u_n)\to E_0$.  In particular, $D_c=D(M_0,E_0)$.

Our first step is to introduce {\it nonlinear} profiles $v^j_n$  associated to each $\phi^j_n$.  Fix $j\geq 1$.  If $\ld_n^j\equiv 1$ and $t_n^j\equiv 0$, then we define $v^j$ to be the global solution to \eqref{3-5} with initial data $v^j(0)=\phi^j$.  If $\ld_n^j\equiv 1$ and $t_n^j\to\pm\infty$ as $n\to\infty$, then we define $v^j$ to be the global solution to \eqref{3-5} which scatters forward/backward in time to $e^{it\Delta}\phi^j$ in $H^1_x$.  In both cases, we define the global solutions
$$
v_n^j(t,x):=v^j(t+t_n^j, x-x_n^j).
$$
If $\ld_n^j\to 0$ as $n\to\infty$, we define $v_n^j$ to be the global solution to \eqref{3-5} with initial data $v_n^j(0)= \phi_n^j$ given by Proposition~\ref{thm: embedding}.

As a consequence of our construction, in all cases above we have
\begin{equation}\label{eq: approx v_n^j(0)}
\big\|v_n^j(0)-\phi_n^j\|_{H^1_x}\to 0 \qtq{as} n\to\infty.
\end{equation}
Thus, by the decoupling of mass \eqref{eq: profile decoupling L2} and energy \eqref{eq: profile decoupling H1}, for each finite $J\leq J^*$ we obtain
\begin{align}
\limsup_{n\to\infty} \sum_{j=1}^{J}M(v_n^j)+M(w_n^{J})&\leq M_0 \label{eq: mass decoup}\\
\limsup_{n\to\infty} \sum_{j=1}^{J}E(v_n^j)+E(w_n^{J})&\leq E_0. \label{eq: sumable E}
\end{align}
Note that by \eqref{E:norm equiv}, the summands in \eqref{eq: sumable E} are non-negative (for $n$ large); indeed, as the profiles $\phi^j$ are non-trivial we actually have $\liminf_{n\to \infty}E(v_n^j)>0$.

To continue, we discuss separately the following two cases:
\begin{align*}
& \text{Case 1:} \quad \sup_j \limsup_{n\to\infty} M(v_n^j)=M_0  \qtq{and}  \sup_j \limsup_{n\to\infty} E(v_n^j)=E_0\\
& \text{Case 2:} \quad \sup_j \limsup_{n\to\infty}  M(v_n^j)< M_0  \qtq{or}  \sup_j \limsup_{n\to\infty}  E(v_n^j)< E_0.
\end{align*}

\textbf{Case 1:}  By \eqref{eq: mass decoup}, \eqref{eq: sumable E}, and the non-negativity of the summands in \eqref{eq: sumable E}, in this case there exists only one profile (i.e., $J^* = 1$) and the decomposition \eqref{eq: linear profile decomp} simplifies to
\begin{equation}\label{eq: p-s5}
u_n(0)=\phi_n+w_n \qtq{with} \lim_{n\to \infty} \|w_n\|_{H^1_x}=0.
\end{equation}

If $\ld_n \equiv 1$ and $t_n\equiv 0$, then we obtain the desired compactness conclusion.  We will show that all other scenarios contradict \eqref{ZPS0}.

Suppose that $\ld_n\equiv 1$  and $t_n\to\infty$ as $n \to \infty$; the case when $t_n\to-\infty$ can be treated similarly.  Using \eqref{eq: p-s5} together with Strichartz and the monotone convergence theorem, we deduce that
\begin{align*}
\big\|e^{it\Delta}u_n(0)\big\|_{L^{10}_{t,x}([0,\infty)\times \R^3)}
&\leq \big\|e^{it\Delta}\phi_n\big\|_{L^{10}_{t,x}([0,\infty)\times \R^3)} +\big\|e^{it\Delta}w_n\big\|_{L^{10}_{t,x}([0,\infty)\times \R^3)}\\
&\les \big\|e^{it\Delta}\phi\big\|_{L^{10}_{t,x}([t_n,\infty)\times \R^3)} + \|w_n\|_{H^1_x} \to 0 \qtq{as} n\to \infty.
\end{align*}
Thus, using the stability result Proposition~\ref{Perturbation of {3-5}} with  $u = u_n$ and $\tilde u = e^{it\Delta}u_n(0)$, we can deduce that $u_n$ has negligible
$L^{10}_{t,x}$-norm on $[0,\infty)\times\R^3$.  This contradicts \eqref{ZPS0}.

Finally, suppose that $\ld_n\to 0$ as $n\to\infty$.  From Proposition~\ref{thm: embedding} we have
\begin{equation*}
\|\nabla v_n\|_{S^0(\R)} \leq C\big( \|\phi\|_{\dot H^1_x}\big)
\end{equation*}
while by construction,
\begin{equation*}
\|u_n(0)-v_n(0)\|_{H^1_x}=\|w_n\|_{H^1_x}\to 0  \text{ as } n\to\infty.
\end{equation*}
Thus, by the stability result Proposition~\ref{Perturbation of {3-5}}, we derive a contradiction to \eqref{ZPS0}.

\smallskip

\textbf{Case 2:}  In this case there exists $\eps>0$ so that
\begin{align*}
\sup_j \limsup_{n\to\infty} M(v_n^j)\leq M_0-\eps \quad \text{or}\quad \sup_j \limsup_{n\to\infty} E(v_n^j)\leq E_0-\eps.
\end{align*}
Then for each finite $J\leq J^*$, we must have $M(v_n^j) \leq M_0-\eps/2$ or $E(v_n^j) \leq E_0-\eps/2$ for all $1\leq j\leq J$ and $n$ sufficiently large.
Thus, by Proposition~\ref{P:all about D}(vi) and the inductive hypothesis \eqref{E:induct}, we get
\begin{align*}
\|v_n^j\|_{L^{10}_{t,x}(\R \times \R^3)}\les_{D_c, \eps}1
\end{align*}
for all $1\leq j\leq J$ and $n$ sufficiently large.  In fact, by \eqref{E:small data bounds}, \eqref{eq: sumable E}, and the finiteness of $D_c$, we have
\begin{align}\label{ZPS7}
\|v_n^j\|_{L^{10}_{t,x}(\R \times \R^3)} + \|\nabla v_n^j\|_{L^{\frac{10}{3}}_{t,x}(\R\times\R^3)}\les_{D_c, \eps} E(v_n^j)^\frac{1}{2}.
\end{align}
Combining this with the persistence of regularity result Lemma~\ref{LEM:small2}, we obtain
\begin{align}\label{eq: gen bounds v_n^j}
\|v_n^j\|_{L^{\frac{10}{3}}_{t,x}(\R\times\R^3)} \les_{D_c, \eps}M(v_n^j)^{\frac12}
\end{align}
for all $1\leq j\leq J$ and $n$ sufficiently large.

Next we define
\begin{equation}\label{ZPS9}
u_n^J(t):=\sum_{j=1}^Jv_n^j(t)+e^{it\Delta}w_n^J.
\end{equation}
We will prove that $u_n^J$ are increasingly accurate approximate solutions to \eqref{3-5} with uniform global spacetime bounds and that $u_n^J(0)$ asymptotically matches the initial data $u_n(0)$.  Then by the stability result Proposition~\ref{Perturbation of {3-5}}, the solutions $u_n$ must satisfy uniform global spacetime bounds, which contradicts \eqref{ZPS0}.

A key step in the plan described above is to prove that the nonlinear profiles $v_n^j$ decouple asymptotically (as $n\to\infty$).  Indeed, as \eqref{3-5} is a nonlinear equation, the sum of solutions is no longer a solution; however, if the solutions $v_n^j$ decouple asymptotically, then we expect that $u_n^J$ will be an approximate solution for $n$ large.  We will show that the asymptotic decoupling of $v_n^j$ is a consequence of the asymptotic orthogonality relation \eqref{eq: profile orthogonality}.

\begin{lemma}[Asymptotic decoupling of nonlinear profiles]\label{lemma: v^jv^k to 0}
If $j\neq k$ we have
\begin{align}\label{ZPS9a}
\lim_{n\to\infty}\Big[\|v_n^jv_n^k \|_{L^5_{t,x}} + \|v_n^j\nabla v_n^k\|_{L^{\frac52}_{t,x}}+\|\nabla v_n^j\nabla v_n^k\|_{L^{\frac53}_{t,x}}
+\|v_n^jv_n^k\|_{L^{\frac53}_{t,x}}\Big]=0,
\end{align}
where all spacetime norms are over $\R\times\R^3$.
\end{lemma}

\begin{proof}
We first prove
\begin{equation}\label{ZPS15}
\lim_{n\to\infty}\|v_n^j\nabla v_n^k\|_{L^{\frac{5}{2}}_{t,x}(\R\times\R^3)}=0;
\end{equation}
the first and third terms in \eqref{ZPS9a} follow in a similar manner.  We will discuss the fourth summand afterwards.

Observe that for any small $\delta>0$ there exist $\phi^j_\delta, \psi^k_\delta\in C^{\infty}_c(\R\times\R^3)$ such that
\begin{align*}
\Big\|v^j_n(t,x) - (\ld_n^j)^{-\frac12}\phi_\delta^j \Big(\tfrac{t}{(\ld_n^j)^2}+t_n^j, \tfrac{x-x_n^j}{\ld_n^j}\Big)\Big\|_{L^{10}_{t,x}} &\leq \delta,\\
\Big\|\nabla v^k_n(t,x) - (\ld_n^k)^{-\frac32}\psi_\delta^k \Big(\tfrac{t}{(\ld_n^k)^2}+t_n^k, \tfrac{x-x_n^k}{\ld_n^k}\Big)\Big\|_{L^{\frac{10}{3}}_{t,x}}&\leq \delta,
\end{align*}
for all $n$ sufficiently large.  When the spatial scale converges to zero, this follows from Proposition~\ref{thm: embedding}; in the other case
it follows directly from the definition of $v_n^j$ and $v_n^k$ and \eqref{ZPS7}.
Thus, by H\"older's inequality together with \eqref{ZPS7},
\begin{align*}
\|v_n^j\nabla v_n^k\|_{L^{\frac{5}{2}}_{t,x}}
\les_{D_c,\eps} \delta+ \Big\|(\ld_n^j)^{-\frac12}\phi_\delta^j \Big(\tfrac{t}{(\ld_n^j)^2}+t_n^j, \tfrac{x-x_n^j}{\ld_n^j}\Big) (\ld_n^k)^{-\frac32}\psi_\delta^k \Big(\tfrac{t}{(\ld_n^k)^2}+t_n^k, \tfrac{x-x_n^k}{\ld_n^k}\Big)\Big\|_{L_{t,x}^{\frac52}}.
\end{align*}
As $\delta>0$ is arbitrary, to complete the proof of \eqref{ZPS15} it remains to show that
\begin{align}\label{to show decoup}
\lim_{n\to \infty}\Big\|(\ld_n^j)^{-\frac12}\phi_\delta^j \Big(\tfrac{t}{(\ld_n^j)^2}+t_n^j, \tfrac{x-x_n^j}{\ld_n^j}\Big) (\ld_n^k)^{-\frac32}\psi_\delta^k \Big(\tfrac{t}{(\ld_n^k)^2}+t_n^k, \tfrac{x-x_n^k}{\ld_n^k}\Big)\Big\|_{L_{t,x}^{\frac52}}=0.
\end{align}

If $\ld_n^j\equiv 1$ and $\ld_n^k\equiv 1$, then the asymptotic orthogonality condition \eqref{eq: profile orthogonality} implies that the supports of the two functions in \eqref{to show decoup} become disjoint in spacetime for $n$ large.  Thus \eqref{to show decoup} holds in this case.

If $\ld_n^j\to 0$ and $\ld_n^k\equiv 1$, then by H\"older's inequality we estimate
\begin{align*}
\Big\|(\ld_n^j)^{-\frac12}\phi_\delta^j \Big(\tfrac{t}{(\ld_n^j)^2}+t_n^j, \tfrac{x-x_n^j}{\ld_n^j}\Big) &\psi_\delta^k \big(t+t_n^k, x-x_n^k\big)\Big\|_{L_{t,x}^{\frac52}}\les (\ld_n^j)^{\frac32} \|\phi_\delta^j\|_{L_{t,x}^{\frac52}}\|\psi_\delta^k\|_{L_{t,x}^{\infty}},
\end{align*}
which converges to zero as $n\to \infty$.  This proves \eqref{to show decoup} in this case.

If $\ld_n^j\equiv1$ and $\ld_n^k\to 0$, then by H\"older's inequality we get
\begin{align*}
\Big\|\phi_\delta^j \big(t+t_n^j,x-x_n^j\big) (\ld_n^k)^{-\frac32}&\psi_\delta^k \Big(\tfrac{t}{(\ld_n^k)^2}+t_n^k, \tfrac{x-x_n^k}{\ld_n^k}\Big)\Big\|_{L_{t,x}^{\frac52}}
\les(\ld_n^k)^{\frac12} \|\phi_\delta^j\|_{L_{t,x}^{\infty}}\|\psi_\delta^k\|_{L_{t,x}^{\frac52}},
\end{align*}
which converges to zero as $n\to \infty$, thus proving \eqref{to show decoup} in this case.

It remains to treat the case when $\ld_n^j\to 0$ and $\ld_n^k\to 0$.  If $\frac{\ld_n^j}{\ld_n^k} + \frac{\ld_n^k}{\ld_n^j} \to \infty$, then we estimate
\begin{align*}
\Big\|&(\ld_n^j)^{-\frac12}\phi_\delta^j \Big(\tfrac{t}{(\ld_n^j)^2}+t_n^j, \tfrac{x-x_n^j}{\ld_n^j}\Big) (\ld_n^k)^{-\frac32}\psi_\delta^k \Big(\tfrac{t}{(\ld_n^k)^2}+t_n^k, \tfrac{x-x_n^k}{\ld_n^k}\Big)\Big\|_{L_{t,x}^{\frac52}}\\
&\quad\les \min\Big\{\Big(\tfrac{\ld_n^j}{\ld_n^k}\Big)^{\frac32} \|\phi_\delta^j\|_{L_{t,x}^{\frac52}}\|\psi_\delta^k\|_{L_{t,x}^{\infty}}, \Big(\tfrac{\ld_n^k}{\ld_n^j}\Big)^{\frac12} \|\phi_\delta^j\|_{L_{t,x}^{\infty}}\|\psi_\delta^k\|_{L_{t,x}^{\frac52}}\Big\}\to 0\qtq{as} n\to \infty.
\end{align*}
If instead $\frac{\ld_n^j}{\ld_n^k}\to \ld_0\in (0, \infty)$, we use a change of variables to write
\begin{align*}
\Big\|&(\ld_n^j)^{-\frac12}\phi_\delta^j  \Big(\tfrac{t}{(\ld_n^j)^2}+t_n^j, \tfrac{x-x_n^j}{\ld_n^j}\Big) (\ld_n^k)^{-\frac32}\psi_\delta^k \Big(\tfrac{t}{(\ld_n^k)^2}+t_n^k, \tfrac{x-x_n^k}{\ld_n^k}\Big)\Big\|_{L_{t,x}^{\frac52}}\\
&\quad=\Big(\tfrac{\ld_n^j}{\ld_n^k}\Big)^{\frac32} \Big\|\phi_\delta^j \big(t, x\big)\psi_\delta^k \Big(\Big(\tfrac{\ld_n^j}{\ld_n^k}\Big)^2\Big(t-\tfrac{t_n^j(\ld_n^j)^2-t_n^k(\ld_n^k)^2}{(\ld_n^j)^2}\Big), \tfrac{\ld_n^j}{\ld_n^k}\Big(x+\tfrac{x_n^j-x_n^k}{\ld_n^j}\Big)\Big)\Big\|_{L_{t,x}^{\frac52}}.
\end{align*}
Using the asymptotic orthogonality condition \eqref{eq: profile orthogonality}, we see that either the temporal or the spatial supports of the functions on the right-hand side above become disjoint for $n$ large.  This proves \eqref{to show decoup} in this case.

We now turn to the fourth limit in \eqref{ZPS9a}.  We want to show that
\begin{equation}\label{ZPS17}
\lim_{n\to\infty}\|v_n^jv_n^k\|_{L^{\frac{5}{3}}_{t,x}}=0.
\end{equation}
When $\ld_n^j\equiv 1$ and $\ld_n^k\equiv 1$, this follows in the manner shown above, using \eqref{eq: gen bounds v_n^j} in place of \eqref{ZPS7}.
Suppose now that $\ld_n^j\to 0$; the case when $\ld_n^k\to 0$ can be treated similarly.  By \eqref{eq: approx v_n^j(0)}, \eqref{eq: gen bounds v_n^j} and Bernstein,
\begin{align*}
\|v_n^j\|_{L^{\frac{10}{3}}_{t,x}(\R\times\R^3)}
&\les_{D_c,\eps} \|v_n^j(0)\|_{L^2_x} \les_{D_c,\eps} (\ld_n^j)^{1-\theta} \|\phi^j\|_{\dot H^1_x}+ \|v_n^j(0)-\phi_n^j\|_{L^2_x} \to 0
\end{align*}
as $n\to \infty$.  Combining this with H\"older's inequality and \eqref{eq: gen bounds v_n^j} yields \eqref{ZPS17} in this case.
\end{proof}

The next three lemmas show that $u_n^J$ obeys the hypotheses of the stability result Proposition~\ref{Perturbation of {3-5}} for $n$ large enough.

\begin{lemma}[Asymptotic agreement of the initial data]\label{L:data match}
For any finite $J\leq J^*$,
$$
\lim_{n\to \infty} \|u_n^J(0)- u_n(0)\|_{H^1_x}=0.
$$
\end{lemma}

\begin{proof}
The result follows immediately from \eqref{eq: linear profile decomp}, \eqref{eq: approx v_n^j(0)}, and  \eqref{ZPS9}.
\end{proof}

\begin{lemma}[Uniform global spacetime bounds]\label{L:ST bounds}
We have
\begin{align}\label{E:ST bounds}
\sup_J \limsup_{n\to\infty}\Big[\|u_n^J\|_{L_{t,x}^{10}(\R\times\R^3)} + \|u_n^J\|_{L_t^{\frac{10}3}H^{1, \frac{10}3}_x(\R\times\R^3)}\Big]\les_{D_c,\eps} 1.
\end{align}
\end{lemma}

\begin{proof}
Fix $J\leq J^*$.  By Strichartz, Lemma~\ref{lemma: v^jv^k to 0}, \eqref{ZPS7}, \eqref{E:norm equiv}, and \eqref{eq: sumable E}, we obtain
\begin{align*}
\|u_n^J\|_{L^{10}_{t,x}}^2
&\les \sum_{j=1}^J\|v_n^j\|_{L^{10}_{t,x}}^2 + \sum_{j\neq k}\|v_n^jv_n^k\|_{L^5_{t,x}}+\|w_n^J\|_{\dot H^1_x}^2\notag\\
&\les_{D_c,\eps} \sum_{j=1}^J E(v_n^j) + \sum_{j\neq k}o(1)+ E(w_n^J)
\les_{D_c,\eps}1+J^2o(1) \qtq{as} n\to \infty.
\end{align*}
Similarly,
\begin{align*}
\|\nabla u_n^J\|_{L^{\frac{10}{3}}_{t,x}}^2
&\les \sum_{j=1}^J\|\nabla v_n^j\|_{L^{\frac{10}{3}}_{t,x}}^2 + \sum_{j\neq k}\|\nabla v_n^j\nabla v_n^k\|_{L^{\frac{5}{3}}_{t,x}}^2 + \|w_n^J\|_{\dot H^1_x}^2\\
&\les_{D_c,\eps} \sum_{j=1}^J E(v_n^j) + \sum_{j\neq k} o(1)+ E(w_n^J)
\les_{D_c,\eps}1+J^2o(1) \qtq{as} n\to \infty.
\end{align*}
Thus,
$$
\sup_J \limsup_{n\to\infty}\Big[\|u_n^J\|_{L^{10}_{t,x}(\R\times\R^3)}+\|\nabla u_n^J\|_{L^{\frac{10}3}_{t,x}(\R\times\R^3)}\Big]\les_{D_c,\eps}1.
$$
Finally, using \eqref{eq: gen bounds v_n^j} in place of \eqref{ZPS7}, and \eqref{eq: mass decoup} in place of \eqref{eq: sumable E}, and arguing as above we also obtain
\begin{align*}
\sup_J \limsup_{n\to\infty}\|u_n^J\|_{L^{\frac{10}{3}}_{t,x}(\R\times\R^3)}\les_{D_c,\eps}1.
\end{align*}

This completes the proof of the lemma.
\end{proof}

\begin{lemma}[Approximate solution]\label{LEM:ZPS}
With $F(z):=|z|^4z-|z|^2z$, we have
\begin{equation}\label{eq: u_n^J approx solution}
\lim_{J \to J^*} \limsup_{n\to\infty}
\Big\|\nabla\Big[i\partial_t u_n^J+\Delta u_n^J-F(u_n^J)\Big]\Big\|_{L^{\frac{10}{7}}_{t,x}(\R\times\R^3)}=0.
\end{equation}
\end{lemma}

\begin{proof}
By \eqref{ZPS9}, we have
\begin{align}\label{ZPS12}
i\partial_t u_n^J+\Delta u_n^J -F(u_n^J) &=\sum_{j=1}^JF(v_n^j) -F(u_n^J) \notag\\
&= \sum_{j=1}^JF(v_n^j) - F\Big(\sum_{j=1}^Jv_n^j\Big) + F\big(u_n^J - e^{it\Delta}w_n^J\big) - F(u_n^J).
\end{align}

By H\"older's inequality, \eqref{ZPS7}, \eqref{eq: gen bounds v_n^j}, and Lemma~\ref{lemma: v^jv^k to 0}, we obtain
\begin{align*}
\Big\|\nabla\Big[F\Big(\sum_{j=1}^Jv_n^j\Big)-\sum_{j=1}^JF(v_n^j)\Big]\Big\|_{L^{\frac{10}{7}}_{t,x}}
&\lesssim_J \sum_{j\neq k}^J \Bigl[  \|v_n^k\|_{L^{10}_{t,x}}^3\|v_n^k\nabla v_n^j\|_{L^{\frac{5}{2}}_{t,x}}
+\|v_n^j\|_{L^{10}_{t,x}}^3\|v_n^k\nabla v_n^{j}\|_{L^{\frac{5}{2}}_{t,x}} \\
&\qquad \quad +\|v_n^k\|_{L^{\frac{10}{3}}_{t,x}}\|v_n^k\nabla v_n^j\|_{L^{\frac{5}{2}}_{t,x}}
+\|v_n^j\|_{L^{\frac{10}{3}}_{t,x}}\|v_n^k\nabla v_n^{j}\|_{L^{\frac{5}{2}}_{t,x}} \Bigr]\\
&\les_{J,D_c, \eps} o(1) \qtq{as} n\to \infty.
\end{align*}

It thus remains to prove
\begin{align}\label{eq: approx u_n^J 2}
\lim_{J\to J^{\ast}}\limsup_{n\to\infty}\Big\|\nabla\Big[F(u_n^J-e^{it\Delta}w_n^J)-F(u_n^J)\Big]\Big\|_{L^{\frac{10}{7}}_{t,x}(\R\times\R^3)}=0.
\end{align}
By the chain rule and H\"older's inequality, we estimate
\begin{align*}
&\Big\|\nabla\Big[F(u_n^J-e^{it\Delta}w_n^J)-F(u_n^J)\Big]\Big\|_{L^{\frac{10}{7}}_{t,x}}\\
&\lesssim \|e^{it\Delta}w_n^J\|^4_{L^{10}_{t,x}}\|\nabla e^{it\Delta}w_n^J\|_{L^{\frac{10}{3}}_{t,x}}
+\|e^{it\Delta}w_n^J\|^4_{L^{10}_{t,x}}\|\nabla u_n^J\|_{L^{\frac{10}{3}}_{t,x}}\\
&\quad+\|u_n^J\|_{L^{10}_{t,x}}^3\|u_n^J\nabla e^{it\Delta}w_n^J\|_{L^{\frac{5}{2}}_{t,x}}
+\|u_n^J\|_{L^{10}_{t,x}}^3\|e^{it\Delta}w_n^J\|_{L^{10}_{t,x}}\|\nabla u_n^J\|_{L^{\frac{10}{3}}_{t,x}}\\
&\quad+\|e^{it\Delta}w_n^J\|_{L^{10}_{t,x}}\|\nabla e^{it\Delta}w_n^J\|_{L^{\frac{10}{3}}_{t,x}}\|e^{it\Delta}w_n^J\|_{L^{\frac{10}{3}}_{t,x}}
+\|e^{it\Delta}w_n^J\|_{L^{10}_{t,x}}\|e^{it\Delta}w_n^J\|_{L^{\frac{10}{3}}_{t,x}}\|\nabla u_n^J\|_{L^{\frac{10}{3}}_{t,x}}\\
&\quad+\|u_n^J\|_{L^{\frac{10}{3}}_{t,x}}\|u_n^J\nabla e^{it\Delta}w_n^J\|_{L^{\frac{5}{2}}_{t,x}}
+\|u_n^J\|_{L^{\frac{10}{3}}_{t,x}}\|e^{it\Delta}w_n^J\|_{L^{10}_{t,x}}\|\nabla u_n^J\|_{L^{\frac{10}{3}}_{t,x}}.
\end{align*}
In view of \eqref{eq: profile reminder}, the Strichartz inequality together with \eqref{eq:  mass decoup} and \eqref{eq:  sumable E}, and Lemma~\ref{L:ST bounds}, to prove \eqref{eq: approx u_n^J 2} it suffices to show
\begin{equation}\label{ZPS18}
\lim_{J\to J^{\ast}}\limsup_{n\to\infty}\|u_n^J\nabla e^{it\Delta}w_n^J\|_{L^{\frac{5}{2}}_{t,x}(\R\times\R^3)}=0.
\end{equation}

By the triangle inequality, the Strichartz inequality combined with \eqref{eq:  sumable E}, and~\eqref{eq: profile reminder},
\begin{align*}
&\lim_{J\to J^{\ast}}\limsup_{n\to\infty}\|u_n^J\nabla e^{it\Delta}w_n^J\|_{L^{\frac{5}{2}}_{t,x}}\\
&\leq \lim_{J\to J^{\ast}}\limsup_{n\to\infty}\Big\|\Big(\sum_{j=1}^Jv_n^j\Big)\nabla e^{it\Delta}w_n^J\Big\|_{L^{\frac{5}{2}}_{t,x}}
+\lim_{J\to J^{\ast}}\limsup_{n\to\infty} \|e^{it\Delta}w_n^J\|_{L^{10}_{t,x}}\|\nabla e^{it\Delta}w_n^J\|_{L^{\frac{10}{3}}_{t,x}}\\
&\leq \lim_{J\to J^{\ast}}\limsup_{n\to\infty}\Big\|\Big(\sum_{j=1}^Jv_n^j\Big)\nabla e^{it\Delta}w_n^J\Big\|_{L^{\frac{5}{2}}_{t,x}}.
\end{align*}
By \eqref{ZPS7} and Lemma~\ref{lemma: v^jv^k to 0},
\begin{align*}
\Big\|\sum_{j=J'}^Jv_n^j\Big\|_{L^{10}_{t,x}}^2
\les \sum_{j=J'}^J\|v_n^j\|_{L^{10}_{t,x}}^2 + \sum_{j\neq k}\|v_n^jv_n^k\|_{L^5_{t,x}}
\les \sum_{j=J'}^J E(v_n^j) + \sum_{j\neq k}o(1) \qtq{as} n\to \infty.
\end{align*}
Thus, by \eqref{eq: sumable E}, for any $\eta>0$ there exists a finite $J'=J'(\eta)$ such that
$$
\limsup_{n\to\infty}\Big\|\sum_{j=J'}^Jv_n^j\Big\|_{L^{10}_{t,x}(\R\times\R^3)}\leq \eta
$$
for all $J\leq J^*$.  Combining this with H\"older, Strichartz, and \eqref{eq:  sumable E}, we see that
$$
\lim_{J\to J^{\ast}}\limsup_{n\to\infty}\Big\|\Big(\sum_{j=J'}^Jv_n^j\Big)\nabla e^{it\Delta}w_n^J\Big\|_{L^{\frac{5}{2}}_{t,x}(\R\times\R^3)}\les \eta.
$$
As $\eta>0$ is arbitrary, to prove \eqref{ZPS18} (and so \eqref{eq: approx u_n^J 2}), it remains to show
\begin{equation}\label{eq: remainder p-s}
\lim_{J\to J^{\ast}}\lim_{n\to\infty}\|v_n^j\nabla e^{it\Delta}w_n^J\|_{L^{\frac{5}{2}}_{t,x}(\R\times\R^3)}=0 \qtq{for all} 1\leq j< J'.
\end{equation}

We will present the details for \eqref{eq: remainder p-s} only for those $j$'s for which $\ld_n^j\to 0$ as $n\to\infty$; the proof for $j$'s with $\ld_n^j\equiv 1$ is analogous.  Fix therefore $1\leq j< J'$, such that $\ld_n^j\to 0$ as $n\to\infty$.  By Proposition~\ref{thm: embedding}, for any $\delta>0$ there exists
$\phi_{\delta}^j\in C_c^{\infty}(\R\times\R^3)$ with compact support $K_{\delta}^j$ such that
\begin{align*}
\Big\|v_n^j- (\ld_n^j)^{-\frac12}\phi_{\delta}^j\Big(\tfrac{t}{(\ld_n^j)^2}+t_n^j, \tfrac{x-x_n^j}{\ld_n^j}\Big)\Big\|_{L^{10}_{t,x}}\leq\delta.
\end{align*}
Let $\tilde w_n^J(t,x):= (\ld_n^j)^{1/2} [e^{it\Delta} w_n^J]( (\ld_n^j)^2(t-t_n^j), \ld_n^j x+ x_n^j)$; by Strichartz, \eqref{eq: sumable E}, and \eqref{E:norm equiv} we have
$$
\|\nabla \tilde w_n^J\|_{L_{t,x}^{\frac{10}3}} = \|\nabla e^{it\Delta} w_n^J\|_{L_{t,x}^{\frac{10}3}}\les_{D_c, \eps} 1 \qtq{and}
\|\tilde w_n^J\|_{L_{t,x}^{10}} = \|e^{it\Delta} w_n^J\|_{L_{t,x}^{10}}\les_{D_c, \eps} 1.
$$
Thus, by H\"older's inequality, commuting the dilation through the propagator, and applying Lemma~\ref{L:Keraani3.7},
\begin{align*}
\|v_n^j\nabla e^{it\Delta} w_n^J\|_{L^{\frac{5}{2}}_{t,x}}
&\les \delta \|\nabla e^{it\Delta} w_n^J\|_{L_{t,x}^{\frac{10}3}}+\big\| \phi^j_\delta \nabla \tilde w_n^J\big\|_{L^{\frac{5}{2}}_{t,x}}\\
&\les_{D_c, \eps}\delta + C(\delta) \|\nabla \tilde w_n^J\|_{L_{t,x}^{\frac{10}3}}^{\frac12} \|\nabla \tilde w_n^J\|_{L_{t,x}^2(K_\delta^j)}^{\frac12} \\
&\les_{D_c, \eps}\delta + C(\delta,|K_\delta^j|)\|e^{it\Delta}w_n^J\|_{L_{t,x}^{10}}^{\frac16} \|\nabla w_n^J\|_{L_x^2}^{\frac13}.
\end{align*}
As $\delta>0$ is arbitrary, \eqref{eq: remainder p-s} follows by invoking \eqref{eq: profile reminder}.
\end{proof}

Writing Duhamel's formula for $u_n^J$, an application of the Strichartz inequality combined with Lemmas~\ref{L:data match}, \ref{L:ST bounds}, and \ref{LEM:ZPS}, leads quickly to
$$
\limsup_{n\to \infty}\|u_n^J\|_{L_t^\infty H^1_x(\R\times\R^3)}\lesssim_{D_c,\eps}1 \quad\text{for all $J$ sufficiently large.}
$$
Thus $u_n^J$ satisfies the hypotheses of the stability result Proposition~\ref{Perturbation of {3-5}} for all $n$ and $J$ sufficiently large, which allows us to deduce that
$$
\|u_n\|_{L_{t,x}^{10}(\R\times\R^3)}\les_{D_c,\eps}1 \quad\text{for all $n$ sufficiently large.}
$$
This contradicts \eqref{ZPS0}, thus showing that Case~2 cannot occur.  This completes the proof of Proposition~\ref{prop:  Palais-Smale}.
\end{proof}

With the Palais--Smale condition in place, we are now ready to show the existence of a minimal blowup solution.

\begin{theorem}[Existence of a minimal blowup solution]\label{prop: existence min blowup sol}
Suppose that Theorem~\ref{thm: main} failed.  Then there exists $0<D_c<\infty$ and a global solution $u$ to \eqref{3-5} satisfying $D(u)=D_c$ that blows up in both time directions in the sense that
$$
\|u\|_{L^{10}_{t,x}([0,\infty)\times\R^3)}= \|u\|_{L^{10}_{t,x}((-\infty,0]\times\R^3)}=\infty.
$$
Moreover, $u$ is almost periodic in $H^1_x$ modulo translations, that is, the orbit $\{u(t): \,t\in \R\}$ is precompact in $H^1_x$ modulo translations.
\end{theorem}

\begin{proof}
Suppose that Theorem \ref{thm: main} fails.   As discussed at the beginning of this section, this implies the existence of a critical $0<D_c<\infty$ such that
$$
L(D)<\infty \qtq{if} D<D_c \qquad \qtq{and} \qquad L(D)=\infty \qtq{if} D>D_c,
$$
where $L(D)$ is defined as in \eqref{QL}.  Thus, there exists a sequence of global solutions $\{u_n\}_{n\in\mathbb{N}}$ to \eqref{3-5} such that $D(u_n)\to D_c$ and
\begin{equation*}
\lim_{n\to\infty}\|u_n\|_{L^{10}_{t,x}(\R\times\R^3)}=\infty.
\end{equation*}
Let  $t_n\in\R$  be the `median' time of the $L^{10}_{t,x}$-norm of $u_n$ so that
$$
\|u_n\|_{L^{10}_{t,x}((-\infty, t_n] \times\R^3)}= \|u_n\|_{L^{10}_{t,x}([t_n, \infty) \times\R^3)} \to \infty \qtq{as} n\to \infty.
$$
As \eqref{3-5} is time-translation invariant, without loss of generality we may assume  that $t_n\equiv 0$.  Then
\begin{equation}\label{760}
\lim_{n\to\infty}\|u_n\|_{L^{10}_{t,x}((-\infty,0]\times\R^3)} =\lim_{n\to\infty}\|u_n\|_{L^{10}_{t,x}([0,\infty)\times\R^3)}=\infty.
\end{equation}
Invoking the Palais--Smale condition Proposition \ref{prop: Palais-Smale} and passing to a subsequence, we find $u_0\in H^1_x$ such that $u_n(0)$ converges to $u_0$ in $H^1_x$ modulo translations.  Using the space-translation invariance of \eqref{3-5} and modifying $u_n$ appropriately, we obtain that $u_n(0)\to u_0$ in $H^1_x$.  In particular, we have $D(u_0)=D_c$.

Let $u$ be the global solution to \eqref{3-5} with $u(0)=u_0$.  Combining \eqref{760} and the stability result Proposition~\ref{Perturbation of {3-5}}, we deduce that
$$
\|u\|_{L^{10}_{t,x}([0,\infty)\times\R^3)}= \|u\|_{L^{10}_{t,x}((-\infty,0]\times\R^3)}=\infty.
$$

Finally, we prove that $u$ is almost periodic in $H^1_x$ modulo translations.  Let  $\{u(\tau_n)\}_{n\in \mathbb N}$ be an arbitrary sequence in the orbit $\{u(t): \,t\in\R\}$.  As the $L^{10}_{t, x}$-norm of $u$ blows up both forward and backward in time and $u$ is locally in $L^{10}_{t,x}$, we have that
\begin{equation*}
\|u\|_{L^{10}_{t,x}((-\infty,\tau_n]\times\R^3)}=\|u\|_{L^{10}_{t,x}([\tau_n,\infty)\times\R^3)}=\infty.
\end{equation*}
The claim now follows by applying Proposition \ref{prop: Palais-Smale} to $\{u(\tau_n)\}$.

This completes the proof of the theorem.
\end{proof}

\begin{remark}\label{REM:XX}\rm
A simple consequence of the almost periodicity in $H^1_x$ (modulo translations) of the solution $u$ constructed in Theorem~\ref{prop: existence min blowup sol} together with Gagliardo--Nirenberg and Sobolev embedding inequalities is the following: there exists $x:\R\to\R^3$ so that for all $\eta>0$ there exists $C(\eta)>0$ such that
\begin{align}\label{Zmin2}
\sup_{t\in\R}\int_{|x-x(t)|>C(\eta)} |\nabla u(t,x)|^2+|u(t,x)|^6+|u(t,x)|^4+|u(t,x)|^2 \,dx\leq\eta.
\end{align}
\end{remark}


\section{Impossibility of minimal blowup solutions}\label{SEC:8}

In this section, we complete the proof of Theorem~\ref{thm: main}.  The argument is similar to those in \cite{DHR08,Berbec}.

Arguing by contradiction, we saw that failure of Theorem~\ref{thm: main} implies the existence of a minimal blowup solution $u$ as in Theorem~\ref{prop: existence min blowup sol}.
By Proposition~\ref{P:all about D},  $D(u) = D_c < \infty$ implies that $V(u(t)) > 0$ for all $t\in \R$.  We will show that in fact the virial $V(u(t))$ is bounded away from zero uniformly for $t\in \R$; this will then be used to derive a contradiction to the fact that the solution $u$ is global in time.

We start by showing that the momentum of a minimal blowup solution must be zero.  This is intuitively clear since a traveling solution must spend energy in order to travel; however, a minimal blowup solution must use all its energy to drive the blowup.

\begin{proposition}\label{prop: P(u)=0}
Let $u$ be a minimal blowup solution as in Theorem~\ref{prop: existence min blowup sol}.  Then the momentum of $u$ must be zero, that is,
\begin{equation*}
P(u):= 2\Im\int_{\R^3}\overline{u(t,x)}\nabla u(t,x)\, dx=0.
\end{equation*}
\end{proposition}

\begin{proof}
For $\xi\in\R^3$, let $\tilde u$ be the global solution to \eqref{3-5} given by the Galilei boost
\begin{equation*}
\tilde{u}(t,x):=e^{ix\cdot\xi-it|\xi|^2}u(t,x-2\xi t).
\end{equation*}
Note that $M(\tilde{u})=M(u)$ and
\begin{equation*}
E(\tilde{u})=E(u)+\tfrac{1}{2}\xi\cdot P(u)+\tfrac{1}{2}|\xi|^2M(u).
\end{equation*}

If the momentum $P(u)$, which is a conserved quantity, is not zero, then taking $\xi:=-\frac{P(u)}{2M(u)}$ we obtain
\begin{equation*}
E(\tilde{u})=E(u)-\frac{|P(u)|^2}{8M(u)}<E(u).
\end{equation*}
Thus, by Proposition~\ref{P:all about D}, we must have that $D(\tilde{u})<D(u)=D_c$.  As
$$
\|\tilde{u}\|_{L^{10}_{t,x}([0, \infty)\times\R^3)}\!=\!\|u\|_{L^{10}_{t,x}([0, \infty)\times\R^3)}\!=\!\infty\!=\!\|u\|_{L^{10}_{t,x}((-\infty, 0]\times\R^3)}\!=\!\|\tilde{u}\|_{L^{10}_{t,x}((-\infty, 0]\times\R^3)},
$$
this contradicts the minimality of $u$ as a blowup solution.
\end{proof}

Relying on Proposition~\ref{prop: P(u)=0}, we next show that the spatial center function $x(t)$ from \eqref{Zmin2} is $o(|t|)$ as $t\to\pm\infty$.

\begin{proposition}[Control of $x(t)$]\label{prop: control x(t)}
Let $u$ be a minimal blowup solution as in Theorem~\ref{prop: existence min blowup sol}.  Then the spatial center function $x(t)$ from \eqref{Zmin2}
satisfies
\begin{equation}\label{Z71}
|x(t)|=o(|t|) \qtq{as} t\to\pm\infty.
\end{equation}
\end{proposition}

\begin{proof}
By the space-translation invariance of \eqref{3-5}, we may assume that $x(0)=0$.

We will prove \eqref{Z71} for $t\to \infty$; the argument in the negative time direction is similar.  We argue by contradiction.  Suppose that there exist $\delta>0$ and a sequence $t_n\to\infty$ such that
\begin{equation*}
|x(t_n)|>\delta t_n \quad \text{for all } n\geq 1.
\end{equation*}
We note that standard arguments show that $x(t)$ is bounded on compact sets; indeed, the local constancy property (cf. (5.36) in \cite[Lemma 5.18]{ClayNotes}) shows that $|x(t)|=O(t)$ for $t$ large.  Consequently, for $n\in \mathbb N$ there exists $T_n\in [0,t_n]$ such that
\begin{equation}\label{Z72}
\sup_{t\in[0,t_n]}|x(t)|\leq |x(T_n)|+1.
\end{equation}

Now let $\eta$ be a small parameter to be chosen later.  By \eqref{Zmin2},
\begin{equation}\label{eq: almost periodic}
\sup_{t\in\R}\int_{|x-x(t)|>C(\eta)}\big[|\nabla u(t,x)|^2+| u(t,x)|^2\big]\,dx\leq\eta.
\end{equation}
For $n\in \mathbb N$, we define
\begin{equation}\label{eq: R_n}
R_n:=C(\eta)+\sup_{t\in [0,T_n]}|x(t)|.
\end{equation}
Note that by construction, $C(\eta) \leq R_n\leq C(\eta)+|x(T_n)|+1$.

Next, with $\phi \in C^\infty_c(\R)$ such that
\begin{align}\label{Z72a}
\phi (r)=
\begin{cases}
1, \quad \text{if } r\leq 1,\\
0, \quad \text{if } r\geq 2,
\end{cases}
\end{align}
we define a truncated center of mass
\begin{equation*}
X_R(t):=\int_{\R^3}x\phi\big(\tfrac{|x|}{R}\big)|u(t,x)|^2\,dx \qtq{for} R>0.
\end{equation*}

By \eqref{eq: almost periodic} and \eqref{eq: R_n},
\begin{align}\label{Z73}
|X_{R_n}(0)|
&\leq \Big|\int_{|x|\leq C(\eta)}x\phi\big(\tfrac{x}{R_n}\big)|u(0,x)|^2\,dx\Big|+\Big|\int_{|x|\geq C(\eta)}x\phi\big(\tfrac{x}{R_n}\big)|u(0,x)|^2\,dx\Big| \notag\\
&\leq C(\eta)M(u)+2\eta R_n\leq C(\eta)\big(M(u) + 2\eta\big)+ 2\eta|x(T_n)|+ 2\eta.
\end{align}
On the other hand, by the triangle inequality, \eqref{eq: almost periodic}, and \eqref{eq: R_n}, we have
\begin{align}\label{Z74}
|X_{R_n}(T_n)|
&\geq |x(T_n)|M(u)-|x(T_n)|\Big|\int_{\R^3} \Big(1-\phi\big(\tfrac{x}{R_n}\big)\Big)|u(T_n,x)|^2\,dx\Big| \notag\\
&\qquad \qquad \qquad \quad\!-\Big|\int_{|x-x(T_n)|\leq C(\eta)}\big(x-x(T_n)\big)\phi\big(\tfrac{x}{R_n}\big)|u(T_n,x)|^2\,dx\Big|\notag\\
&\qquad \qquad \qquad \quad\!-\Big|\int_{|x-x(T_n)|\geq C(\eta)}\big(x-x(T_n)\big)\phi\big(\tfrac{x}{R_n}\big)|u(T_n,x)|^2\,dx\Big|\notag\\
&\geq \big(M(u) -\eta\big)|x(T_n)|-C(\eta)M(u)-\eta \big(2R_n+|x(T_n)|\big)\notag\\
&\geq \big(M(u)-4\eta\big) |x(T_n)|-C(\eta)\big(M(u)+2\eta\big)-2\eta.
\end{align}
Hence taking $\eta \leq \frac18M(u)$, \eqref{Z73} and \eqref{Z74} yield that
\begin{equation}\label{eq: X(T_n)-X(0)}
|X_{R_n}(T_n)-X_{R_n}(0)|\gtrsim_{M(u)} |x(T_n)|-\tilde C(\eta).
\end{equation}

On the other hand, a direct computation using Proposition~\ref{prop: P(u)=0} gives
\begin{align*}
\partial_tX_R(t)&=2\Im\int_{\R^3}\Big[\phi\big(\tfrac{|x|}{R}\big)-1\Big]\nabla u(t,x)\overline{u(t,x)}\,dx\\
&\quad+2\Im\int_{\R^3}\tfrac{x}{|x|R}\phi ' \big(\tfrac{x}{R}\big)\big(x\cdot \nabla u\big)(t,x)\overline{u(t,x)}\,dx.
\end{align*}
Therefore, by the Cauchy--Schwarz inequality, \eqref{eq: almost periodic}, and \eqref{eq: R_n}, we obtain
\begin{align*}
|\partial_tX_{R_n}(t)|\leq 6\eta \qtq{for all} t\geq 0
\end{align*}
and so, by the fundamental theorem of calculus,
\begin{equation}
|X_{R_n}(T_n)-X_{R_n}(0)|\leq 6\eta T_n.
\label{Z75}
\end{equation}

Combining \eqref{eq: X(T_n)-X(0)} and \eqref{Z75} with \eqref{Z72}, we obtain
\begin{align*}
|x(t_n)|-\tilde C(\eta)\les_{M(u)} 6\eta t_n.
\end{align*}
Recalling that $|x(t_n)|>\delta t_n$ and taking $\eta$ even smaller if necessary (depending on $M(u)$ and $\delta$), we derive a contradiction by letting $n\to \infty$.
\end{proof}

Our next result is a uniform in time lower bound for $V(u(t))$, whenever $u$ is a minimal blowup solution as in Theorem~\ref{prop: existence min blowup sol}.

\begin{proposition}
\label{prop: virial bounded away from zero}
Let $u$ be a minimal blowup solution as in Theorem~\ref{prop: existence min blowup sol}.  Then, there exits $\delta>0$
such that
\begin{equation*}
V(u(t))=\int_{\R^3} |\nabla u(t,x)|^2-\tfrac{3}{4}|u(t,x)|^4+|u(t,x)|^6 \, dx\geq \delta \qtq{for all} t\in \R.
\end{equation*}
\end{proposition}

\begin{proof}
Suppose that no such $\delta>0$ existed. Then there exists a sequence $\{t_n\}_{n\in\mathbb{N}}\subset\R$ such that $V(u(t_n))\to 0$.
By almost periodicity, there exists $v_0\in H^1_x$ such that $u(t_n)$ converges to $v_0$ in $H^1_x$ modulo translations.  Combining this with the continuity of the
functionals $D$ and $V$, we obtain that
$$
D(v_0)=D(u)=D_c\in(0,\infty) \qtq{and} V(v_0)=\lim_{n\to \infty} V(u(t_n))=0,
$$
which contradicts Proposition~\ref{P:all about D}(iv).
\end{proof}

We now have all the necessary ingredients to complete the proof of Theorem~\ref{thm: main}.

\begin{proof}[Proof of Theorem~\ref{thm: main}]
Arguing by contradiction, we assume that the conclusion of Theorem~\ref{thm: main} does not hold.  Then, by Theorem~\ref{prop: existence min blowup sol}, we can find an almost periodic minimal blowup solution $u$.  In the following, we will show that such a solution cannot exist.

With $\phi$ as in \eqref{Z72a} and $R\geq 1$, we define a truncation of the weighted momentum:
\begin{equation*}
A_R(t)=2\Im\int_{\R^3} \phi\big(\tfrac{|x|}{R}\big)\overline{u(t,x)}\,x\cdot \nabla u(t,x)\,dx.
\end{equation*}
By the Cauchy--Schwarz inequality and Proposition~\ref{P:all about D}, we obtain
\begin{equation}\label{eq: upper bound V_R}
|A_R(t)|\leq 4R\|u\|_{L^{\infty}_tL^2_x(\R\times\R^3)}\|\nabla u\|_{L^{\infty}_tL^2_x(\R\times\R^3)}\les_{D_c}R.
\end{equation}

On the other hand, a straightforward computation yields
\begin{align*}
\partial_t A_R(t)
&=-\int_{\R^3} \Big[\tfrac{8}{|x|R}\phi '\big(\tfrac{|x|}{R}\big)+\tfrac{7}{R^2}\phi ''\big(\tfrac{|x|}{R}\big)+\tfrac{|x|}{R^3}\phi '''\big(\tfrac{|x|}{R}\big)\Big]|u(t,x)|^2\,dx\\
&\quad +4\int_{\R^3}\phi \big(\tfrac{|x|}{R}\big) |\nabla u(t,x)|^2\,dx + 4\int_{\R^3} \tfrac{|x\cdot \nabla u(t,x)|^2}{|x|R}\phi '\big(\tfrac{|x|}{R}\big)\,dx\\
&\quad +4\int_{\R^3} \Big[\phi \big(\tfrac{|x|}{R}\big)+\tfrac{|x|}{3R}\phi '\big(\tfrac{|x|}{R}\big)\Big]|u(t,x)|^6\,dx\\
&\quad -4\int_{\R^3}\Big[\tfrac{3}{4}\phi \big(\tfrac{|x|}{R}\big)+\tfrac{|x|}{4R}\phi '\big(\tfrac{|x|}{R}\big)\Big]|u(t,x)|^4\,dx.
\end{align*}
Therefore,
\begin{align*}
\partial_t A_R(t)
& \geq 4\int_{\R^3} |\nabla u(t,x)|^2-\tfrac{3}{4}|u(t,x)|^4+|u(t,x)|^6 \,dx\\
&\quad+O\Big(\tfrac{1}{R^2}\int_{R\leq |x|\leq 2R}|u(t,x)|^2\,dx\Big)\\
&\quad +O\Big(\int_{|x|\geq R}\big[|\nabla u|^2+|u|^6+|u|^4\big](t,x)\,dx\Big)\\
&=4V(u(t))+O\Big(\int_{|x|\geq R}\big[|\nabla u|^2+|u|^6+|u|^4+|u|^2\big](t,x)\,dx\Big).
\end{align*}

To continue, we note that by Proposition~\ref{prop: virial bounded away from zero} there exists $\delta>0$ such that
\begin{equation}\label{eq: V(u)>delta}
V(u(t))\geq\delta \quad \text{for all } t\in\R.
\end{equation}
By \eqref{Zmin2}, for $\eta>0$ to be chosen later there exists $C(\eta)>0$ such that
\begin{align}\label{eq: amost periodic}
\sup_{t\in\R}\int_{|x-x(t)| >C(\eta)} |\nabla u(t,x)|^2+|u(t,x)|^6+|u(t,x)|^4+|u(t,x)|^2 \,dx\leq\eta.
\end{align}
Moreover, by Proposition~\ref{prop: control x(t)} we have that $|x(t)|=o(t)$ as $t\to\infty$ and so there exists $T_0=T_0(\eta)\in\R$ such that
\begin{equation}\label{eq: x(t)}
|x(t)|\leq \eta t \quad \text{for all } t\geq T_0.
\end{equation}

Now given $T_1>T_0$, let
\begin{equation}\label{eq: def of R}
R:=C(\eta)+\sup_{t\in [T_0,T_1]}|x(t)|.
\end{equation}
Note that $\{x:|x|\geq R\}\subset\{x:|x-x(t)|\geq C(\eta)\}$ for $t\in [T_0,T_1]$.  Therefore, using \eqref{eq: V(u)>delta} and \eqref{eq: amost periodic} and choosing $\eta$ sufficiently small depending on $\delta$, we obtain
\begin{align*}
\partial_t A_R(t)\geq \delta \quad\text{uniformly for } t\in [T_0,T_1].
\end{align*}
By the fundamental theorem of calculus combined with \eqref{eq: upper bound V_R}, \eqref{eq: x(t)}, and \eqref{eq: def of R}, we get
\begin{equation}
\delta(T_1-T_0) \leq A_R (T_1) - A_R(T_0) \les_{D_c} R\les_{D_c} C(\eta) + \eta T_1.
\label{Z76}
\end{equation}
Choosing  $\eta=\eta\big(\delta, D_c\big)$ sufficiently small and letting $T_1\to\infty$, we derive a contradiction.

This completes the proof  of Theorem \ref{thm: main}.
\end{proof}


\subsection*{Acknowledgements}
R.~K.~was partially supported by NSF grants DMS-1001531 and DMS-1265868.  T.~O.~acknowledges support from an AMS--Simons Travel Grant.  O. P. was supported by NSF grant under agreement DMS-1128155.  M.~V.~was partially supported by the Sloan Foundation and NSF grants DMS-0901166 and DMS-1161396.  The authors are grateful to the Hausdorff Institute of Mathematics and to the organizers of the program in Harmonic Analysis and Partial Differential Equations; they provided the perfect environment for us to meet and bring this
project to completion.  O. P. would like to thank Prof. M. Mari\c{s} for edifying discussions during her visit to Toulouse.
Any opinions, findings, and conclusions or recommendations expressed in this material are those of the authors and do not necessarily reflect the views of the
National Science Foundation.


\begin{thebibliography}{99}

\bibitem{Akahori} T. Akahori, S. Ibrahim, H. Kikuchi, and H. Nawa,
{\it Existence of a ground state and blow-up problem for a nonlinear Schr\"odinger equation with critical growth.}
Differential Integral Equations {\bf 25} (2012), no. 3--4, 383--402.

\bibitem{Akahori'} T. Akahori, S. Ibrahim, H. Kikuchi, and H. Nawa,
\textit{Existence of a ground state and scattering for a nonlinear Schr\"odinger equation with critical growth.}
Selecta Math. (N.S.) \textbf{19} (2013), no. 2, 545--609.

\bibitem{AkahoriNawa}
T. Akahori and H. Nawa,
\textit{Blowup and scattering problems for the nonlinear Schr\"odinger equations.}
Kyoto J. Math. \textbf{53} (2013), 629--672.

\bibitem{Anderson} D. L. T. Anderson, \textit{Stability of time-dependent particlelike solutions in nonlinear field theories. II.}
J. Math. Phys. \textbf{12} (1971), no. 6, 945--952.

\bibitem{bahouri-gerard} H. Bahouri and P. G\'{e}rard,
\emph{High frequency approximation of solutions to critical nonlinear wave equations.}
Amer. J. Math. \textbf{121} (1999), 131--175.

\bibitem{BegoutVargas} P. Begout and A. Vargas,
\emph{Mass concentration phenomena for the $L^2$-critical nonlinear Schr\"odinger equation.}
Trans. Amer. Math. Soc. \textbf{359} (2007), 5257--5282.

\bibitem{Berestycki}H. Berestycki and P.-L. Lions,
{\it Nonlinear scalar field equations, I, Existence of a ground state.}
Arch. Rational Mech. Anal. {\bf 82} (1983), no. 4, 313--345.

\bibitem{BGS} F.~B\'ethuel, P.~Gravejat, and J.-C.~Saut,
{\it Travelling waves for the Gross--Pitaevskii equation II.}
Commun. Math. Phys. \textbf{285} (2009), 567--651.

\bibitem{BO99}
J.~Bourgain,
{\it  Global well-posedness of defocusing critical nonlinear Schr\"odinger equation in the radial case.}
J. Amer. Math. Soc. {\bf 12} (1999), no. 1, 145--171.

\bibitem{BL}
H. Br{\'e}zis and E. Lieb,
\textit{A relation between pointwise convergence of functions and convergence of functionals.}
Proc. Amer. Math. Soc. \textbf{88} (1983), no. 3, 486--490. 

\bibitem{BrothersZiemer}
J. E.~Brothers and W. P.~Ziemer, \emph{Minimal rearrangements of Sobolev functions.}
J. Reine Angew. Math. \textbf{384} (1988), 153--179.

\bibitem{Grikurov} V. B. Buslaev and V. E. Grikurov,
{\it Simulation of instability of bright solitons for NLS with saturating nonlinearity.}
IMACS Journal Math. and Computers in Simulation {\bf 56} (2001), no. 6, 539--546.

\bibitem{carles-keraani}
R. Carles and S. Keraani, \emph{On the role of quadratic oscillations in nonlinear Schr\"odinger equation II. The $L^2$-critical case.}
Trans. Amer. Math. Soc. \textbf{359} (2007), 33--62.

\bibitem{Cazenave} T.~Cazenave, \textit{Semilinear Schr\"odinger equations.}
Courant Lecture Notes in Mathematics, \textbf{10}. American Mathematical Society, Providence, RI, 2003.

\bibitem{Clausius} R.~Clausius, \emph{On a mechanical theorem applicable to heat.}
Philosophical Magazine, Ser. 4 \textbf{40} (1870), 122--127.

\bibitem{CodLev}
E.~A. Coddington and N. Levinson, \textit{Theory of ordinary differential equations.} McGraw-Hill Book Company, Inc., New York-Toronto-London, 1955.

\bibitem{Coffman}
C. V. Coffman, \textit{Uniqueness of the ground state solution for $\Delta u-u+u^3=0$ and a variational characterization of other solutions.}
Arch. Rational Mech. Anal. \textbf{46} (1972), 81--95.

\bibitem{CKSTT}
J.~Colliander, M.~Keel, G.~Staffilani, H.~Takaoka, and T.~Tao,
{\it Global well-posedness and scattering for the energy-critical nonlinear Schr\"odinger equation in $\R^3$.}
Ann. of Math. (2) \textbf{167} (2008),  no. 3, 767--865.

\bibitem{Desyatnikov} A. Desyatnikov, A. Maimistov, and B. Malomed,
{\it Three-dimensional spinning solitons in dispersive media with the cubic-quintic nonlinearity.}
Phys. Rev. E. {\bf  61} (2000), no. 3, 3107--3113.

\bibitem{Dodson}
B. Dodson, \textit{Global well-posedness and scattering for the defocusing, $L^2$-critical nonlinear Schr\"odinger equation when $d\geq 3$.}
J. Amer. Math. Soc. \textbf{25} (2012), no. 2, 429--463.

\bibitem{DHR08} T.~Duyckaerts, J.~Holmer, and S.~Roudenko, \textit{Scattering for the non-radial 3D cubic nonlinear Schr\"odinger equation.}
Math. Res. Lett. \textbf{15} (2008), no. 6, 1233--1250.

\bibitem{DR10} T. Duyckaerts and S. Roudenko, \textit{Threshold solutions for the focusing 3D cubic Schr\"odinger equation.}
Rev. Mat. Iberoam. \textbf{26} (2010), no. 1, 1--56.

\bibitem{FXC11}
D. Fang, J. Xie, and T. Cazenave,
\textit{Scattering for the focusing energy-subcritical nonlinear Schr\"odinger equation.}
Sci. China Math. \textbf{54} (2011), 2037--2062.

\bibitem{GerardMO} P.~G\'erard, Y.~Meyer, and F.~Oru,
\emph{In\'egalit\'es de Sobolev pr\'ecis\'ees.} S\'eminaire \'E.D.P. (1996--1997), Exp. No. IV, 11pp.

\bibitem{PG}
P. G\'erard, {\it The Cauchy problem for the Gross--Pitaevskii equation.}
Ann. Inst. H. Poincar\'e Anal. Non Lin\'eaire {\bf 23} (2006), no. 5, 765--779.

\bibitem{GNN}
B.~Gidas, W. M. Ni, and L. Nirenberg, \textit{Symmetry of positive solutions of nonlinear elliptic equations in $\R^n$.} Mathematical analysis and applications, Part A, pp. 369--402,
Adv. in Math. Suppl. Stud., \textbf{7a}, Academic Press, New York-London, 1981.

\bibitem{GV}
J.~Ginibre and G.~Velo, \textit{Smoothing properties and retarded estimates for some dispersive evolution equations.}
Comm. Math. Phys. \textbf{144} (1992), no. 1, 163--188.

\bibitem{Ginzburg} V.L. Ginzburg, \textit{Theories of superconductivity (a few remarks).}
Helv. Phys. Acta \textbf{65} (1992), 173--186.

\bibitem{Glassey} R. T. Glassey, \emph{On the blowing up of solutions to the Cauchy problem for nonlinear Schr\"odinger equations.}
J. Math. Phys. \textbf{18} (1977), 1794--1797.


\bibitem{GNT3} S. Gustafson, K. Nakanishi, and T.~P. Tsai,
\textit{Scattering theory for the Gross--Pitaevskii equation in three dimensions.}
Commun. Contemp. Math. \textbf{11} (2009), no. 4, 657--707.

\bibitem{GNT4} S.~Gustafson, K.~Nakanishi, and T.~P.~Tsai,
\textit{Scattering for the Gross--Pitaevskii equation.}
Math. Res. Lett. \textbf{13} (2006), no. 2, 273--285.

\bibitem{IonPaus} A. D. Ionescu and B. Pausader,
\textit{The energy-critical defocusing NLS on $\mathbb{T}^3$.}
Duke Math. J. \textbf{161} (2012), no. 8, 1581--1612.

\bibitem{KeelTao}
M.~Keel and T.~Tao, {\it Endpoint Strichartz estimates.}
Amer. J. Math. \textbf{120} (1998), no. 5, 955--980.

\bibitem{KenigMerle}
C. Kenig and F. Merle,
\textit{Global well-posedness, scattering and blow-up for the energy-critical, focusing, non-linear Schr\"odinger equation in the radial case.}
Invent. Math. \textbf{166} (2006), no. 3, 645--675.

\bibitem{keraani-h1}
S. Keraani, \emph{On the defect of compactness for the Strichartz estimates for the Schr\"odinger equations.}
J. Differential Equations \textbf{175} (2001), no. 2, 353--392.

\bibitem{keraani-l2}
S. Keraani, \emph{On the blow-up phenomenon of the critical nonlinear Schr\"odinger equation.}
J. Funct. Anal. \textbf{235} (2006), 171--192.

\bibitem{KKSV:KdV}
R.~Killip, S. Kwon, S. Shao, and M.~Vi\c{s}an,
\textit{On the mass-critical generalized KdV equation.}
Discrete Contin. Dyn. Syst. \textbf{32} (2012), no. 1, 191--221.

\bibitem{KOPV2011}
R.~Killip, T.~Oh, O.~Pocovnicu, and M.~Vi\c{s}an,
\textit{Global well-posedness of the Gross--Pitaevskii and cubic-quintic nonlinear Schr\"odinger equations with non-vanishing boundary conditions.}
Math. Res. Lett. \textbf{19} (2012), no. 5, 969--986.

\bibitem{KSV:2DKG}
R. Killip, B. Stovall, and M. Visan, \emph{Scattering for the cubic Klein--Gordon equation in two space dimensions.}
Trans. Amer. Math. Soc. \textbf{364} (2012), no. 3, 1571--1631.

\bibitem{Berbec}
R. Killip and M. Vi\c{s}an, \emph{The focusing energy-critical nonlinear Schr\"odinger equation in dimensions five and higher.}
Amer. J. Math. \textbf{132} (2010), no.~2, 361--424.

\bibitem{KV:Gopher} R. Killip and M. Vi\c{s}an, \emph{Global well-posedness and scattering for the defocusing quintic NLS in three dimensions.}
Anal. PDE \textbf{5} (2012), no. 4, 855--885.

\bibitem{ClayNotes} R. Killip and M. Vi\c{s}an, {\it Nonlinear Schr\"odinger Equations at Critical Regularity.}
Evolution equations, 325--437,  Clay Math. Proc., 17, Amer. Math. Soc., Providence, RI, 2013.

\bibitem{KVZ:Ob} Rowan Killip, Monica Vi\c{s}an, and Xiaoyi Zhang,
\textit{Quintic NLS in the exterior of a strictly convex obstacle.}
Preprint \texttt{arXiv:1208.4904}.

\bibitem{Visan:Oberwolfach} H. Koch, D. Tataru, and M. Vi\c{s}an,
\textit{Dispersive equations and nonlinear waves.  Generalized Korteweg--de Vries, nonlinear Schr\"odinger, wave and Schr\"odinger maps.}
Oberwolfach Seminars, \textbf{45}, Birkh\"auser, Basel, 2014.

\bibitem{Kolodner}
I.~I.~Kolodner, \textit{Heavy rotating string---a nonlinear eigenvalue problem.}
Comm. Pure Appl. Math. \textbf{8} (1955), 395--408.

\bibitem{Kwong}
M.~K. Kwong, \textit{Uniqueness of positive solutions of $\Delta u-u+u^p=0$ in $\R^n$.}
Arch. Rational Mech. Anal. \textbf{105} (1989), no. 3, 243--266.

\bibitem{Sulem} B. J. LeMesurier, G. Papanicolaou, C. Sulem, and P.-L. Sulem,
{\it Focusing and multi-focusing solutions of the nonlinear Schr\"odinger equation.} Phys. D {\bf 31} (1988), no. 1, 78--102.

\bibitem{Lions} P.-L. Lions, \textit{The concentration-compactness principle in the calculus of variations. The locally compact case. I.}
Ann. Inst. H. Poincar\'e Anal. Non Lin\'eaire \textbf{1} (1984), no. 2, 109--145.

\bibitem{Maris_Existence of TW} M.~Mari\c{s},
{\it Travelling waves for nonlinear Schr\"odinger equations with nonzero conditions at infinity.}
Ann. of Math. (2) \textbf{178} (2013), no. 1, 107--182.

\bibitem{McLeod}
K.~McLeod, \textit{Uniqueness of positive radial solutions of $\Delta u+f(u)$ in $\R^n$. II.}
Trans. Amer. Math. Soc. \textbf{339} (1993), no. 2, 495--505.

\bibitem{MerleRaphael} F. Merle, P. Rapha\"el, \textit{Blow up of the critical norm for some radial $L^2$ super critical nonlinear Schr\"odinger equations.}
Amer. J. Math. \textbf{130} (2008), no. 4, 945--978.

\bibitem{merle-vega}
F. Merle and L. Vega, \emph{Compactness at blow-up time for $L^2$ solutions of the critical nonlinear Schr\"odinger equation in $2D$.}
Int. Math. Res. Not. \textbf{8} (1998), 399--425.

\bibitem{Miao} C. Miao, G. Xu, and L. Zhao,
{\it The dynamics of the 3D radial NLS with the combined terms.}
Comm. Math. Phys. \textbf{318} (2013), no. 3, 767--808.

\bibitem{Mihalache2000} D. Mihalache, D. Mazilu, L.-C. Crasovan, B. A. Malomed, and F. Lederer,
{\it Three-dimensional spinning solitons in the cubic-quintic nonlinear medium.} Phys. Rev. E {\bf 61} (2000), no. 6, 7142--7145.

\bibitem{Mihalache2002} D. Mihalache, D. Mazilu, L.-C. Crasovan, I. Towers, A. V. Buryak, B. A. Malomed, and L. Torner,
{\it Stable spinning solitons in three dimensions.} Phys. Rev. Lett. {\bf 88} (2002), no. 7, 4pp..

\bibitem{PayneSattinger} L. E. Payne and D. H. Sattinger,
\textit{Saddle points and instability of nonlinear hyperbolic equations.}
Israel J. Math. \textbf{22} (1975), no. 3--4, 273--303.

\bibitem{RS4}
M. Reed and B. Simon,
\textit{Methods of modern mathematical physics. IV. Analysis of operators.} Academic Press, New York--London, 1978.

\bibitem{SerrinTang2000} J. Serrin and M. Tang,
{\it Uniqueness of ground states for quasilinear elliptic equations.} Indiana Univ. Math. J. {\bf 49} (2000), no. 3, 897--923.


\bibitem{Shatah} J. Shatah and W. Strauss,
{\it Instability of nonlinear bound states.} Comm. Math. Phys. {\bf 100} (1985), no. 2, 173--190.

\bibitem{Strichartz} R.~S.~Strichartz,
\emph{Restrictions of Fourier transforms to quadratic surfaces and decay of solutions of wave equations.} Duke Math. J. \textbf{44} (1977), no. 3, 705--714.

\bibitem{SSulem} C. Sulem and  P.-L. Sulem,
\emph{The nonlinear Schr\"odinger equation. Self-focusing and wave collapse.}
Applied Mathematical Sciences, \textbf{139}. Springer-Verlag, New York, 1999.

\bibitem{TaoVisanZhang} T. Tao, M. Vi\c{s}an, and X. Zhang, {\it The nonlinear Schr\"odinger equation with combined power-type nonlinearities.}
Comm. Partial Differential Equations {\bf 32} (2007), no. 7--9, 1281--1343.

\bibitem{VlasovPT} S.~N.~Vlasov, V.~A.~Petrishchev, and V.~I.~Talanov,
\emph{Averaged description of wave beams in linear and nonlinear media (the method of moments).}
Radiophys. Quantum Electron. \textbf{14} (1971), 1062--1070.


\bibitem{Weinstein} M. Weinstein, {\it Nonlinear Schr\"odinger equations and sharp interpolation estimates.}
Comm. Math. Phys. {\bf 87} (1982/83), no. 4, 567--576.

\bibitem{Weinstein'} M. Weinstein, {\it Modulational stability of ground states of nonlinear Schr\"odinger equations.}
SIAM J. Math. Anal. \textbf{16} (1985), no. 3, 472--491.

\bibitem{Yajima}  K. Yajima, \textit{Existence of solutions for Schr\"odinger evolution equations.}
Comm. Math. Phys. \textbf{110} (1987), no. 3, 415--426.

\bibitem{Zhang} X. Zhang, \textit{On the Cauchy problem of 3-D energy-critical Schr\"odinger equations with subcritical perturbations.}
J. Differential Equations \textbf{230} (2006), no. 2, 422--445.

\end{thebibliography}
\end{document}